\documentclass[11pt,letterpaper]{amsart}
\usepackage{amssymb,amscd,stmaryrd}
\usepackage{amsmath}
\usepackage{amsfonts}
\usepackage[top=3cm,bottom=3cm,left=2.5cm,right=2.5cm,footskip=17mm]{geometry}
\usepackage[mathcal]{eucal}
\usepackage{array,float}
\usepackage{enumitem}
\usepackage[OT2,T1]{fontenc}
\usepackage[russian,USenglish]{babel}
\usepackage{chngcntr}
\usepackage{apptools}
\AtAppendix{\counterwithin{theorem}{section}}
\usepackage{float}
\newfloat{Diagram}{htbp}{dia}
\usepackage{tabularx}
\usepackage{longtable}

\usepackage{amsmath}
\usepackage{amssymb, xy}
\input xy
\xyoption{all}
\usepackage{pdflscape}
\usepackage{hyperref}
\usepackage[draft]{graphicx}
\usepackage{tikz,pgf}
\usepackage{tikzit}

\tikzstyle{1function}=[fill=white, draw=black, shape=rectangle, minimum width=0.75cm, minimum height=1 cm]
\tikzstyle{2function}=[fill=white, draw=black, shape=rectangle, minimum width=1cm, minimum height=1 cm]
\tikzstyle{3function}=[fill=white, draw=black, shape=rectangle, minimum width=1.5cm, minimum height=1cm]
\tikzstyle{multi function}=[fill=white, draw=black, shape=rectangle, minimum width=5cm, minimum height=1cm]
\tikzstyle{multi function small}=[fill=white, draw=black, shape=rectangle, minimum width=4cm, minimum height=1cm]
\tikzstyle{Multi function smaller}=[fill=white, draw=black, shape=rectangle, minimum width=3.5 cm, minimum height=1 cm]

\tikzstyle{black dashed}=[-, draw=black, dashed]

\usepackage{xcolor}
\usepackage{color}

\numberwithin{equation}{section}

\newtheorem{theorem}{Theorem}[section]
\newtheorem{lemma}[theorem]{Lemma}
\newtheorem{proposition}[theorem]{Proposition}
\newtheorem{corollary}[theorem]{Corollary}
\newtheorem{claim}[theorem]{Claim}
\theoremstyle{definition}
\newtheorem{definition}[theorem]{Definition}

\newtheorem{remark}[theorem]{Remark}

\newcommand{\ot}{\otimes}
\renewcommand{\C}{\mathcal{C}}
\newcommand{\N}{\mathbb{N}}

\newcommand{\one}{\textbf{1}}

\newcommand{\Z}{\mathbb{Z}}

\newcommand{\Ker}{\text{Ker}}

\newcommand{\End}{\text{End}}
\newcommand{\Hom}{\text{Hom}}
\newcommand{\M}{\text{M}}

\newcommand{\la}{\lambda}
\newcommand{\ol}{\overline}

\newcommand{\GL}{\text{GL}}

\newcommand{\Delot}{\Delta^{\otimes}}
\newcommand{\Tr}{\text{Tr}}
\newcommand{\Id}{\text{Id}}
\newcommand{\T}{\mathcal{T}}

\newcommand{\wt}{\widetilde}
\newcommand{\winf}{\mathfrak{U}}
\newcommand{\Ow}{\mathcal{O}}
\renewcommand{\Re}{\text{Re}}
\renewcommand{\Im}{\text{Im}}
\newcommand{\pair}{\text{pair}}
\newcommand{\Aut}{\text{Aut}}

\newcommand{\I}{\mathfrak{I}}

\renewcommand{\Vec}{\text{Vec}}

\newcommand{\Rep}{\text{Rep}}
\newcommand{\cchi}{\C_{\chi}}
\newcommand{\rad}{\text{rad}}
\newcommand{\D}{\mathcal{D}}
\newcommand{\Nn}{\mathcal{N}}
\newcommand{\Alt}{\text{Alt}}
\newcommand{\ul}{\underline}
\newcommand{\dcob}{\text{DCob}}
\newcommand{\Stab}{\text{Stab}}

\title[Interpolations and invariants]{Interpolations of monoidal categories and algebraic structures by invariant theory}
\author{Ehud Meir}
\address{Institute of Mathematics, University of Aberdeen, Fraser Noble Building, Aberdeen AB24 3UE, UK}
  \email{ehud.meir@abdn.ac.uk, meirehud@gmail.com}

\begin{document}
\maketitle
\begin{abstract}
In this paper we give a general construction of symmetric monoidal categories that generalizes Deligne's interpolated categories, the categories introduced by Knop, and the recent TQFT construction of Khovanov, Ostrik, and Kononov. 
The categories we will consider are generated by an algebraic structure. In a previous work by the author a universal ring of invariants $\mathfrak{U}$ for algebraic structures of a specific type was constructed. 
It was shown that any algebraic structure of this type in $\Vec_K$ gives rise to a character $\chi:\mathfrak{U}\to K$. In this paper we consider algebraic structure in general symmetric monoidal categories, not only in $\Vec_K$, and general characters on $\mathfrak{U}$. From any character $\chi:\mathfrak{U}\to K$ we construct a symmetric monoidal category $\C_{\chi}$, analogous to the universal construction from TQFT.
We then give necessary and sufficient conditions for a given character to arise from a structure in an abelian category with finite dimensional hom-spaces. We call such characters good characters. We show that if $\chi$ is good then $\C_{\chi}$ is abelian and semisimple, and that the set of good characters forms a $K$-algebra. We also show that the categories $\C_{\chi}$ contain all categories of the form $\Rep(G)$, where $G$ is reductive. The construction of $\C_{\chi}$ gives a way to interpolate algebraic structures, and also symmetric monoidal categories, in a way that generalizes Deligne's categories $\Rep(S_t)$, $\Rep(\GL_t(K))$, and $\Rep(\text{O}_t)$. We also explain how one can recover the recent construction of 2 dimensional TQFT of Khovanov, Ostrik, and Kononov, by the methods presented here. 
We give new examples, of interpolations of the categories $\Rep(\Aut_{\Ow}(M))$ where $\Ow$ is a discrete valuation ring with a finite residue field, and $M$ is a finite module over it. We also generalize the construction of wreath products with $S_t$, which was introduced by Knop. 
\end{abstract}

\section{Introduction}
In \cite{Deligne3} Deligne interpolated the categories $\Rep(S_n)$ and introduced the celebrated family of categories $\Rep(S_t)$ where $t\in K$ is any element in a characteristic zero field $K$. He also presented several other families of symmetric monoidal categories, such as $\Rep(\GL_t),\Rep(\text{O}_t)$ and $\Rep(\text{Sp}_t)$. These categories were studied and generalized by many different authors. See for example \cite{CO}, \cite{EAH}, \cite{Knop}; \cite{HK} for a study of algebras inside $\Rep(S_t)$; and \cite{KKO} and \cite{KS} for the relation to topological quantum field theories. 

In this paper we present a generalization of Deligne's construction, which generalizes the families presented by Deligne, Knop, and Khovanov, Ostrik, and Kononov, in \cite{Deligne3}, \cite{Knop}, and \cite{KKO} respectively. All the categories in this paper can be seen as generalizations of categories of the form $\Rep(\Aut(W))$, where $W$ is some algebraic structure. 
We construct interpolations of symmetric monoidal categories by interpolating the algebraic structure $W$, which we shall do by interpolating the scalar invariants of the structure, following \cite{meirUIR}. In this setting, the categories $\Rep(S_n)$ can be understood as $\Rep(K^n)$, where $K^n$ has the canonical commutative separable algebra structure (see Section \ref{sec:examples})

We recall here the definitions. Fix $((p_i,q_i))\in (\N^2)^r$. An \emph{algebraic structure} of type $((p_i,q_i))$ on a finite dimensional vector space $W$ is given by specifying structure tensors $x_i\in W^{p_i,q_i}:=W^{\ot p_i}\ot (W^*)^{\ot q_i}$ for $i=1,\ldots, r$. These can specify, for example, multiplication, comultiplication, counit, endomorphisms of $W$ et cetera. For example, a unital algebra has type $((1,2),(1,0))$, where the first tensor specifies the multiplication and the second tensor specifies the unit. We assume throughout the paper that the type of the algebraic structure is fixed. In Section \ref{sec:examples} we give examples of algebraic structures of various types. 

Fixing the vector space $W=K^d$, the set of structure tensors $(x_i)$ can be considered as a point in the affine space $U_d = \bigoplus_i W^{p_i,q_i}$. However, this point is not uniquely defined by the isomorphism type of the structure $(W,(x_i))$. The group $\GL_d(K)$ acts on $U_d$, and two points in $U_d$ define isomorphic structures if and only if they are in the same $\GL_d(K)$-orbit. 
The ring of invariants $K[U_d]^{\GL_d(K)}$ arises naturally in this context, as the characters $K[U_d]^{\GL_d(K)}\to K$ are in one to one correspondence with closed $\GL_d(K)$-orbits in $U_d$. In \cite{meirUIR} a universal ring of invariants $\winf$ was introduced. This ring captures simultaneously algebraic structures for different values of $d$ in the sense that for every $d$ there is an ideal $I_d\subseteq \winf$ such that $\winf/I_d\cong K[U_d]^{\GL_d(K)}$. The ring $\winf$ was also shown to be a Hopf algebra with some additional structure (see Section 8 of \cite{meirUIR}). It is a polynomial algebra generated by the set $P$ of closed connected diagrams arising from the structure tensors. 
 
An algebraic structure $(W,(x_i))$ of dimension $d$ induces a character $K[U_d]\to K$ and by restriction also a character $K[U_d]^{\GL_d(K)}\to K$. By pulling back we get a character $\chi_{(W,(x_i))}:\winf\to K$. We call this character the \emph{character of invariants} of $(W,(x_i))$. If the $\GL_d(K)$-orbit of $(W,(x_i))$ in $U_d$ is closed then the isomorphism class of $(W,(x_i))$ can be reconstructed from its character of invariants (see Section \ref{sec:vecstr}). 

This raises the question: what about all the other characters of $\winf$, which do not split via one of the quotients $\winf\to\winf/I_d$? We answer this question by studying algebraic structures in general symmetric monoidal categories, and not only in $\Vec_K$. 

If $\D$ is a rigid $K$-linear symmetric monoidal category we can consider algebraic structures of type $((p_i,q_i))$ inside $\D$. Similarly to the case of structures in $\Vec_K$, such a structure $(A,(y_i))$ will induce a character $\winf\to \End_{\D}(\one)$. In particular, if $\End_{\D}(\one)=K$ we get a character $\chi:\winf\to K$. We will say that $(A,(y_i))$ \emph{affords} the character $\chi$. 

For any character $\chi:\winf\to K$ we will construct in Section \ref{sec:constcchi} a symmetric monoidal category $\C_{\chi}$ that is tensor-generated by a structure that affords the character $\chi$, and its dual. This category will be $K$-linear, rigid, additive, Karoubi closed, but not necessarily abelian. It holds that $\End_{\C_{\chi}}(\one)=K$. 

To construct the category $\C_{\chi}$ we will first construct in Section \ref{sec:unicat} a category $\C_{univ}$, which is the universal category freely generated by a structure $(W,(x_i))$ of type $((p_i,q_i))$. In Section 5 of \cite{meirUIR} we constructed for every $p,q\in \N$ a vector space $Con^{p,q}$ of formal linear combinations of formal compositions of the structure tensors. We will use these spaces to construct the hom-spaces in $\C_{univ}$. 

The endomorphism ring $\End_{\C_{univ}}(\one)$ is $\winf$. 
The trace pairing in $\C_{univ}$ gives a pairing $\pair^{p,q}:Con^{p,q}\ot Con^{q,p}\to \winf$ (see Definition 7.5. in \cite{meirUIR}). We can compose this pairing with the character $\chi$ to get a pairing $\pair^{p,q}_{\chi}:Con^{p,q}\ot Con^{q,p}\to \winf\stackrel{\chi}{\to} K$. We then form the category $\C_{\chi}$ by dividing out by the negligible morphisms with respect to this pairing, and taking the Karoubian envelope. We call a morphism $T$ in $\C_{univ}$ $\chi$-negligible if it is negligible under the pairing induced by $\chi$. 

We thus see that every character $\chi:\winf\to K$ is the character of invariants of some algebraic structure in some $K$-linear rigid symmetric monoidal category. However, the symmetric monoidal categories one encounters when studying algebraic structures are often not only $K$-linear and rigid but also abelian with finite dimensional vector spaces as hom-spaces. We will call $\D$ a $K$-good category if it is a symmetric monoidal $K$-linear abelian rigid category in which the hom-spaces are finite dimensional and $\End_{\D}(\one)=K$. 

The first main theorem of this paper is the following:
\begin{theorem}\label{thm:main1}
Let $\chi:\winf\to K$ be a character. The following conditions are equivalent:
\begin{enumerate}
\item The category $\C_{\chi}$ is a semisimple $K$-good category.
\item The character $\chi$ is afforded by an algebraic structure in some $K$-good category.
\item The following two conditions are satisfied: 
\begin{itemize} 
\item The radical $\rad^{p,q}_{\chi}\subseteq Con^{p,q}$ of $\pair^{p,q}_{\chi}$ has finite codimension.
\item If $A\in \C_{univ}$ and $T:A\to A$ satisfies that $T^r$ is $\chi$-negligible for some $r>0$, then $\chi(\Tr(T))=0$. 
\end{itemize}
\end{enumerate}
\end{theorem}
If $\chi$ satisfies the equivalent conditions of the theorem we will say that $\chi$ is a \emph{good character}.
If $\chi$ is the character of invariants of a structure $(Y,(y_i))$ in $\Vec_K$, we will show in Section \ref{sec:vecstr} that there is a unique symmetric fiber functor $F:\C_{\chi}\to \Vec_K$. Using Tannaka reconstruction we will show that we get an equivalence $\C_{\chi}\cong\Rep_K(\Aut(Z,(z_i)))$, where $(Z,(z_i))$ is the unique structure with closed orbit in the closure of the orbit of $(Y,(y_i))$ in $U_{\dim(Y)}$. 
We summarize this in the following theorem:
\begin{theorem}\label{thm:main2}  
Let $(Y,(y_i))$ be an algebraic structure of dimension $d$ in $\Vec_k$. Assume that the $\GL_d$-orbit of this structure is closed in $U_d$. Then there is an equivalence $\C_{\chi}\cong \Rep(\Aut(Y,y_i))$, and for every $p,q\in \N$ the linear span of linear maps $Y^{\ot q}\to Y^{\ot p}$ which are constructible from the structure tensors is equal to  $\Hom_{\Aut(Y,(y_i))}(Y^{\ot q},Y^{\ot p})$. 
\end{theorem}

This theorem raises the question of what representation categories are equivalent to $\C_{\chi}$ for some $\chi$. If $(Y,(y_i))$ has a closed orbit, this orbit will be isomorphic to $\GL_d/\Aut(Y,(y_i))$. In particular, the quotient $\GL_d/\Aut(Y,(y_i))$ is affine, and by Matsushima's criterion (see \cite{Arzhantsev} and references therein) the group $\Aut(Y,(y_i))$ is reductive. Another way to see that the group $\Aut(Y,(y_i))$ is reductive in the above case is that it is equivalent to $\C_{\chi}$, which is semisimple. 
We will prove the following in Section \ref{sec:vecstr}:
\begin{theorem}\label{thm:mainallG}
Let $G$ be a reductive affine algebraic group. Then $\Rep(G)\cong \C_{\chi}$ for some character $\chi$.
\end{theorem}
We will show that even though algebraic structures in $\Vec_K$ might be scarce, good characters are abundant. In fact, as we will see in the theorem below, they form a $K$-algebra that can be calculated explicitly in many cases. The diversity of good characters will also enable us to interpolate the categories of representations of automorphism groups of algebraic structure to more general categories, which generalize the categories $\Rep(S_t)$ of Deligne.

It was shown in Section 5 of \cite{meirUIR} that the ring $\winf$ is a polynomial algebra on the set $P$ of closed connected diagrams constructed from the structure tensors. There is therefore a natural bijection between characters on $\winf$ and functions $P\to K$. We will consider this bijection as an identification in what follows. The set of characters $K^P$ is a $K$-algebra under pointwise addition and multiplication. 
It follows from Section 6 of \cite{meirUIR} that taking direct sums and tensor products corresponds to taking sums and products of characters in $K^P$ respectively. We will prove the following result in Section \ref{sec:goodalgebra}: 
\begin{theorem}\label{thm:main3}
The set of good characters in $K^P$ forms a $K$-subalgebra. 
\end{theorem}

In Section \ref{sec:interpolations} we will show that if $(\chi_t)$ is a one-parameter family of characters such that $\chi_t$ is good for countably many values of $t$, then under some mild conditions it holds that $\chi_t$ is good for every value of $t$.  

In Section \ref{sec:examples} we will give examples.
We will show how we can recover the constructions of Deligne for $\Rep(S_t)$, $\Rep(\GL_t)$, $\Rep(\text{O}_t)$ and $\Rep(\text{Sp}_t)$ from \cite{Deligne3}. Knop has generalized the construction of Deligne in \cite{Knop} by using degree functions. The values of his degree functions correspond to some values of the character $\chi$ we have here. Knop constructed categories such as $\Rep(S_t\wr G)$ where $G$ is a finite group, and $\Rep(\GL_t(\Ow/(\pi^r)))$ where $\Ow$ is a discrete valuation ring with a uniformizer $\pi$ and a finite residue field. We will give examples that generalize his constructions. If $\chi$ is a good character, we will show how to construct a family of categories $\C_{t\cdot \ol{\chi}}$ such that if $\C_{\chi}\cong \Rep(G)$ then $\C_{n\cdot\ol{\chi}}\cong \Rep(S_n\wr G)$ for $n\in\N$. Here $G$ can be a reductive group, and not necessarily a finite one. We will generalize the categories $\Rep(\GL_t(\Ow/(\pi^r)))$ in the following way: every finite $\Ow/(\pi^r)$-module has the form $M_{a_1,\ldots, a_r}:= (\Ow/(\pi))^{a_1}\oplus (\Ow/(\pi^2))^{a_2}\oplus\cdots\oplus (\Ow/(\pi^r))^{a_r}$. We will construct a family of good characters $\chi(t_1,\ldots, t_r)$, which depend on $r$ non-zero parametrs $t_i$, such that $\C_{\chi(q^{a_1},\ldots, q^{a_r})}\cong \Rep(\Aut_{\Ow}(M_{a_1,\ldots, a_r}))$, where $q=|\Ow/(\pi)|$. By considering the family of categories $\C_{\chi(t_1,\ldots, t_r)}$ we get an interpolation of the categories $\Rep(\Aut_{\Ow}(M_{a_1,\ldots, a_r}))$. 

The construction of Deligne was also generalized recently by Khovanov, Ostrik, and Kononov in \cite{KKO}. They constructed a family of $K$-good categories which are generated by a commutative Frobenius algebra, and used these to construct new examples of 2 dimensional TQFTs. Their construction depends on an infinite sequence $(\alpha_0,\alpha_1,\ldots)$ of scalars. In Section \ref{sec:examples} we will show how we can recover this construction, and that the numbers $\alpha_i$ arise as some character values. Khovanov, Ostrik, and Kononov also gave a criterion that characterizes the sequences $(\alpha_i)$ which arise from $K$-good categories. We will show how their criterion relates to the criterion for good characters we have in Theorem \ref{thm:main1}. 
The examples we have are summarized in the following table:

\begin{center}
\newcolumntype{C}[1]{>{\centering\arraybackslash}m{#1}}

\begin{longtable}{|C{4cm}|C{3.5cm}|C{3.5cm}|C{3.5cm}|}
\hline
Type of algebraic structure  & Resulting category & good characters  & See also \\
\hline 
Empty structure & $\Rep(\GL_t)$ & $K$ & \cite{Deligne3} \\ \hline 
Single endomorphism & $\prod_i \Rep(\GL_{t_i})$ & Monoid algebra of $(K,\times)$ & \\  \hline 
Non-degenerate symmetric pairing & $\Rep(\text{O}_t)$ & $K$ & \cite{Deligne3} \\ \hline 
Non-degenerate skew-symmetric pairing & $\Rep(\text{Sp}_t)$ & $K$ & \cite{Deligne3} \\ \hline 
Separable commutative algebra & $\Rep(S_t)$ & $K$ & \cite{Deligne3} \\ \hline 
Commutative Frobenius algebra & $\dcob_{\alpha}$ & Described in Subsection \ref{subsec:dcob} & \cite{KKO} \\ \hline 
Wreath products with $S_t$, for general structures & $\Rep(S_t\wr G)$ & $K$ & Described in some cases in \cite{Knop}. \\ \hline 
Group algebra with operators of a finite module $M_{a_1,\ldots, a_r}$ over a DVR $\Ow$ & $\Rep(\Aut(M_{a_1,\ldots, a_r}))$ and their interpolations, & $(K^{\times})^r$& The case $a_1=a_2=\cdots =a_{r-1}=0$ appears in \cite{Knop}. \\
\hline
\end{longtable}
\end{center}
The construction of the category $\C_{\chi}$ is strongly related to the universal construction of TQFTs from \cite{BHMV}. We replace the cobordism by the constructible morphisms and invariants of closed manifold are replaced by values of the character $\chi$. I believe that the methods presented in this paper can be further used to the study and construction of other TQFTs as well. 

\section{Preliminaries and notations}
\subsection{Algebraic structures in symmetric monoidal categories}\label{subsec:algstr}
Throughout this paper $K$ will be an algebraically closed field of characteristic zero, and all categories will be $K$-linear rigid symmetric monoidal categories. Recall that a symmetric monoidal category $\D$ is called rigid if for every $X\in\D$ there is an object $X^*$ together with maps $ev_X:X^*\ot X\to \one$ and $coev_X:\one\to X\ot X^*$, such that $ev_X$ and $coev_X$ induce an adjunction isomorphism 
$\Hom_{\D}(X^*\ot Y ,Z)\cong \Hom_{\D}(Y,X\ot Z)$ for every $Y,Z\in \D$. The object $\one$ here is the tensor unit of $\D$. Since we will only consider symmetric monoidal categories in this paper, we will not distinguish between left and right duals. In particular, we also get natural isomorphisms $\Hom_{\D}(X\ot Y ,Z)\cong \Hom_{\D}(Y,X^*\ot Z)$ induced by $ev_X$ and $coev_X$, together with the symmetry of the category. 

If $\D$ is a symmetric monoidal category and $X\in \D$ is some object, then we get an action of the symmetric group $S_n$ on $X^{\ot n}$. To be more precise, we have a group homomorphism $\phi:S_n\to \Aut_{\D}(X^{\ot n})$ given by permuting the tensor factors. We will write $\phi(\sigma) = L^{(n)}_{\sigma}(X)$, or just $L^{(n)}_{\sigma}$ if the object $X$ is clear. 

If $f:X\to X$ is an endomorphism in $\D$, we define the \emph{trace of $f$}, $\Tr(f)$, to be $$\Tr(f):\one\stackrel{coev_X}{\to}X\ot X^*\stackrel{f\ot 1}{\to } X\ot X^*\stackrel{c_{X,X^*}}{\to} X^*\ot X\stackrel{ev_X}{\to}\one.$$ Thus, $\Tr(f)\in \End_{\D}(\one)$. Of particular importance is the trace of the identity morphism, which we shall denote by $\dim(X)=\Tr(\Id_X)$. If $\End_{\D}(\one)=K$, the trace is just a scalar.
If $F:\D\to \D'$ is a symmetric monoidal functor, then the definition of trace immediately implies that $\Tr(F(f)) = F(\Tr(f))$, where we consider here the ring homomorphism $\End_{\D}(\one)\to \End_{\D'}(\one)$ induced by $F$.  

Given a symmetric monoidal category $\D$ and a tuple $((p_i,q_i))\in (\N^2)^r$, an algebraic structure of type $((p_i,q_i))$ in $\D$ is a pair $(A,(y_i))$ where $A$ is an object of $\D$ and for every $i=1,\ldots, r$ the \emph{structure tensor} $y_i$ is an element of $\Hom_{\D}(A^{\ot q_i},A^{\ot p_i})\cong \Hom_{\D}(\one, A^{p_i,q_i})$, where $A^{p,q}:= A^{\ot p}\ot (A^*)^{\ot q}$. For general algebraic structure we do not require the structure tensors $y_i$ to satisfy any particular set of axioms, but in most practical cases they do. We fix the type $((p_i,q_i))$. For $p,q\geq 1$ we will write $ev:A^{p,q}\to A^{p-1,q-1}$ for the map which applies evaluation on the last tensor copy of $A$ with the last tensor copy of $A^*$. 

If $(A,(y_i))$ is a structure of type $((p_i,q_i))$ in $\D$, we can use evaluation, tensor products and composition with $L^{(n)}_{\sigma}$ to form morphisms in $\Hom_{\D}(A^{\ot q},A^{\ot p})$ for different values of $p$ and $q$. All such morphisms can be described pictorially using strings diagrams (see Section 4 of \cite{meirUIR}).
In Section 5 of \cite{meirUIR} we constructed a graded associative algebra $Con = \bigoplus_{p,q\in\N} Con^{p,q}$ (here $\N = \{0,1,2,\ldots\}$). Note that this algebra depends on the type $((p_i,q_i))$ of the structure. The vector space $Con^{p,q}$ is freely spanned by equivalence classes of diagrams representing morphisms from $A^{\ot q}$ to $A^{\ot p}$. For example, if $(p_1,q_1) = (1,2)$ and $(p_2,q_2)=(2,1)$, then the following diagram in $Con^{2,2}$ represents the morphism $ev(L_{(123)}^{(3)}(y_1\ot y_2))$:
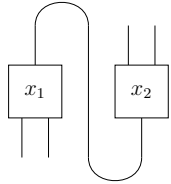
\begin{center}\scalebox{0.7}{
\begin{tikzpicture}
	\begin{pgfonlayer}{nodelayer}
		\node [style=2function] (0) at (-4, 3) {$x_1$};
		\node [style=2function] (1) at (-2, 3) {$x_2$};
		\node [style=none] (2) at (-3.75, 2.75) {};
		\node [style=none] (3) at (-4.25, 2.75) {};
		\node [style=none] (4) at (-4, 3.25) {};
		\node [style=none] (5) at (-2, 2.75) {};
		\node [style=none] (6) at (-1.75, 3.25) {};
		\node [style=none] (7) at (-2.25, 3.25) {};
		\node [style=none] (8) at (-4, 4.25) {};
		\node [style=none] (9) at (-1.75, 4.25) {};
		\node [style=none] (10) at (-2.25, 4.25) {};
		\node [style=none] (11) at (-3.75, 1.75) {};
		\node [style=none] (12) at (-4.25, 1.75) {};
		\node [style=none] (13) at (-2, 1.75) {};
		\node [style=none] (14) at (-3, 4.25) {};
		\node [style=none] (15) at (-3, 1.75) {};
	\end{pgfonlayer}
	\begin{pgfonlayer}{edgelayer}
		\draw (5.center) to (13.center);
		\draw (9.center) to (6.center);
		\draw (10.center) to (7.center);
		\draw (8.center) to (4.center);
		\draw (2.center) to (11.center);
		\draw (3.center) to (12.center);
		\draw [bend left=270, looseness=1.50] (14.center) to (8.center);
		\draw (14.center) to (15.center);
		\draw [bend right=90, looseness=1.50] (15.center) to (13.center);
	\end{pgfonlayer}
\end{tikzpicture}}
\end{center}
The permutation $(123)$ is needed here to make sure that the output string of $x_1$ enters the is connected with the input string of $x_2$. 
The multiplication on $Con$ is given by taking tensor products of maps. The \emph{constructible morphisms} in $Con$ make sense in any rigid symmetric monoidal category. Thus, for every $p,q\in \N$ we have a realization map 
$\Re_A^{p,q}:Con^{p,q}\to \Hom_{\D}(A^{\ot q},A^{\ot p})$, given by sending every basis element of $Con^{p,q}$ to the morphism it represents (see also Definition 7.2. in \cite{meirUIR}). 
In particular, we get an algebra homomorphism $Con^{0,0} \to \Hom_{\D}(\one,\one)$. We write $\winf = Con^{0,0}$. This algebra, which is in fact a commutative Hopf algebra, was the main object of study of the paper \cite{meirUIR}. This algebra was named $K[X]_{aug}$ in \cite{meirUIR}. We rename it here, for the sake of simplicity. 
If $\D'$ is another $K$-linear rigid symmetric monoidal category, and $F:\D\to \D'$ is a symmetric monoidal functor, then $(F(A),(F(y_i))$ is an algebraic structure of type $((p_i,q_i))$ in $\D'$. 
For every $p,q\in\N$ we have the following commutative diagram:
$$\xymatrix{Con^{p,q}\ar[rd]_{\Re_{F(A)}^{p,q}}\ar[r]^{\Re_A^{p,q}}&\Hom_{\D}(A^{\ot q},A^{\ot p})\ar[d]^{F} \\ 
& \Hom_{\D'}(F(A)^{\ot q},F(A)^{\ot p})}.$$
If the structure $(A,(y_i))$ is clear from the context we will also write $\Re^{p,q}$ for $\Re^{p,q}_A$. Since elements in $Con^{p,q}$ represent morphisms $A^{\ot q}\to A^{\ot p}$, we have a pairing $\pair^{p,q}:Con^{p,q}\ot Con^{q,p}\to \winf$ given by $\pair^{p,q}(T_1\ot T_2) = \Tr(T_1\circ T_2)$. The following diagram is then commutative (see also Diagram 7.1. in \cite{meirUIR}):
$$\xymatrix{Con^{p,q}\ot Con^{q,p}\ar[rr]^(0.35){\Re^{p,q}\ot \Re^{q,p}}\ar[d]^{\pair^{p,q}} & & \Hom_{\D}(A^{\ot q},A^{\ot p})\ot\Hom_{\D}(A^{\ot p},A^{\ot q})\ar[d]\\ \winf\ar[rr]^{\Re^{0,0}} & & \End_{\D}(\one),}$$ where the map $\Hom_{\D}(A^{\ot q},A^{\ot p})\ot\Hom_{\D}(A^{\ot p},A^{\ot q})\to \End_{\D}(\one)$ is given by $f\ot g\mapsto \Tr(f\circ g)$. 

If $\Hom_{\D}(\one,\one)=K$ then the structure $(A,(y_i))$ gives a character $\winf\to K$. We will refer to this character as the \emph{character of invariants} of $(A,(y_i))$, and we will often write it as $\chi_{(A,(y_i))}$. We will also say that $(A,(y_i))$ \emph{affords} the character $\chi_{(A,(y_i))}$. 

\begin{definition} A $K$-good category is a $K$-linear symmetric monoidal category $\D$ that satisfies the following conditions:
\begin{enumerate}
\item The hom-spaces in $\D$ are finite dimensional.
\item It holds that $\End_{\D}(\one)\cong K$. 
\item The category $\D$ is abelian and rigid.
\end{enumerate}
\end{definition}

The following lemma is Proposition 4.7.5. in \cite{EGNO} and Lemme 3.5 in \cite{Deligne3}. We will use it in this paper to calculate traces. 
\begin{lemma}
Let $\D$ be a $K$-good category. 
Assume that we have a commutative diagram 
$$\xymatrix{0\ar[r] & A\ar[r]^i\ar[d]^f & B\ar[r]^p\ar[d]^g & C\ar[r]\ar[d]^h & 0 \\ 
0 \ar[r] & A\ar[r]^i& B\ar[r]^p & C\ar[r] & 0}
$$
where the two rows are the same short exact sequence. Then 
$\Tr(g) = \Tr(f)+\Tr(h)$.
\end{lemma}
The above lemma has the following corollary:
\begin{corollary}[See also Corollaire 3.6 in \cite{Deligne3}]\label{cor:endshape}
Let $\D$ be a $K$-good category, and let $T\in\End_{\D}(B)$ for some object $B\in \D$. Then $\sum_{i\geq 0}\Tr(T^i)X^i\in K[[X]]$ is a rational function of the form $\frac{P(X)}{Q(X)}$ where $\deg(P)\leq \deg(Q)$, $Q$ has no multiple roots, and $Q(0)\neq 0$. In particular, if $T$ is nilpotent, then $\Tr(T)=0$. 
\end{corollary}
\begin{proof}
Notice first that the set of rational functions that satisfy the condition stated in the corollary is the same as the linear span of functions of the form $\frac{1}{1-\la X}$ for some $\la\in K$. In particular, it is a linear subspace of $K[[X]]$. 

Since the hom-spaces in $\D$ are finite dimensional, the set $\{\Id_B,T,T^2,\ldots\}$ is linearly dependent, and $T$ solves some non-zero polynomial. Let $f(t)$ be the minimal polynomial of $T$. We shall proceed by induction on $\deg(f)$. If $\deg(f)=1$ then $T=\la\Id_B$ for some $\la\in K$. This implies that $\sum_i \Tr(T^i)X^i = \dim(B)\sum_i\la^iX^i = \frac{\dim(B)}{1-\la X}$, and we get a rational function of the desired form. 
If $\deg(f)>1$, then since $K$ is algebraically closed the polynomial $f$ splits into linear terms. In particular, $T-\la\Id_B$ is not invertible for some $\la\in K$. Write $A =  \Ker(T-\la\Id_B)$ and $C=\Im(T-\la\Id_B)$. We thus have a short exact sequence $0\to A\to B\to C\to 0$. Moreover, $T$ induces a morphism from this short exact sequence to itself. The lemma now implies that $\Tr(T^i) = \Tr(T^i|_A) + \Tr(T^i_C) = \dim(A)\la^i + \Tr(T^i_C)$ where $T_C:C\to C$ is the endomorphism induced by $T$. Since the minimal polynomial of $T_C$ is $f(t)/(t-\la)$ we get that 
$$\sum_i\Tr(T^i)X^i = \frac{\dim(A)}{1-\la X}+ \sum_i \Tr(T_C^i)X^i,$$ and by induction we are done.
In particular, if $T$ is nilpotent then $T^r=0$ for big enough $r$. This implies that the resulting rational function is a polynomial, and this can only happen if $\Tr(T^i)=0$ for every $i>0$.
\end{proof}
\begin{definition}
A good rational function $Z(X)\in K[[X]]$ is a rational function that can be written as a quotient $\frac{P(X)}{Q(X)}$ such that $Q$ has no multiple roots, $\deg(P)\leq \deg(Q)$, and $Q(0)\neq 0$.
\end{definition}
The categories we will construct in this paper will use the notion of tensor ideals, which we recall here.
\begin{definition}\label{def:tensorideal}
Let $\D$ be a $K$-linear rigid symmetric monoidal category. A tensor ideal $\Nn\lhd \D$ is a collection of subspaces $\Nn(A,B)\subseteq \Hom_{\D}(A,B)$ for every $A,B\in\D$ that satisfies the following conditions:
\begin{itemize}
\item If $A,B,C,D\in \D$, $f\in \Nn(A,B)$, $g\in \Hom_{\D}(B,C)$, and $h\in \Hom_{\D}(D,A)$, then  $gfh\in \Nn(D,C)$. 
\item If $f\in \Nn(A,B)$ and $C\in\D$ then $f\ot 1_C\in \Nn(A\ot C,B\ot C)$.
\end{itemize}
\end{definition}
\begin{remark}
By using the symmetry isomorphisms and the first part of the definition one can show that a tensor ideal is also closed under taking tensor product from the left with the identity morphism. 
\end{remark}
\begin{definition}
If $\D$ is a rigid symmetric monoidal category and $\Nn$ is a tensor ideal in $\D$, we define the quotient category $\D/\Nn$ to be the category with objects $\text{Ob}(\D/\Nn) = \text{Ob}(\D)$ and hom-spaces $\Hom_{\D/\Nn}(A,B) = \Hom_{\D}(A,B)/\Nn(A,B)$. This category is again a rigid symmetric monoidal category, and there is a canonical functor $\D\to \D/\Nn$. 
\end{definition}

\subsection{Geometric invariant theory}\label{subsec:GIT}
Fix a type $((p_i,q_i))$ of algebraic structures and a dimension $d\in\N$. In \cite{meirUIR} it was shown that there is an affine variety $U_d$ equipped with a $\GL_d(K)$-action such that isomorphism classes of structures of type $((p_i,q_i))$ of dimension $d$ in $\Vec_K$ are in one-to-one correspondence with $\GL_d(K)$-orbits in $U_d$.
We summarize in the next proposition some results about this action:
\begin{proposition}
\begin{enumerate}
\item If $W_1$ and $W_2$ are two closed $\GL_d(K)$-stable subsets of $U_d$ then there is $f\in K[U_d]^{\GL_d(K)}$ such that $f(W_1)=0$ and $f(W_2)=1$.  
\item The ring of invariants $K[U_d]^{\GL_d(K)}$ is finitely generated and therefore its maximal spectrum $Z=Spec_m(K[U_d]^{\GL_d(K)})$ is an affine variety. 
\item The map $\pi:U_d\to Z$ induced by the inclusion $K[U_d]^{\GL_d(K)}\hookrightarrow K[U_d]$ is surjective. 
\item The points in $Z$ are in one to one correspondence with the closed $\GL_d(K)$-orbits in $U_d$. 
\item For every $z\in Z$, the preimage $\pi^{-1}(z)$ is a union of $\GL_d(K)$-orbits, and it contains exactly one closed orbit. This closed orbit lies in the closure of all other orbits in $\pi^{-1}(z)$. 
\end{enumerate}
\end{proposition}
\begin{proof}
The first four parts of the proposition were proven in \cite{meirUIR}, Section 2. They follow from Lemma 3.3, Theorem 3.4, and Theorem 3.5 in \cite{Newstead}. The last part follows from the fact that if there are two closed orbits $\Ow_1$ and $\Ow_2$ in $\pi^{-1}(z)$ then we cannot distinguish them using invariant polynomials, contradicting the first part of the proposition.
\end{proof} 
\subsection{Finite dimensional algebras}\label{subsec:findimalg}
Let $R$ be a finite dimensional $K$-algebra. For $r\in R$ we define $L_r:R\to R$ to be the linear map $x\mapsto r\cdot x$. We write $\Tr_{reg}(r) = \Tr(L_r)$, and we define the trace pairing $R\ot R\to K$ by $r_1\ot r_2\mapsto \Tr_{reg}(r_1r_2)$. We claim the following:
\begin{lemma} The trace pairing is non-degenerate if and only if $R$ is semisimple. 
\end{lemma}
\begin{proof}
Assume first that $R$ is semisimple. By Wedderburn's Theorem, and since $K$ is algebraically closed, $R$ splits as the direct sum of matrix algebras, $R = \bigoplus_t \M_{n_t}(K)$. The trace pairing is non-degenerate on every matrix algebra $\M_n(K)$, where a dual basis of $\{e_{ij}\}$ is given by $\{\frac{1}{n}e_{ji}\}$. It follows that the trace pairing is non-degenerate also on $R$. 

In the other direction, assume that the trace pairing is non-degenerate. We will show that the Jacobson radical $J$ of $R$ is zero, which will imply that $R$ is semisimple. Let $r\in J$. Then $rx\in J$ is nilpotent for every $x\in R$, and as a result $L_{rx}:R\to R$ is nilpotent. This implies that $\Tr_{reg}(rx)=0$ for every $x\in R$. Since the trace pairing is non-degenerate, we get that $r=0$, and therefore $J=0$. 
\end{proof}

\section{Construction of the universal category}\label{sec:unicat}
Fix a type $((p_i,q_i))$. We begin by constructing the universal category $\C_{univ}$ and explain its universal property.
We will start with an auxiliary category $\C_0$. The objects in this category are symbols $W^{a,b}$ for $a,b\in\N$ which we shall think of as $W^{\ot a}\ot (W^*)^{\ot b}$.
The duality adjunction implies that we should have 
$$\Hom_{\C_0}(W^{\ot a}\ot (W^*)^{\ot b},W^{\ot c}\ot(W^*)^{\ot d})\cong \Hom_{\C_0}(W^{\ot d}\ot W^{\ot a},W^{\ot c}\ot W^{\ot b})\cong$$ $$ \Hom_{\C_0}(W^{\ot d+a},W^{\ot c+b}).$$
We define 
$$\Hom_{\C_0}(W^{a,b},W^{c,d}) = Con^{c+b,d+a}$$
If $f:W^{a,b}\to W^{c,d}$ is represented by a diagram $Di_2$ and $g:W^{c,d}\to W^{e,h}$ is represented by a diagram $Di_1$, then the composition $gf:W^{a,b}\to W^{e,h}$ is given pictorially by the following diagram:
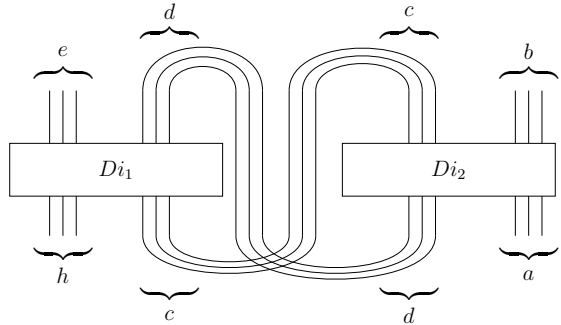
\begin{center}\scalebox{0.7}{
\begin{tikzpicture}
	\begin{pgfonlayer}{nodelayer}
		\node [style=none] (2) at (-5.75, 1.75) {};
		\node [style=none] (3) at (-5.5, 1.75) {};
		\node [style=none] (4) at (-4.25, 1.75) {};
		\node [style=none] (5) at (-4, 1.75) {};
		\node [style=none] (8) at (-5.75, 2.75) {};
		\node [style=none] (9) at (-5.5, 2.75) {};
		\node [style=none] (10) at (-4.25, 2.75) {};
		\node [style=none] (11) at (-4, 2.75) {};
		\node [style=none] (12) at (-5.75, 3) {$\overbrace{\phantom{aaaaa}}$};
		\node [style=none] (13) at (-5.75, 3.5) {$e$};
		\node [style=none] (14) at (-3.75, 3.75) {$\overbrace{\phantom{aaaaa}}$};
		\node [style=none] (15) at (-3.75, 4.25) {$d$};
		\node [style=none] (16) at (-2, 0) {};
		\node [style=none] (17) at (-2.25, 0) {};
		\node [style=none] (19) at (-2, 2.75) {};
		\node [style=none] (20) at (-2.25, 2.75) {};
		\node [style=none] (21) at (1, 1.75) {};
		\node [style=none] (22) at (0.75, 1.75) {};
		\node [style=none] (23) at (2.75, 1.75) {};
		\node [style=none] (24) at (3, 1.75) {};
		\node [style=none] (27) at (1, 2.75) {};
		\node [style=none] (28) at (0.75, 2.75) {};
		\node [style=none] (29) at (2.75, 2.75) {};
		\node [style=none] (30) at (3, 2.75) {};
		\node [style=none] (31) at (0.75, 3.75) {$\overbrace{\phantom{aaaaa}}$};
		\node [style=none] (32) at (0.75, 4.25) {$c$};
		\node [style=none] (33) at (3, 3) {$\overbrace{\phantom{aaaaa}}$};
		\node [style=none] (34) at (3, 3.5) {$b$};
		\node [style=none] (35) at (-5.75, 0) {};
		\node [style=none] (36) at (-5.5, 0) {};
		\node [style=none] (37) at (-4, 0) {};
		\node [style=none] (38) at (-4.25, 0) {};
		\node [style=none] (41) at (-5.75, 1) {};
		\node [style=none] (42) at (-5.5, 1) {};
		\node [style=none] (43) at (-4, 1) {};
		\node [style=none] (44) at (-4.25, 1) {};
		\node [style=none] (45) at (-5.75, -0.25) {$\underbrace{\phantom{aaaaa}}$};
		\node [style=none] (46) at (-5.75, -0.75) {$h$};
		\node [style=none] (47) at (-3.75, -1) {$\underbrace{\phantom{aaaaa}}$};
		\node [style=none] (48) at (-3.75, -1.5) {$c$};
		\node [style=none] (49) at (1, 0) {};
		\node [style=none] (50) at (0.75, 0) {};
		\node [style=none] (51) at (2.75, 0) {};
		\node [style=none] (52) at (3, 0) {};
		\node [style=none] (55) at (1, 1) {};
		\node [style=none] (56) at (0.75, 1) {};
		\node [style=none] (57) at (2.75, 1) {};
		\node [style=none] (58) at (3, 1) {};
		\node [style=none] (59) at (0.75, -1) {$\underbrace{\phantom{aaaaa}}$};
		\node [style=none] (60) at (0.75, -1.5) {$d$};
		\node [style=none] (61) at (3, -0.25) {$\underbrace{\phantom{aaaaa}}$};
		\node [style=none] (62) at (3, -0.75) {$a$};
		\node [style=none] (63) at (-1.25, 0) {};
		\node [style=none] (64) at (-1, 0) {};
		\node [style=none] (66) at (-1.25, 2.75) {};
		\node [style=none] (67) at (-1, 2.75) {};
		\node [style=multi function small] (68) at (-4.75, 1.25) {$Di_1$};
		\node [style=multi function small] (69) at (1.5, 1.25) {$Di_2$};
		\node [style=none] (70) at (-6, 1.75) {};
		\node [style=none] (71) at (-6, 2.75) {};
		\node [style=none] (72) at (-6, 0) {};
		\node [style=none] (73) at (-6, 1) {};
		\node [style=none] (74) at (3.25, 1.75) {};
		\node [style=none] (75) at (3.25, 2.75) {};
		\node [style=none] (76) at (3.25, 0) {};
		\node [style=none] (77) at (3.25, 1) {};
		\node [style=none] (78) at (1.25, 1.75) {};
		\node [style=none] (79) at (1.25, 2.75) {};
		\node [style=none] (80) at (1.25, 0) {};
		\node [style=none] (81) at (1.25, 1) {};
		\node [style=none] (83) at (-1.5, 2.75) {};
		\node [style=none] (84) at (-1.5, 0) {};
		\node [style=none] (87) at (-2.5, 2.75) {};
		\node [style=none] (88) at (-2.5, 0) {};
		\node [style=none] (90) at (-3.75, 1.75) {};
		\node [style=none] (91) at (-3.75, 2.75) {};
		\node [style=none] (92) at (-3.75, 0) {};
		\node [style=none] (93) at (-3.75, 1) {};
	\end{pgfonlayer}
	\begin{pgfonlayer}{edgelayer}
		\draw (8.center) to (2.center);
		\draw (9.center) to (3.center);
		\draw (10.center) to (4.center);
		\draw (11.center) to (5.center);
		\draw (19.center) to (16.center);
		\draw (20.center) to (17.center);
		\draw [bend left=90, looseness=1.25] (10.center) to (19.center);
		\draw [bend left=90, looseness=1.25] (11.center) to (20.center);
		\draw (27.center) to (21.center);
		\draw (28.center) to (22.center);
		\draw (29.center) to (23.center);
		\draw (30.center) to (24.center);
		\draw (41.center) to (35.center);
		\draw (42.center) to (36.center);
		\draw (43.center) to (37.center);
		\draw (44.center) to (38.center);
		\draw (55.center) to (49.center);
		\draw (56.center) to (50.center);
		\draw (57.center) to (51.center);
		\draw (58.center) to (52.center);
		\draw [bend right=90, looseness=0.75] (17.center) to (49.center);
		\draw [bend right=90, looseness=0.75] (16.center) to (50.center);
		\draw (66.center) to (63.center);
		\draw (67.center) to (64.center);
		\draw [bend right=90, looseness=0.75] (38.center) to (64.center);
		\draw [bend right=90, looseness=0.75] (37.center) to (63.center);
		\draw [bend left=90] (67.center) to (28.center);
		\draw [bend left=90] (66.center) to (27.center);
		\draw (71.center) to (70.center);
		\draw (73.center) to (72.center);
		\draw (75.center) to (74.center);
		\draw (77.center) to (76.center);
		\draw (79.center) to (78.center);
		\draw (81.center) to (80.center);
		\draw (91.center) to (90.center);
		\draw (93.center) to (92.center);
		\draw [bend left=90, looseness=1.25] (91.center) to (87.center);
		\draw [bend right=90, looseness=0.75] (92.center) to (84.center);
		\draw [bend left=270] (79.center) to (83.center);
		\draw [bend right=270, looseness=0.75] (80.center) to (88.center);
		\draw (83.center) to (84.center);
		\draw (87.center) to (88.center);
	\end{pgfonlayer}
\end{tikzpicture}}
\end{center}
The identity morphism in $\Hom_{\C_0}(W^{a,b},W^{a,b})$ is given by 
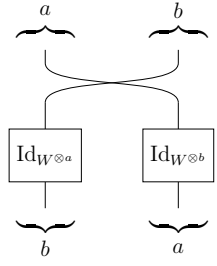
\begin{center}\scalebox{0.7}{
\begin{tikzpicture}
	\begin{pgfonlayer}{nodelayer}
		\node [style=1function] (0) at (-3.75, 2.5) {$\Id_{W^{\ot a}}$};
		\node [style=1function] (6) at (-1.25, 2.5) {$\Id_{W^{\ot b}}$};
		\node [style=none] (8) at (-3.75, 3.5) {};
		\node [style=none] (11) at (-1.25, 3.5) {};
		\node [style=none] (12) at (-3.75, 4.25) {};
		\node [style=none] (14) at (-1.25, 4.25) {};
		\node [style=none] (16) at (-3.75, 2) {};
		\node [style=none] (19) at (-1.25, 2) {};
		\node [style=none] (20) at (-3.75, 1.5) {};
		\node [style=none] (23) at (-1.25, 1.5) {};
		\node [style=none] (24) at (-3.75, 3) {};
		\node [style=none] (27) at (-1.25, 3) {};
		\node [style=none] (28) at (-3.75, 1.25) {$\underbrace{\phantom{aaaaa}}$};
		\node [style=none] (29) at (-3.75, 0.75) {$b$};
		\node [style=none] (30) at (-1.25, 1.25) {$\underbrace{\phantom{aaaaa}}$};
		\node [style=none] (31) at (-1.25, 0.75) {$a$};
		\node [style=none] (32) at (-3.75, 4.5) {};
		\node [style=none] (33) at (-1.25, 4.5) {};
		\node [style=none] (34) at (-3.75, 4.75) {$\overbrace{\phantom{aaaaa}}$};
		\node [style=none] (35) at (-1.25, 5.25) {$b$};
		\node [style=none] (36) at (-1.25, 4.75) {$\overbrace{\phantom{aaaaa}}$};
		\node [style=none] (37) at (-3.75, 5.25) {$a$};
	\end{pgfonlayer}
	\begin{pgfonlayer}{edgelayer}
		\draw (16.center) to (20.center);
		\draw (19.center) to (23.center);
		\draw [in=90, out=-90, looseness=0.50] (14.center) to (8.center);
		\draw [in=90, out=-90, looseness=0.50] (12.center) to (11.center);
		\draw (8.center) to (24.center);
		\draw (11.center) to (27.center);
		\draw (32.center) to (12.center);
		\draw (33.center) to (14.center);
	\end{pgfonlayer}
\end{tikzpicture}}
\end{center}
We used here single strings to represent bundles of $a$ strings and of $b$ strings. 
If $m_3:W^{a,b}\to W^{c,d}, m_2:W^{c,d}\to W^{e,f}$, and $m_1:W^{e,f}\to W^{s,t}$, are morphisms represented by the diagrams $Di_1,Di_2$ and $Di_3$ respectively, then the associativity axiom $(m_1m_2)m_3 = m_1(m_2m_3)$ follows from showing that both compositions are represented by the diagram:
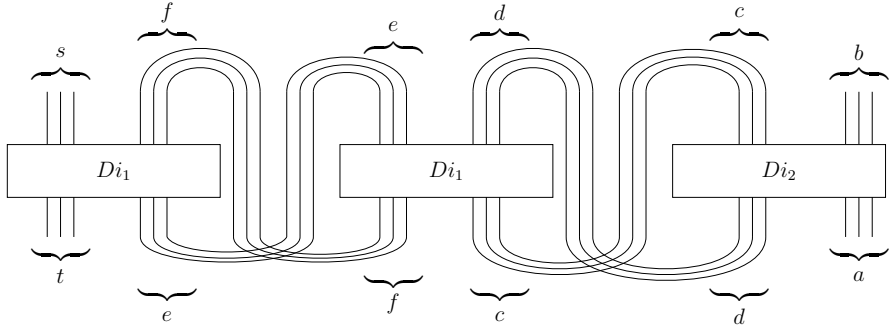
\begin{center}\scalebox{0.7}{
\begin{tikzpicture}
	\begin{pgfonlayer}{nodelayer}
		\node [style=none] (0) at (0.75, 1.75) {};
		\node [style=none] (1) at (1, 1.75) {};
		\node [style=none] (2) at (2.25, 1.75) {};
		\node [style=none] (3) at (2.5, 1.75) {};
		\node [style=none] (4) at (0.75, 2.75) {};
		\node [style=none] (5) at (1, 2.75) {};
		\node [style=none] (6) at (2.25, 2.75) {};
		\node [style=none] (7) at (2.5, 2.75) {};
		\node [style=none] (8) at (0.75, 3.5) {$\overbrace{\phantom{aaaaa}}$};
		\node [style=none] (9) at (0.75, 4) {$e$};
		\node [style=none] (10) at (2.75, 3.75) {$\overbrace{\phantom{aaaaa}}$};
		\node [style=none] (11) at (2.75, 4.25) {$d$};
		\node [style=none] (12) at (4.5, 0) {};
		\node [style=none] (13) at (4.25, 0) {};
		\node [style=none] (14) at (4.5, 2.75) {};
		\node [style=none] (15) at (4.25, 2.75) {};
		\node [style=none] (16) at (7.5, 1.75) {};
		\node [style=none] (17) at (7.25, 1.75) {};
		\node [style=none] (18) at (9.25, 1.75) {};
		\node [style=none] (19) at (9.5, 1.75) {};
		\node [style=none] (20) at (7.5, 2.75) {};
		\node [style=none] (21) at (7.25, 2.75) {};
		\node [style=none] (22) at (9.25, 2.75) {};
		\node [style=none] (23) at (9.5, 2.75) {};
		\node [style=none] (24) at (7.25, 3.75) {$\overbrace{\phantom{aaaaa}}$};
		\node [style=none] (25) at (7.25, 4.25) {$c$};
		\node [style=none] (26) at (9.5, 3) {$\overbrace{\phantom{aaaaa}}$};
		\node [style=none] (27) at (9.5, 3.5) {$b$};
		\node [style=none] (28) at (0.75, 0) {};
		\node [style=none] (29) at (1, 0) {};
		\node [style=none] (30) at (2.5, 0) {};
		\node [style=none] (31) at (2.25, 0) {};
		\node [style=none] (32) at (0.75, 1) {};
		\node [style=none] (33) at (1, 1) {};
		\node [style=none] (34) at (2.5, 1) {};
		\node [style=none] (35) at (2.25, 1) {};
		\node [style=none] (36) at (0.75, -0.75) {$\underbrace{\phantom{aaaaa}}$};
		\node [style=none] (37) at (0.75, -1.25) {$f$};
		\node [style=none] (38) at (2.75, -1) {$\underbrace{\phantom{aaaaa}}$};
		\node [style=none] (39) at (2.75, -1.5) {$c$};
		\node [style=none] (40) at (7.5, 0) {};
		\node [style=none] (41) at (7.25, 0) {};
		\node [style=none] (42) at (9.25, 0) {};
		\node [style=none] (43) at (9.5, 0) {};
		\node [style=none] (44) at (7.5, 1) {};
		\node [style=none] (45) at (7.25, 1) {};
		\node [style=none] (46) at (9.25, 1) {};
		\node [style=none] (47) at (9.5, 1) {};
		\node [style=none] (48) at (7.25, -1) {$\underbrace{\phantom{aaaaa}}$};
		\node [style=none] (49) at (7.25, -1.5) {$d$};
		\node [style=none] (50) at (9.5, -0.25) {$\underbrace{\phantom{aaaaa}}$};
		\node [style=none] (51) at (9.5, -0.75) {$a$};
		\node [style=none] (52) at (5.25, 0) {};
		\node [style=none] (53) at (5.5, 0) {};
		\node [style=none] (54) at (5.25, 2.75) {};
		\node [style=none] (55) at (5.5, 2.75) {};
		\node [style=multi function small] (56) at (1.75, 1.25) {$Di_1$};
		\node [style=multi function small] (57) at (8, 1.25) {$Di_2$};
		\node [style=none] (58) at (0.5, 1.75) {};
		\node [style=none] (59) at (0.5, 2.75) {};
		\node [style=none] (60) at (0.5, 0) {};
		\node [style=none] (61) at (0.5, 1) {};
		\node [style=none] (62) at (9.75, 1.75) {};
		\node [style=none] (63) at (9.75, 2.75) {};
		\node [style=none] (64) at (9.75, 0) {};
		\node [style=none] (65) at (9.75, 1) {};
		\node [style=none] (66) at (7.75, 1.75) {};
		\node [style=none] (67) at (7.75, 2.75) {};
		\node [style=none] (68) at (7.75, 0) {};
		\node [style=none] (69) at (7.75, 1) {};
		\node [style=none] (71) at (5, 2.75) {};
		\node [style=none] (72) at (5, 0) {};
		\node [style=none] (75) at (4, 2.75) {};
		\node [style=none] (76) at (4, 0) {};
		\node [style=none] (78) at (2.75, 1.75) {};
		\node [style=none] (79) at (2.75, 2.75) {};
		\node [style=none] (80) at (2.75, 0) {};
		\node [style=none] (81) at (2.75, 1) {};
		\node [style=none] (82) at (-5.5, 1.75) {};
		\node [style=none] (83) at (-5.25, 1.75) {};
		\node [style=none] (84) at (-4, 1.75) {};
		\node [style=none] (85) at (-3.75, 1.75) {};
		\node [style=none] (86) at (-5.5, 2.75) {};
		\node [style=none] (87) at (-5.25, 2.75) {};
		\node [style=none] (88) at (-4, 2.75) {};
		\node [style=none] (89) at (-3.75, 2.75) {};
		\node [style=none] (90) at (-5.5, 3) {$\overbrace{\phantom{aaaaa}}$};
		\node [style=none] (91) at (-5.5, 3.5) {$s$};
		\node [style=none] (92) at (-3.5, 3.75) {$\overbrace{\phantom{aaaaa}}$};
		\node [style=none] (93) at (-3.5, 4.25) {$f$};
		\node [style=none] (94) at (-1.75, 0) {};
		\node [style=none] (95) at (-2, 0) {};
		\node [style=none] (96) at (-1.75, 2.75) {};
		\node [style=none] (97) at (-2, 2.75) {};
		\node [style=none] (110) at (-5.5, 0) {};
		\node [style=none] (111) at (-5.25, 0) {};
		\node [style=none] (112) at (-3.75, 0) {};
		\node [style=none] (113) at (-4, 0) {};
		\node [style=none] (114) at (-5.5, 1) {};
		\node [style=none] (115) at (-5.25, 1) {};
		\node [style=none] (116) at (-3.75, 1) {};
		\node [style=none] (117) at (-4, 1) {};
		\node [style=none] (118) at (-5.5, -0.25) {$\underbrace{\phantom{aaaaa}}$};
		\node [style=none] (119) at (-5.5, -0.75) {$t$};
		\node [style=none] (120) at (-3.5, -1) {$\underbrace{\phantom{aaaaa}}$};
		\node [style=none] (121) at (-3.5, -1.5) {$e$};
		\node [style=multi function small] (138) at (-4.5, 1.25) {$Di_1$};
		\node [style=none] (140) at (-5.75, 1.75) {};
		\node [style=none] (141) at (-5.75, 2.75) {};
		\node [style=none] (142) at (-5.75, 0) {};
		\node [style=none] (143) at (-5.75, 1) {};
		\node [style=none] (157) at (-2.25, 2.75) {};
		\node [style=none] (158) at (-2.25, 0) {};
		\node [style=none] (160) at (-3.5, 1.75) {};
		\node [style=none] (161) at (-3.5, 2.75) {};
		\node [style=none] (162) at (-3.5, 0) {};
		\node [style=none] (163) at (-3.5, 1) {};
		\node [style=none] (164) at (-0.75, 0) {};
		\node [style=none] (165) at (-1, 0) {};
		\node [style=none] (166) at (-0.75, 2.75) {};
		\node [style=none] (167) at (-1, 2.75) {};
		\node [style=none] (169) at (-1.25, 2.75) {};
		\node [style=none] (170) at (-1.25, 0) {};
	\end{pgfonlayer}
	\begin{pgfonlayer}{edgelayer}
		\draw (4.center) to (0.center);
		\draw (5.center) to (1.center);
		\draw (6.center) to (2.center);
		\draw (7.center) to (3.center);
		\draw (14.center) to (12.center);
		\draw (15.center) to (13.center);
		\draw [bend left=90, looseness=1.25] (6.center) to (14.center);
		\draw [bend left=90, looseness=1.25] (7.center) to (15.center);
		\draw (20.center) to (16.center);
		\draw (21.center) to (17.center);
		\draw (22.center) to (18.center);
		\draw (23.center) to (19.center);
		\draw (32.center) to (28.center);
		\draw (33.center) to (29.center);
		\draw (34.center) to (30.center);
		\draw (35.center) to (31.center);
		\draw (44.center) to (40.center);
		\draw (45.center) to (41.center);
		\draw (46.center) to (42.center);
		\draw (47.center) to (43.center);
		\draw [bend right=90, looseness=0.75] (13.center) to (40.center);
		\draw [bend right=90, looseness=0.75] (12.center) to (41.center);
		\draw (54.center) to (52.center);
		\draw (55.center) to (53.center);
		\draw [bend right=90, looseness=0.75] (31.center) to (53.center);
		\draw [bend right=90, looseness=0.75] (30.center) to (52.center);
		\draw [bend left=90] (55.center) to (21.center);
		\draw [bend left=90] (54.center) to (20.center);
		\draw (59.center) to (58.center);
		\draw (61.center) to (60.center);
		\draw (63.center) to (62.center);
		\draw (65.center) to (64.center);
		\draw (67.center) to (66.center);
		\draw (69.center) to (68.center);
		\draw (79.center) to (78.center);
		\draw (81.center) to (80.center);
		\draw [bend left=90, looseness=1.25] (79.center) to (75.center);
		\draw [bend right=90, looseness=0.75] (80.center) to (72.center);
		\draw [bend left=270] (67.center) to (71.center);
		\draw [bend right=270, looseness=0.75] (68.center) to (76.center);
		\draw (86.center) to (82.center);
		\draw (87.center) to (83.center);
		\draw (88.center) to (84.center);
		\draw (89.center) to (85.center);
		\draw (96.center) to (94.center);
		\draw (97.center) to (95.center);
		\draw [bend left=90, looseness=1.25] (88.center) to (96.center);
		\draw [bend left=90, looseness=1.25] (89.center) to (97.center);
		\draw (114.center) to (110.center);
		\draw (115.center) to (111.center);
		\draw (116.center) to (112.center);
		\draw (117.center) to (113.center);
		\draw (141.center) to (140.center);
		\draw (143.center) to (142.center);
		\draw (161.center) to (160.center);
		\draw (163.center) to (162.center);
		\draw [bend left=90, looseness=1.25] (161.center) to (157.center);
		\draw (166.center) to (164.center);
		\draw (167.center) to (165.center);
		\draw [bend right=90, looseness=0.50] (94.center) to (60.center);
		\draw [bend right=90, looseness=0.50] (95.center) to (28.center);
		\draw [bend right=90, looseness=0.50] (158.center) to (29.center);
		\draw [bend right=60, looseness=0.50] (162.center) to (170.center);
		\draw [bend right=90, looseness=0.50] (112.center) to (165.center);
		\draw [bend right=90, looseness=0.50] (113.center) to (164.center);
		\draw [bend left=90] (166.center) to (59.center);
		\draw [bend left=90] (169.center) to (5.center);
		\draw [bend left=90] (167.center) to (4.center);
		\draw (157.center) to (158.center);
		\draw (169.center) to (170.center);
		\draw (75.center) to (76.center);
		\draw (72.center) to (71.center);
	\end{pgfonlayer}
\end{tikzpicture}}
\end{center}
The identity axiom follows from Definition 4.1. in \cite{meirUIR}.

The category $\C_0$ is also a symmetric monoidal category. 
The tensor product functor is given on objects by $$W^{a,b}\ot W^{c,d}= W^{a+c,b+d}.$$ If $m_1:W^{a_1,b_1}\to W^{a_2,b_2}$ and $m_2:W^{c_1,d_1}\to W^{c_2,d_2}$ are morphisms represented by diagrams $Di_1$ and $Di_2$, then the tensor product $m_1\ot m_2$ is represented by the diagram
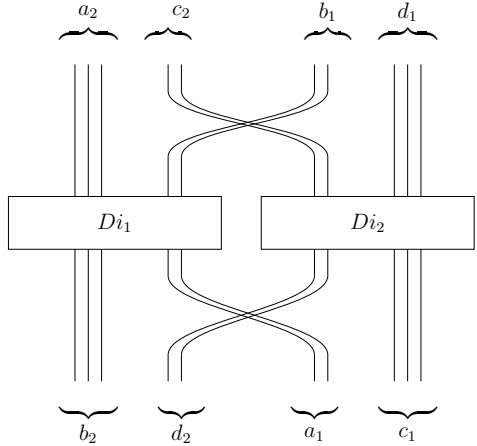
\begin{center}\scalebox{0.7}{
\begin{tikzpicture}
	\begin{pgfonlayer}{nodelayer}
		\node [style=none] (2) at (-3.75, 1.75) {};
		\node [style=none] (3) at (-3.5, 1.75) {};
		\node [style=none] (4) at (-2.25, 1.75) {};
		\node [style=none] (5) at (-2, 1.75) {};
		\node [style=none] (8) at (-3.75, 4.25) {};
		\node [style=none] (9) at (-3.5, 4.25) {};
		\node [style=none] (10) at (-2.25, 2.5) {};
		\node [style=none] (11) at (-2, 2.5) {};
		\node [style=none] (12) at (-3.75, 4.75) {$\overbrace{\phantom{aaaaa}}$};
		\node [style=none] (13) at (-3.75, 5.25) {$a_2$};
		\node [style=none] (14) at (-2.25, 4.75) {$\overbrace{\phantom{aaaa}}$};
		\node [style=none] (15) at (0.75, 5.25) {$b_1$};
		\node [style=none] (21) at (0.75, 1.75) {};
		\node [style=none] (22) at (0.5, 1.75) {};
		\node [style=none] (23) at (2, 1.75) {};
		\node [style=none] (24) at (2.25, 1.75) {};
		\node [style=none] (27) at (0.75, 2.5) {};
		\node [style=none] (28) at (0.5, 2.5) {};
		\node [style=none] (29) at (2, 4.25) {};
		\node [style=none] (30) at (2.25, 4.25) {};
		\node [style=none] (31) at (0.75, 4.75) {$\overbrace{\phantom{aaaa}}$};
		\node [style=none] (32) at (-2, 5.25) {$c_2$};
		\node [style=none] (33) at (2.25, 4.75) {$\overbrace{\phantom{aaaaa}}$};
		\node [style=none] (34) at (2.25, 5.25) {$d_1$};
		\node [style=none] (35) at (-3.75, -1.75) {};
		\node [style=none] (36) at (-3.5, -1.75) {};
		\node [style=none] (37) at (-2, 0.25) {};
		\node [style=none] (38) at (-2.25, 0.25) {};
		\node [style=none] (41) at (-3.75, 1) {};
		\node [style=none] (42) at (-3.5, 1) {};
		\node [style=none] (43) at (-2, 1) {};
		\node [style=none] (44) at (-2.25, 1) {};
		\node [style=none] (45) at (-3.75, -2.25) {$\underbrace{\phantom{aaaaa}}$};
		\node [style=none] (46) at (-3.75, -2.75) {$b_2$};
		\node [style=none] (47) at (-2, -2.25) {$\underbrace{\phantom{aaaa}}$};
		\node [style=none] (48) at (0.5, -2.75) {$a_1$};
		\node [style=none] (49) at (0.75, 0.25) {};
		\node [style=none] (50) at (0.5, 0.25) {};
		\node [style=none] (51) at (2, -1.75) {};
		\node [style=none] (52) at (2.25, -1.75) {};
		\node [style=none] (55) at (0.75, 1) {};
		\node [style=none] (56) at (0.5, 1) {};
		\node [style=none] (57) at (2, 1) {};
		\node [style=none] (58) at (2.25, 1) {};
		\node [style=none] (59) at (0.5, -2.25) {$\underbrace{\phantom{aaaa}}$};
		\node [style=none] (60) at (-2, -2.75) {$d_2$};
		\node [style=none] (61) at (2.25, -2.25) {$\underbrace{\phantom{aaaaa}}$};
		\node [style=none] (62) at (2.25, -2.75) {$c_1$};
		\node [style=multi function small] (68) at (-3.25, 1.25) {$Di_1$};
		\node [style=multi function small] (69) at (1.5, 1.25) {$Di_2$};
		\node [style=none] (70) at (-2.25, -1.25) {};
		\node [style=none] (71) at (-2, -1.25) {};
		\node [style=none] (72) at (-2, -1.75) {};
		\node [style=none] (73) at (-2.25, -1.75) {};
		\node [style=none] (74) at (0.5, -1.25) {};
		\node [style=none] (75) at (0.75, -1.25) {};
		\node [style=none] (76) at (0.75, -1.75) {};
		\node [style=none] (77) at (0.5, -1.75) {};
		\node [style=none] (79) at (0.5, 3.75) {};
		\node [style=none] (80) at (0.75, 3.75) {};
		\node [style=none] (81) at (0.75, 4.25) {};
		\node [style=none] (82) at (0.5, 4.25) {};
		\node [style=none] (83) at (-2, 4.25) {};
		\node [style=none] (84) at (-2, 3.75) {};
		\node [style=none] (85) at (-2.25, 3.75) {};
		\node [style=none] (86) at (-2.25, 4.25) {};
		\node [style=none] (87) at (-4, 1.75) {};
		\node [style=none] (88) at (-4, 4.25) {};
		\node [style=none] (89) at (-4, -1.75) {};
		\node [style=none] (90) at (-4, 1) {};
		\node [style=none] (91) at (2.5, 1.75) {};
		\node [style=none] (92) at (2.5, 4.25) {};
		\node [style=none] (93) at (2.5, -1.75) {};
		\node [style=none] (94) at (2.5, 1) {};
	\end{pgfonlayer}
	\begin{pgfonlayer}{edgelayer}
		\draw (8.center) to (2.center);
		\draw (9.center) to (3.center);
		\draw (10.center) to (4.center);
		\draw (11.center) to (5.center);
		\draw (27.center) to (21.center);
		\draw (28.center) to (22.center);
		\draw (29.center) to (23.center);
		\draw (30.center) to (24.center);
		\draw (41.center) to (35.center);
		\draw (42.center) to (36.center);
		\draw (43.center) to (37.center);
		\draw (44.center) to (38.center);
		\draw (55.center) to (49.center);
		\draw (56.center) to (50.center);
		\draw (57.center) to (51.center);
		\draw (58.center) to (52.center);
		\draw [in=90, out=-90, looseness=0.50] (38.center) to (74.center);
		\draw [in=90, out=-90, looseness=0.50] (37.center) to (75.center);
		\draw (75.center) to (76.center);
		\draw (74.center) to (77.center);
		\draw [in=90, out=-90, looseness=0.50] (50.center) to (70.center);
		\draw [in=90, out=-90, looseness=0.50] (49.center) to (71.center);
		\draw (71.center) to (72.center);
		\draw (70.center) to (73.center);
		\draw [in=-90, out=90, looseness=0.50] (11.center) to (80.center);
		\draw [in=-90, out=90, looseness=0.50] (10.center) to (79.center);
		\draw [in=-90, out=90, looseness=0.50] (28.center) to (85.center);
		\draw [in=-90, out=90, looseness=0.50] (27.center) to (84.center);
		\draw (84.center) to (83.center);
		\draw (86.center) to (85.center);
		\draw (82.center) to (79.center);
		\draw (81.center) to (80.center);
		\draw (88.center) to (87.center);
		\draw (90.center) to (89.center);
		\draw (92.center) to (91.center);
		\draw (94.center) to (93.center);
	\end{pgfonlayer}
\end{tikzpicture}}
\end{center}
We have $$(W^{a,b}\ot W^{c,d})\ot W^{e,f} = W^{a+c+e,b+d+f} = W^{a,b}\ot (W^{c,d}\ot W^{e,f}).$$ The associativity isomorphism is given simply by the identity. The tensor unit is $\one:=W^{0,0}$. The unit isomorphism $W^{0,0}\ot W^{a,b} = W^{a,b}\to W^{a,b}$ and $W^{a,b}\ot W^{0,0}= W^{a,b}\to W^{a,b}$ are just the identity morphisms as well, and $\C_0$ is thus a strict monoidal category.
 
The symmetric structure on $\C_{0}$ is given by the collection of morphisms $W^{a,b}\ot W^{c,d}\to W^{c,d}\ot W^{a,b}$ or $W^{a+c,b+d}\to W^{c+a,d+b}$ given diagrammatically by 
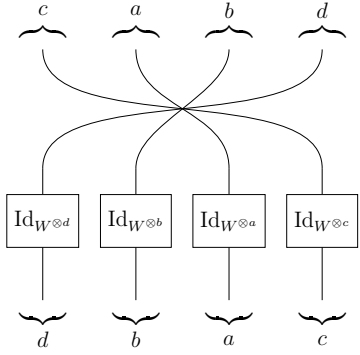
\begin{center}\scalebox{0.7}{
\begin{tikzpicture}
	\begin{pgfonlayer}{nodelayer}
		\node [style=1function] (0) at (-6.25, 2.5) {$\Id_{W^{\ot d}}$};
		\node [style=1function] (1) at (-4.5, 2.5) {$\Id_{W^{\ot b}}$};
		\node [style=1function] (5) at (-2.75, 2.5) {$\Id_{W^{\ot a}}$};
		\node [style=1function] (6) at (-1, 2.5) {$\Id_{W^{\ot c}}$};
		\node [style=none] (8) at (-6.25, 3.5) {};
		\node [style=none] (9) at (-4.5, 3.5) {};
		\node [style=none] (10) at (-2.75, 3.5) {};
		\node [style=none] (11) at (-1, 3.5) {};
		\node [style=none] (12) at (-6.25, 5.75) {};
		\node [style=none] (13) at (-4.5, 5.75) {};
		\node [style=none] (14) at (-1, 5.75) {};
		\node [style=none] (15) at (-2.75, 5.75) {};
		\node [style=none] (16) at (-6.25, 2) {};
		\node [style=none] (17) at (-4.5, 2) {};
		\node [style=none] (18) at (-2.75, 2) {};
		\node [style=none] (19) at (-1, 2) {};
		\node [style=none] (20) at (-6.25, 1) {};
		\node [style=none] (21) at (-4.5, 1) {};
		\node [style=none] (22) at (-2.75, 1) {};
		\node [style=none] (23) at (-1, 1) {};
		\node [style=none] (24) at (-6.25, 3) {};
		\node [style=none] (25) at (-4.5, 3) {};
		\node [style=none] (26) at (-2.75, 3) {};
		\node [style=none] (27) at (-1, 3) {};
		\node [style=none] (28) at (-6.25, 0.75) {$\underbrace{\phantom{aaaa}}$};
		\node [style=none] (29) at (-6.25, 6.5) {$c$};
		\node [style=none] (30) at (-4.5, 0.75) {$\underbrace{\phantom{aaaa}}$};
		\node [style=none] (31) at (-4.5, 6.5) {$a$};
		\node [style=none] (32) at (-2.75, 0.75) {$\underbrace{\phantom{aaaa}}$};
		\node [style=none] (33) at (-2.75, 6.5) {$b$};
		\node [style=none] (34) at (-1, 0.75) {$\underbrace{\phantom{aaaa}}$};
		\node [style=none] (35) at (-1, 6.5) {$d$};
		\node [style=none] (36) at (-6.25, 6) {$\overbrace{\phantom{aaaa}}$};
		\node [style=none] (37) at (-6.25, 0.25) {$d$};
		\node [style=none] (38) at (-4.5, 6) {$\overbrace{\phantom{aaaa}}$};
		\node [style=none] (39) at (-4.5, 0.25) {$b$};
		\node [style=none] (40) at (-2.75, 6) {$\overbrace{\phantom{aaaa}}$};
		\node [style=none] (41) at (-2.75, 0.25) {$a$};
		\node [style=none] (42) at (-1, 6) {$\overbrace{\phantom{aaaa}}$};
		\node [style=none] (43) at (-1, 0.25) {$c$};
	\end{pgfonlayer}
	\begin{pgfonlayer}{edgelayer}
		\draw (16.center) to (20.center);
		\draw (17.center) to (21.center);
		\draw (18.center) to (22.center);
		\draw (19.center) to (23.center);
		\draw [in=90, out=-90, looseness=0.75] (15.center) to (9.center);
		\draw [in=90, out=-90, looseness=0.75] (14.center) to (8.center);
		\draw [in=90, out=-90] (13.center) to (10.center);
		\draw [in=90, out=-90, looseness=0.75] (12.center) to (11.center);
		\draw (8.center) to (24.center);
		\draw (9.center) to (25.center);
		\draw (10.center) to (26.center);
		\draw (11.center) to (27.center);
	\end{pgfonlayer}
\end{tikzpicture}}
\end{center}
(In this diagram, as in the diagram representing the identity morphism, a single string represents a bundle of strings for simplification).
The category $\C_0$ is also rigid. That is, every object has a dual object, the dual of $W^{a,b}$ being $W^{b,a}$. It will be enough to describe the evaluation and coevaluation for $W:=W^{0,1}$ since $W$ and $W^*$ tensor-generate the category $\C_0$. The evaluation $W^*\ot W\to \one$ is given by $\Id_W\in Con^{1,1}$. Similarly, the coevaluation $\one\to W\ot W^*$ is given by the same element $Id_W\in Con^{1,1}$ (notice that we realize here $Con^{1,1}$ as the hom-space between objects in the category in two different ways).
A direct calculation with diagrams now reveals the fact that the compositions $W\to W\ot W^*\ot W\to W$ and $W^*\to W^*\ot W\ot W^*\to W^*$ are indeed the identity morphisms. 

Finally, we define $\C_{univ}$ to be the additive envelope of $\C_0$. Its objects are finite direct sums of objects of $\C_0$, and the morphisms are given by $$\Hom_{\C_{univ}}(\oplus_i A_i,\oplus_j B_j) = \bigoplus_{i,j}\Hom_{\C_0}(A_i,B_j),$$ where $A_i$ and $B_j$ are objects of $\C_0$. We can thus consider $W = W^{1,0}$ as an object of $\C_{univ}$. Moreover, since $x_i\in Con^{p_i,q_i}$, we get the structure $(W,(x_i))$ of type $((p_i,q_i))$ in $\C_{univ}$, which we call the \emph{tautological structure}. We claim the following:

\begin{proposition}\label{prop:universalproperty}
Let $\D$ be a $K$-linear symmetric monoidal category. Isomorphism classes of symmetric monoidal $K$-linear functors $F:\C_{univ}\to \D$ are in one to one correspondence with isomorphism classes of dualizable algebraic structures of type $((p_i,q_i))$ in $\D$. 
\end{proposition}
\begin{proof}
Given such a functor $F$ we can consider the $\D$-object $A:=F(W)$. 
Since $x_i$ can be considered as a morphism in $\Hom_{\C_{univ}}(W^{\ot q_i},W^{\ot p_i})$, its image $y_i:=F(x_i)$ will be a morphism in $\Hom_{\D}(A^{\ot q_i},A^{\ot p_i})$, where we use here the monoidality of $F$ to identify $F(W^{\ot n})$ with $F(W)^{\ot n}$. We thus get the algebraic structure $(A,(y_i))$ in $\D$. The object $A$ is dualizable with dual $F(W^*)$ (we use here again the monoidality of $F$).

Conversely, assume that $(A,(y_i))$ is an algebraic structure of type $((p_i,q_i))$ in $\D$. It will be enough to define the corresponding functor on the category $\C_0$, as this will extends uniquely to a functor from $\C_{univ}$ by additivity. We define $F$ on objects by $F(W^{a,b}) = A^{a,b} = A^{\ot a}\ot (A^*)^{\ot b}$. The action of $F$ on morphisms is defined using the maps $\Re^{p,q}$ from Subsection \ref{subsec:algstr}.

The monoidal structure on $F$ is given by the composition $$F(W^{a,b})\ot F(W^{c,d}) = A^{\ot a}\ot (A^*)^{\ot b}\ot A^{\ot c}\ot (A^*)^{\ot d}\cong $$ $$ A^{\ot a}\ot A^{\ot c}\ot (A^*)^{\ot b}\ot (A^*)^{\ot d} = A^{\ot a+c}\ot (A^*)^{b+d} = F(W^{a+c,b+d}) = F(W^{a,b}\ot W^{c,d})$$ where we used here the symmetric monoidal isomorphism $(A^*)^{\ot b}\ot A^{\ot c}\to A^{\ot c}\ot (A^*)^{\ot b}$ in $\D$. Finally. an isomorphism between structures will give rise to an isomorphism between functors and vice versa.
\end{proof}
\begin{definition} The functor $F:\C_{univ}\to \D$ constructed in the proposition will be denoted by $F_A$. 
\end{definition}
\subsection{Universal categories for theories}
We recall from Section 7 in \cite{meirUIR} that a \emph{theory} is a subset $\T\subseteq \sqcup Con^{p,q}$. The elements of $\T$ are called \emph{axioms}. \emph{Models} of $\T$ are structures $(A,(y_i))$ in some $K$-linear rigid symmetric monoidal category $\D$ such that for every $x\in \T$ it holds that $F_A(x)=0$, where we interpret $x\in Con^{p,q}$ as a morphism $W^{\ot q}\to W^{\ot p}$ in $\C_{univ}$. Axioms can describe associativity, commutativity, the Jacobi identity, and so on.
\begin{definition} We write $\I^{\T}$ for the tensor ideal in $\C_{univ}$ generated by the elements of $\T$, interpreted as morphisms in $\C_{univ}$ as above. We define $\C_{univ}^{\T} = \C_{univ}/\I_{\T}$. 
\end{definition}
\begin{remark} The tensor ideal generated by a set of morphisms always exists. We can either describe it as the intersection of all tensor ideals that contain the given set, or as the collection of morphisms generated from the given set by taking compositions and tensor products. 
\end{remark}
The following universal property follows immediately from the universal property of $\C_{univ}$ (Proposition \ref{prop:universalproperty}), and the universal property of the quotient category: 
\begin{proposition}\label{prop:universaltheory}
Let $\D$ be a $K$-linear symmetric monoidal category. Isomorphism classes of symmetric monoidal $K$-linear functors $F:\C_{univ}^{\T}\to \D$ are in one to one correspondence with isomorphism classes of dualizable models of $\T$ in $\D$. 
\end{proposition}
\begin{definition} We write $\winf^{\T} := \winf/I_{\T}$, where $I_{\T} = \I_{\T}(\one,\one)$
\end{definition}
\begin{remark}
We have $\End_{\C^{\T}_{univ}}(\one) = \winf /I_{\T}$, where $I_{\T}$ is the ideal generated by $\pair(x,y)\in \winf$ for $x\in Con^{p,q}\cap \T$ and $y\in Con^{q,p}$, for some $p,q\in \N$. This ideal was denoted by $\I_{\T}$ in Section 7 of \cite{meirUIR}. We write here $\I_{\T}$ for the tensor ideal in $\C_{univ}$. 
\end{remark} 
We shall see in Section \ref{sec:examples} that theories can provide many good characters. Indeed, in many interesting cases it will hold that the hom-spaces in $\C_{univ}^{\T}$ are already finite dimensional. 

\section{The category $\C_{\chi}$ and a proof of Theorem \ref{thm:main1}}\label{sec:constcchi}
The construction in this section will be based on a fixed type $((p_i,q_i))$ of algebraic structures. Let $\chi:\winf\to K$ be a character. It induces pairings $$\pair^{p,q}_{\chi}:Con^{p,q}\ot Con^{q,p}\stackrel{\pair^{p,q}}{\longrightarrow} \winf\stackrel{\chi}{\to} K$$ for every $p,q\in \N$. We call $f\in Con^{p,q}$ \emph{$\chi$-negligible} if it is in the radical of the above pairing. We denote by $N^{p,q}_{\chi}\subseteq Con^{p,q}$ the subspace of all $\chi$-negligible morphisms. 
We will now construct the category $\C_{\chi}$ as a quotient of the category $\C_{univ}$. For every $a,b,c,d\in \N$ we define $$\Nn_{\chi}(W^{a,b},W^{c,d}) = N^{c+b,d+a}_{\chi}\subseteq Con^{c+b,d+a} = \Hom_{\C_{univ}}(W^{a,b},W^{c,d}).$$
We extend the definition of $\Nn_{\chi}$ to all objects of $\C_{univ}$ by the rule 
$\Nn_{\chi}(\oplus_i A_i,\oplus_j B_j) = \bigoplus_{i,j}\Nn_{\chi}(A_i,B_j)$.
We claim the following:
\begin{lemma}\label{lem:pairideal}
A morphism $f\in \Hom_{\C_{univ}}(A_1,A_2)$ is in $\Nn_{\chi}(A_1,A_2)$ if and only if for every $g\in \Hom_{\C_{univ}}(A_2,A_1)$ it holds that $\chi(\Tr(f\circ g))=0$.
\end{lemma}
\begin{proof}
It will be enough to prove the claim for the case where $A_1= W^{a,b}$ and $A_2=W^{c,d}$. We use the fact that if $f$ is given by a diagram $Di_1$ and $g$ is given by a diagram $Di_2$ then $\Tr(f\circ g)$ is given by the closed diagram 
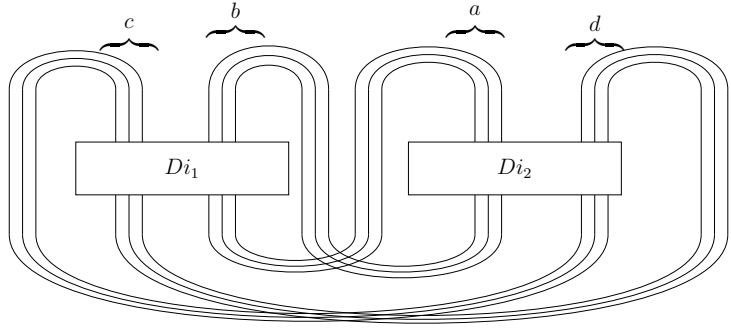
\begin{center}\scalebox{0.7}{
\begin{tikzpicture}
	\begin{pgfonlayer}{nodelayer}
		\node [style=none] (0) at (0.75, 1.75) {};
		\node [style=none] (1) at (1, 1.75) {};
		\node [style=none] (2) at (2.25, 1.75) {};
		\node [style=none] (3) at (2.5, 1.75) {};
		\node [style=none] (4) at (0.75, 2.75) {};
		\node [style=none] (5) at (1, 2.75) {};
		\node [style=none] (6) at (2.25, 2.75) {};
		\node [style=none] (7) at (2.5, 2.75) {};
		\node [style=none] (8) at (0.75, 3.5) {$\overbrace{\phantom{aaaaa}}$};
		\node [style=none] (9) at (0.75, 4) {$c$};
		\node [style=none] (10) at (2.75, 3.75) {$\overbrace{\phantom{aaaaa}}$};
		\node [style=none] (11) at (2.75, 4.25) {$b$};
		\node [style=none] (12) at (4.5, 0) {};
		\node [style=none] (13) at (4.25, 0) {};
		\node [style=none] (14) at (4.5, 2.75) {};
		\node [style=none] (15) at (4.25, 2.75) {};
		\node [style=none] (16) at (7.5, 1.75) {};
		\node [style=none] (17) at (7.25, 1.75) {};
		\node [style=none] (18) at (9.25, 1.75) {};
		\node [style=none] (19) at (9.5, 1.75) {};
		\node [style=none] (20) at (7.5, 2.75) {};
		\node [style=none] (21) at (7.25, 2.75) {};
		\node [style=none] (22) at (9.25, 2.75) {};
		\node [style=none] (23) at (9.5, 2.75) {};
		\node [style=none] (24) at (7.25, 3.75) {$\overbrace{\phantom{aaaaa}}$};
		\node [style=none] (25) at (7.25, 4.25) {$a$};
		\node [style=none] (26) at (9.5, 3.5) {$\overbrace{\phantom{aaaaa}}$};
		\node [style=none] (27) at (9.5, 4) {$d$};
		\node [style=none] (28) at (0.75, 0) {};
		\node [style=none] (29) at (1, 0) {};
		\node [style=none] (30) at (2.5, 0) {};
		\node [style=none] (31) at (2.25, 0) {};
		\node [style=none] (32) at (0.75, 1) {};
		\node [style=none] (33) at (1, 1) {};
		\node [style=none] (34) at (2.5, 1) {};
		\node [style=none] (35) at (2.25, 1) {};
		\node [style=none] (40) at (7.5, 0) {};
		\node [style=none] (41) at (7.25, 0) {};
		\node [style=none] (42) at (9.25, 0) {};
		\node [style=none] (43) at (9.5, 0) {};
		\node [style=none] (44) at (7.5, 1) {};
		\node [style=none] (45) at (7.25, 1) {};
		\node [style=none] (46) at (9.25, 1) {};
		\node [style=none] (47) at (9.5, 1) {};
		\node [style=none] (52) at (5.25, 0) {};
		\node [style=none] (53) at (5.5, 0) {};
		\node [style=none] (54) at (5.25, 2.75) {};
		\node [style=none] (55) at (5.5, 2.75) {};
		\node [style=multi function small] (56) at (1.75, 1.25) {$Di_1$};
		\node [style=multi function small] (57) at (8, 1.25) {$Di_2$};
		\node [style=none] (58) at (0.5, 1.75) {};
		\node [style=none] (59) at (0.5, 2.75) {};
		\node [style=none] (60) at (0.5, 0) {};
		\node [style=none] (61) at (0.5, 1) {};
		\node [style=none] (62) at (9.75, 1.75) {};
		\node [style=none] (63) at (9.75, 2.75) {};
		\node [style=none] (64) at (9.75, 0) {};
		\node [style=none] (65) at (9.75, 1) {};
		\node [style=none] (66) at (7.75, 1.75) {};
		\node [style=none] (67) at (7.75, 2.75) {};
		\node [style=none] (68) at (7.75, 0) {};
		\node [style=none] (69) at (7.75, 1) {};
		\node [style=none] (71) at (5, 2.75) {};
		\node [style=none] (72) at (5, 0) {};
		\node [style=none] (75) at (4, 2.75) {};
		\node [style=none] (76) at (4, 0) {};
		\node [style=none] (78) at (2.75, 1.75) {};
		\node [style=none] (79) at (2.75, 2.75) {};
		\node [style=none] (80) at (2.75, 0) {};
		\node [style=none] (81) at (2.75, 1) {};
		\node [style=none] (82) at (11.75, 0) {};
		\node [style=none] (83) at (11.5, 0) {};
		\node [style=none] (84) at (11.75, 2.75) {};
		\node [style=none] (85) at (11.5, 2.75) {};
		\node [style=none] (87) at (12, 0) {};
		\node [style=none] (88) at (12, 2.75) {};
		\node [style=none] (89) at (-1.25, 0) {};
		\node [style=none] (90) at (-1.5, 0) {};
		\node [style=none] (91) at (-1.25, 2.75) {};
		\node [style=none] (92) at (-1.5, 2.75) {};
		\node [style=none] (93) at (-1, 0) {};
		\node [style=none] (94) at (-1, 2.75) {};
	\end{pgfonlayer}
	\begin{pgfonlayer}{edgelayer}
		\draw (4.center) to (0.center);
		\draw (5.center) to (1.center);
		\draw (6.center) to (2.center);
		\draw (7.center) to (3.center);
		\draw (14.center) to (12.center);
		\draw (15.center) to (13.center);
		\draw [bend left=90, looseness=1.25] (6.center) to (14.center);
		\draw [bend left=90, looseness=1.25] (7.center) to (15.center);
		\draw (20.center) to (16.center);
		\draw (21.center) to (17.center);
		\draw (22.center) to (18.center);
		\draw (23.center) to (19.center);
		\draw (32.center) to (28.center);
		\draw (33.center) to (29.center);
		\draw (34.center) to (30.center);
		\draw (35.center) to (31.center);
		\draw (44.center) to (40.center);
		\draw (45.center) to (41.center);
		\draw (46.center) to (42.center);
		\draw (47.center) to (43.center);
		\draw [bend right=90, looseness=0.75] (13.center) to (40.center);
		\draw [bend right=90, looseness=0.75] (12.center) to (41.center);
		\draw (54.center) to (52.center);
		\draw (55.center) to (53.center);
		\draw [bend right=90, looseness=0.75] (31.center) to (53.center);
		\draw [bend right=90, looseness=0.75] (30.center) to (52.center);
		\draw [bend left=90] (55.center) to (21.center);
		\draw [bend left=90] (54.center) to (20.center);
		\draw (59.center) to (58.center);
		\draw (61.center) to (60.center);
		\draw (63.center) to (62.center);
		\draw (65.center) to (64.center);
		\draw (67.center) to (66.center);
		\draw (69.center) to (68.center);
		\draw (79.center) to (78.center);
		\draw (81.center) to (80.center);
		\draw [bend left=90, looseness=1.25] (79.center) to (75.center);
		\draw [bend right=90, looseness=0.75] (80.center) to (72.center);
		\draw [bend left=270] (67.center) to (71.center);
		\draw [bend right=270, looseness=0.75] (68.center) to (76.center);
		\draw (84.center) to (82.center);
		\draw (85.center) to (83.center);
		\draw (88.center) to (87.center);
		\draw [bend left=90] (63.center) to (85.center);
		\draw [bend left=270] (84.center) to (23.center);
		\draw [bend left=90] (22.center) to (88.center);
		\draw [bend right=90, looseness=0.50] (60.center) to (87.center);
		\draw [bend right=90, looseness=0.50] (28.center) to (82.center);
		\draw [bend right=90, looseness=0.50] (29.center) to (83.center);
		\draw (91.center) to (89.center);
		\draw (92.center) to (90.center);
		\draw (94.center) to (93.center);
		\draw [bend left=270] (59.center) to (94.center);
		\draw [bend left=270] (4.center) to (91.center);
		\draw [bend left=270] (5.center) to (92.center);
		\draw [bend right=90, looseness=0.50] (90.center) to (64.center);
		\draw [bend right=90, looseness=0.50] (89.center) to (43.center);
		\draw [bend right=90, looseness=0.50] (93.center) to (42.center);
		\draw (71.center) to (72.center);
		\draw (75.center) to (76.center);
	\end{pgfonlayer}
\end{tikzpicture}}
\end{center}
Define $\mu_1\in S_{c+b}$ by $\mu_1(i) = i+b \text{ mod } c+b$
and $\mu_2\in S_{a+d}$ by $\mu_2(i) =  i+d\text{ mod }a+d$. 
The above diagram shows that 
$$\Tr(f\circ g) = \pair^{p,q}(Di_1\ot L^{(a+d)}_{\mu_2}Di_2L^{(c+b)}_{\mu_1}).$$
Since pre- and post-composing with $L^{(n)}_{\sigma}$ are invertible maps for every $n$ and every $\sigma\in S_n$, the radical does not change. By applying the character $\chi$ to both sides of the last equation we get the desired result. 
\end{proof}
\begin{proposition}
The collection $\Nn_{\chi}(-,-)$ is a tensor ideal in $\C_{univ}$ (see Definition \ref{def:tensorideal}).  
\end{proposition}
\begin{proof}
The first condition for a tensor ideal follows easily from the previous lemma, using the fact that $\Tr(f\circ g) = \Tr(g\circ f)$. The second condition follows from Lemme 6.2.1 and also from Lemme 7.1.1 of \cite{Andrekahn}.
\end{proof}
If $(A,(y_i))$ is a structure of type $((p_i,q_i))$ in a $K$-good category $\D$, then by Proposition \ref{prop:universalproperty} we get a functor $F_A:\C_{univ}\to \D$. We also get the character of invariants $\chi= \chi_A$. We claim the following:
\begin{lemma}\label{lem:cofinite}
If $B,C$ are two objects of $\C_{univ}$ and $f\in\Hom_{\C_{univ}}(B,C)$ satisfies $F_A(f)=0$ then $f\in \Nn_{\chi}(B,C)$ 
\end{lemma}
\begin{proof}
Since $F_A$ is a symmetric monoidal functor it commutes with taking traces.
(see Subsection \ref{subsec:algstr}). Therefore, for every $g:B\to A$ we have that $F_A(\Tr(f\circ g)) = \chi(\Tr(f\circ g)) = \Tr(F_A(f\circ g)) = \Tr(F_A(f)\circ F_A(g)) = 0$ because $F_A(f)=0$. We used here the fact that $\chi$ is the character of invariants of $(A,(y_i))$. By Lemma \ref{lem:pairideal} we get the result.
\end{proof}
The following definition appeared in the statement of Theorem \ref{thm:main1}.
\begin{definition}\label{def:good}
The character $\chi:\winf\to K$ is called good if the following two conditions hold:
\begin{enumerate}
\item For every $p,q\in\N$ the subspace $N_{\chi}^{p,q}\subseteq Con^{p,q}$ has finite codimension.
\item If $B\in \C_{univ}$ and $T:B\to B$ satisfies that $T^r$ is $\chi$-negligible for some $r>0$ then $\chi(\Tr(T))=0$. 
\end{enumerate}
\end{definition}
We can now prove the following:
\begin{proposition}\label{prop:halfmainthm1}
If $\chi$ is the character of invariants of an algebraic structure $(A,(y_i))$ in a $K$-good category $\D$ then it is a good character.
\end{proposition}
\begin{proof}
We have seen in Lemma \ref{lem:cofinite} that for every $B,C\in \C_{univ}$ it holds that the kernel $K_{B,C}$ of $\Hom_{\C_{univ}}(B,C)\to \Hom_{\D}(F_A(B),F_A(C))$ is contained $\Nn_{\chi}(B,C)$. Since the hom-spaces in $\D$ are finite dimensional, $K_{B,C}$ is cofinite in $\Hom_{\C_{univ}}(B,C)$. Thus, $\Nn_{\chi}(B,C)$ is cofinite in $\Hom_{\C_{univ}}(B,C)$, and the first condition is satisfied. 

For the second condition, let $T:B\to B$ be an endomorphism in $\C_{univ}$. Assume that $T^r\in \Nn_{\chi}(B,B)$ for some $r>0$. 
This means that for $n\geq r$ $\Tr(F_A(T)^n)=0$, and thus $F_A(T)$ is nilpotent in $\End_{\D}(F_A(B))$. By Corollary \ref{cor:endshape} we see that $\Tr(F_A(T)) = \chi(\Tr(T))=0$ as required. 
\end{proof}
\begin{definition}
Let $\chi:\winf\to K$ be any character. We define 
$\wt{\C}_{\chi}:=\C_{univ}/\Nn_{\chi}$.  and we define $\C_{\chi}$ to be the Karoubian envelope of $\wt{\C}_{\chi}$. 
\end{definition}
Thus, $\wt{\C}_{\chi}$ has the same objects as $\C_{univ}$, but the hom-spaces are given by $\Hom_{\wt{\C}_{\chi}}(A,B) = \Hom_{\C_{univ}}(A,B)/\Nn_{\chi}(A,B)$. The objects of $\C_{\chi}$ can be thought of as pairs $(A,p)$ where $A\in \wt{\C}_{\chi}$ and $p:A\to A$ is an idempotent in $\wt{\C}_{\chi}$. We think of $(A,p)$ as the object $\Im(p)$.
The hom-spaces in $\C_{\chi}$ are given by 
$$\Hom_{\C_{\chi}}((A,p),(B,q)) = q\Hom_{\wt{\C}_{\chi}}(A,B)p,$$ where we use the action of $\End_{\wt{\C}_{\chi}}(A)$ from the right and the action of $\End_{\wt{\C}_{\chi}}(B)$ from the left on $\Hom_{\wt{\C}_{\chi}}(A,B)$.  
We can thus think of the category $\wt{\C}_{\chi}$ as being contained in $\C_{\chi}$. In particular, since $\wt{\C}_{\chi}$ is formed from $\C_{univ}$ by dividing out a tensor ideal, we get a symmetric monoidal functor $F_{\chi}:\C_{univ}\to \wt{\C}_{\chi}\to \C_{\chi}$. 
To avoid cumbersome notation, we will also write $F_{\chi}$ for the functor $\C_{univ}\to \wt{\C}_{\chi}$.
Notice that for the tensor unit $\one$ we get 
$\End_{\C_{\chi}}(\one) = \End_{\C_{univ}}(\one)/\Nn_{\chi}(\one,\one) = \winf/\Ker(\chi)\cong K,$ and that the map $\End_{\C_{univ}}(\one)\to \End_{\C_{\chi}}(\one)$ is just the map $\chi:\winf\to K$. 

We are now ready to prove Theorem \ref{thm:main1}.
The implication $1\Rightarrow 2$ is clear, and $2\Rightarrow 3$ is Proposition \ref{prop:halfmainthm1}. The implication $3\Rightarrow 1$ follows from the next proposition. 
\begin{proposition}\label{prop:2halfmainthm1}
Assume that $\chi$ is a good character. Then the category $\C_{\chi}$ is a semisimple $K$-good category. Denote by $(W,(x_i))$ the image of the tautological structure of $\C_{univ}$ under the functor $F_{\chi}$. Then $\chi$ is the character of invariants of $(W,(x_i))$. In particular, every good character arises as the character of invariants of some algebraic structure in some $K$-good category.
\end{proposition}
\begin{proof}
Notice first that by dividing out the ideal $\Nn_{\chi}$ we get finite dimensional hom-spaces in $\wt{\C}_{\chi}$ and in $\C_{\chi}$, because $\chi$ is a good character. Moreover, as was stated before the proposition, $\End_{\C_{\chi}}(\one)=K$. The category $\C_{\chi}$ is also rigid. Indeed, $\C_{univ}$ is rigid and therefore $\wt{\C}_{\chi}$ is rigid as every object in $\wt{\C}_{\chi}$ is the image of some object in $\C_{univ}$ under the quotient functor. Since taking duals commutes with projections, $\C_{\chi}$ is rigid as well. 
Thus, we just need to prove that $\C_{\chi}$ is semisimple. 

For this, we begin by proving that all endomorphism algebras in $\wt{\C}_{\chi}$ are semisimple. Let $A\in \wt{\C}_{\chi}$ and let $R = \End_{\wt{\C}_{\chi}}(A)$. The trace pairing $r_1\ot r_2\mapsto \Tr(r_1r_2)$ defines a non-degenerate pairing on $R$, since we showed in Lemma \ref{lem:pairideal} that $\Nn_{\chi}(A,A)$ is the radical of the trace pairing. Let $J$ be the Jacobson radical of $R$. If $r\in J$ then it holds that $rr'\in J$ for every $r'\in R$. In particular, $rr'$ is nilpotent for every $r'\in R$ because $R$ is finite dimensional. But if $rr'=L$ is nilpotent, we can lift it to an endomorphism $T$ in $\C_{univ}$. The fact that $L$ is nilpotent means that some positive power of $T$ is contained in the ideal $\Nn_{\chi}$. By the second condition in Definition \ref{def:good} we know that this means that $\chi(\Tr(T))=\Tr(L)=0$. but this implies that $\Tr(rr')=0$ for all $r'\in R$.
Since $\Tr$ is a non-degenerate pairing on $R$ this implies that $r=0$, so $J=0$ and $R$ is semisimple.

The objects of $\C_{\chi}$ are of the form $(A,p)$ where $p\in\End_{\wt{\C}_{\chi}}(A)$ is an idempotent. We shall think of the object $(A,p)$ as $\Im(p)$. We then have $\End_{\C_{\chi}}((A,p)) = pRp$ where $R = \End_{\wt{\C}_{\chi}}(A)$. Since $R$ is semisimple, it follows that $pRp$ is semisimple as well. We thus see that all endomorphism algebras in $\C_{\chi}$ are semisimple. By taking a Wedderburn decomposition of the endomorphism algebra, and taking a complete orthogonal set of primitive idempotents, we see that every object $A$ in $\C_{\chi}$ decomposes as the direct sum $A=\oplus_i A_i$ where $\End_{\C_{\chi}}(A_i)=K$. 

We claim that all objects $B$ which satisfy $\End_{\C_{\chi}}(B)=K$ are simple. For this it will be enough to prove that if $B_1$ and $B_2$ are two such objects then either they are isomorphic, or $\Hom_{\C_{\chi}}(B_1,B_2)=0$. To do so we consider the object $A=B_1\oplus B_2$. Assume that $\Hom_{\C_{\chi}}(B_1,B_2)\neq 0$. Then $$\End_{\C_{\chi}}(A) = \Hom_{\C_{\chi}}(B_1,B_2)\oplus \Hom_{\C_{\chi}}(B_2,B_1)\oplus \End_{\C_{\chi}}(B_1)\oplus \End_{\C_{\chi}}(B_2).$$ The non-degeneracy of the trace pairing on $\End_{\C_{\chi}}(A)$ implies that if there is $0\neq f:B_1\to B_2$ then there must be a non-zero morphism $g:B_2\to B_1$ such that $\Tr(f\circ g)\neq 0$. In particular, $f\circ g\neq 0$. Since $\End_{\C_{\chi}}(B_2)=K$ this implies that $f\circ g$ is invertible. By rescaling we can assume that $f\circ g=\Id_{B_2}$. By a symmetric argument we can show that $g\circ f = \Id_{B_1}$ and we are done.
\end{proof}
\begin{remark}\label{rem:reftoexample}
Unlike $\C_{univ}$, the category $\C_{\chi}$ does not satisfy a universal property for structures $(A,(y_i))$ in some $K$-good category $\D$ which afford $\chi$ as character of invariants. The reason for this is that it is possible that $(A,(y_i))$ will afford the character $\chi$ but the resulting functor $F_A:\C_{univ}\to \D$ will not vanish on all morphisms in $\Nn_{\chi}$. We will see an example of this phenomenon in Subsection \ref{subsec:example-idempotent}
\end{remark}
The construction of the category $\C_{\chi}$ enables us to give an alternative definition for good characters.
\begin{lemma}
Let $\chi:\winf\to K$ be a character that satisfies the first condition of Definition \ref{def:good}. Then $\chi$ satisfies the second condition of Definition \ref{def:good}, and is therefore a good character, if and only if the following condition holds:
for every $T\in \C_{univ}$ it holds that $\sum_n\chi(\Tr(T^n))X^n\in K[[X]]$ is a good rational function 
\end{lemma}
\begin{proof}
If $\chi$ is good then $\C_{\chi}$ is a $K$-good category, and by Corollary \ref{cor:endshape} $\sum_n\chi(\Tr(T^n))X^n$ is a good rational function. In the other direction, if $\chi$ satisfies the condition in the lemma, then in particular if $T^r$ is $\chi$-negligible for some $r>0$ it holds that $\sum_n \chi(\Tr(T^n))X^n$ is a polynomial. But the only polynomials of the form $\frac{P(X)}{Q(X)}$ with $\deg(P)\leq \deg(Q)$ are the constant polynomials. In particular, we get that $\chi(\Tr(T))=0$, and $\chi$ is therefore a good character.
\end{proof}
\begin{remark}
The above lemma gives an alternative definition of good characters. The advantage of this definition is that it refers to all morphisms in $\C_{univ}$, and not only the $\chi$-negligible ones. 
\end{remark}

\section{Structures in $\Vec_K$ and a proof of Theorems \ref{thm:main2} and \ref{thm:mainallG}}\label{sec:vecstr}
In this section we will describe the categories $\C_{\chi}$ explicitly, in case $\chi$ is a character arising from a structure in $\Vec_K$. Recall from Subsection \ref{subsec:GIT} that structures in $\Vec_K$ of dimension $d$ are in one-to-one correspondence with $\GL_d(K)$-orbits in the variety $U_d$. If $((Y,(y_i))$ and $(Z,(z_i))$ are two structures of dimension $d$, we will say that $(Y,(y_i))$ \emph{specializes} to $(Z,(z_i))$ if the $\GL_d(K)$-orbit of the isomorphism class of $(Z,(z_i))$ is contained in the closure of the $\GL_d(K)$-orbit of the isomorphism class of $(Y,(y_i))$. If this happens, the characters of invariants of $(Y,(y_i))$ and of $(Z,(z_i))$ are equal. We claim the following:
\begin{lemma}\label{lem:closure}
Assume that $(Y,(y_i))$ and $(Z,(z_i))$ are two structures of dimension $d$, and that $(Y,(y_i))$ specializes to $(Z,(z_i))$. Denote by $F_Y,F_Z:\C_{univ}\to \Vec_K$ the functors constructed in Section \ref{sec:unicat}. If $f:A\to B$ is a morphism in $\C_{univ}$ and $F_Y(f)=0$, then $F_Z(f)=0$ as well.
\end{lemma}
\begin{proof}
By fixing a basis for $Y$ and for $Z$ we can assume without loss of generality that $Y=Z=K^d$ as vector spaces. In this case we can think of the variety $U_d$ as the affine space made by the structure constants of the different structure tensors (See Subsection 2.3. in \cite{meirUIR}). The vanishing of $F_Y(f)$ then boils down to the vanishing of a set of polynomial functions in the structure constants. Since these polynomials vanish on all point in the orbit of $(Y,(y_i))$, they must also vanish on the orbit of $(Z,(z_i))$ by continuity, so $F_Z(f)=0$ as well.
\end{proof}

Let now $(Y,(y_i))$ be an algebraic structure of dimension $d$ in $\Vec_K$. Let $\chi=\chi_{(Y,(y_i))}$. The isomorphism class of $(Y,(y_i))$ gives a $\GL_d(K)$-orbit in $U_d$.
Let $\Ow$ be the unique closed orbit in the closure of this orbit (Uniqueness follows from the results in Subsection \ref{subsec:GIT}). We write $(Z,(z_i))$ for a representative of the orbit $\Ow$. 
We claim the following:
\begin{proposition}\label{prop:mainvecstr}
There is a unique symmetric fiber functor $F:\C_{\chi}\to \Vec_K$. The image of the tautological structure in $\C_{\chi}$ under $F$ is isomorphic to $(Z,(z_i))$, and $\C_{\chi}\cong \Rep(\Aut(Z,(z_i)))$. 
\end{proposition}
\begin{proof}
To prove the existence and uniqueness of the symmetric fiber functor $F$ we will use the theory of Deligne on Tannakian categories (see Th{\'e}or{\`e}me 7.1. in \cite{Deligne2} and also Proposition 0.5. in \cite{Deligne1} for the more general case of an $\text{sVec}_K$ valued functor). Applied to the present situation, the theorem of Deligne tells us that there is a (unique) symmetric monoidal functor $\C_{\chi}\to \Vec_K$ if and only if for every $B\in \C_{\chi}$ there is an integer $r$ such that $\bigwedge^r B=0$. This condition is equivalent to the condition that the constructible idempotent $\Alt^B_r:B^{\ot r}\to B^{\ot r}$ defined by $$\Alt^B_r = \frac{1}{r!}\sum_{\sigma\in S_r}(-1)^{\sigma}L^{(r)}_{\sigma}$$ vanishes in $\C_{\chi}$. It will be enough to prove this statement for objects of the form $W^{a,b}$ because the collection of objects which satisfies this condition is closed under taking direct sums and direct summands. So we need to prove that for every $a$ and every $b$ there is an $r$ such that $\Alt^{W^{a,b}}_r\in \Nn_{\chi}((W^{a,b})^{\ot r},(W^{a,b})^{\ot r})$.  

Consider now the functor $F_Y:\C_{univ}\to \Vec_K$. We do know that $F_Y(W^{a,b}) = Y^{a,b}$ is a finite dimensional vector space. This means that for $r= \dim_K Y^{a,b}+1$ it holds that $$0=\bigwedge^r Y^{a,b} = \bigwedge^r F_Y(W^{a,b}) = F_Y(\bigwedge^r W^{a,b}),$$ where we use the fact that $F_Y$ is monoidal and symmetric, so it commutes with taking tensor products and taking exterior powers. 
But this means that $F_Y(\Alt^{W^{a,b}}_r)=0$. We have seen in \ref{lem:cofinite} that this implies that $\Alt^{W^{a,b}}_r\in \Nn_{\chi}((W^{a,b})^{\ot r},(W^{a,b})^{\ot r})$, so the first statement is proved.

We thus have a symmetric fiber functor $F:\C_{\chi}\to \Vec_K$. We write $F(W) = D$ and $F(x_i) = t_i$. We get a structure $(D,(t_i))$ in $\Vec_K$. The character of invariants of $(D,(t_i))$ is still $\chi$, and $(D,(t_i))$ specializes to $(Z,(z_i))$. Following Section 6 of \cite{meirDescent} we get an equivalence between $\C_{\chi}$ and $\Rep(\Aut(D,(t_i)))$ (notice that the framework here is a bit different than that of \cite{meirDescent}, because we work here over an algebraically closed field. This is why we get here a proper equivalence of the categories, and not just up to a form). We need to prove that $(D,(t_i))$ is isomorphic to $(Z,(z_i))$. 

For this, let $F_Z:\C_{univ}\to \Vec_K$ be the functor sending the tautological structure in $\C_{univ}$ to $(Z,(z_i))$. We will show that this functor splits via $\C_{\chi}$. This will imply that we get a fiber functor $\C_{\chi}\to \Vec_K$ sending the tautological structure in $\C_{\chi}$ to $(Z,(z_i))$. By the uniqueness of the fiber functor this will prove that $(Z,(z_i))\cong (D,(t_i))$. 

To prove the above statement, it will be enough to show that for every $\chi$-negligible morphism $f$ we have $F_Z(f)=0$. But since $(D,(t_i))$ spcializes to $(Z,(z_i))$ and $F_D(f)=0$, this follows from Lemma \ref{lem:closure}.
\end{proof}
Write now $X^{(p,q)}\subseteq Y^{p,q}$ for the image of $\Hom_{\C_{univ}}(W^{\ot q},W^{\ot p})$ under the functor $F_Y$. In other words: $X^{(p,q)}$ contains all the linear transformations one can form from the structure tensors of $(Y,(y_i))$ using linear algebra operations. We have the following:
\begin{corollary}\label{cor:closedorbit}
The restriction of the natural pairing $Y^{p,q}\ot Y^{q,p}\to K$ to $X^{(p,q)}\ot X^{(q,p)}\to K$ is non-degenerate if and only if the orbit of $(Y,(y_i))$ is closed in $U_d$. 
\end{corollary}
\begin{proof}
We have already seen that if the orbit of $(Y,(y_i))$ is closed then we have an equivalence $\C_{\chi}\cong \Rep(Aut(Y,(y_i)))$. By the way $\C_{\chi}$ is constructed we see that the pairing $\Hom_{\C_{\chi}}(W^{\ot q},W^{\ot p})\ot \Hom_{\C_{\chi}}(W^{\ot p},W^{\ot q})\to K$ given by $T_1\ot T_2\mapsto \Tr(T_1T_2)$ is non-degenerate. Since we have a fiber functor $F:\C_{\chi}\to \Vec_K$ which sends $(W,(x_i))$ to $(Y,(y_i))$ the result follows.

In the other direction, assume that the pairing $X^{(p,q)}\ot X^{(q,p)}\to K$ is non-degenerate. This means that if $f$ is a $\chi$-negligible morphism in $\C_{univ}$ then $F_Y(f)=0$. This implies that $F_Y$ splits via $\C_{\chi}$, and we have already seen that this implies that $(Y,(y_i))\cong (Z,(z_i))$, where $(Z,(z_i))$ is the closed orbit which $(Y,(y_i))$ specializes to. 
\end{proof}

A special case of the next corollary was used in \cite{meirHopfI} to prove that every finite dimensional semisimple Hopf algebra admits at most finitely many Hopf orders over any number ring:
\begin{corollary}
Assume that $(Y,(y_i))$ has a closed orbit. Let $X^{(p,q)}\subseteq Y^{p,q}$ be the subspace of constructible elements. Then $X^{(p,q)} = (Y^{p,q})^G$ where $G = \Aut(Y,(y_i))$. 
\end{corollary}
\begin{proof}
This follows immediately because in this case $\C_{\chi}\cong \Rep(G)$ and the hom-spaces in $\C_{\chi}$ are spanned by the constructible elements.
\end{proof}
\begin{remark}
A simple example of a non-closed orbit is given as follows: Let $Y=K^2$ and consider a single tensor $T$ of type $(1,1)$ (i.e. an endomorphism). The orbit of the nilpotent linear transformation $e_{12}$ is not closed, and contains the zero endomorphism in its closure. The algebra of invariants $\winf$ is a polynomial algebra on $\Tr(T^n)$ $n=0,1,2,\ldots$, and the character of invariants of $(K^2,e_{12})$ and $(K^2,0)$ is given by $\chi(\Tr(T^n)) = 0$ for $n>0$ and $\chi(\dim) = 2$. Even though $e_{12}\neq 0$, $T$ will be a negligible morphism and will therefore vanish in $\C_{\chi}$. The category that we will get is $\Rep(\GL_2)$, as $\GL_2$ is the automorphism group of $(K^2,0)$. 
\end{remark}

\subsection{Proof of Theorem \ref{thm:mainallG}}
Let now $G$ be any reductive affine algebraic group. 
We show that $\Rep(G)\cong \cchi$ for a suitable type of structure and a suitable character $\chi$. 
This would have been fairly easy to prove if we allowed infinitely many structure tensors. We will show here that it is also possible when considering only finitely many structure tensors.
Since $G$ is an affine algebraic group, we know that $G$ is a subgroup of $\GL(V)$ for some finite dimensional vector space $V$. We begin with the following lemma:
\begin{lemma} Assume that $G\subseteq H \subseteq \GL(V)$ are algebraic groups, and that $G$ is also reductive. If $\Hom_H(\one,V^{p,q})= \Hom_G(\one,V^{p,q})$ for every $p,q\in\N$, then $G=H$. 
\end{lemma}
\begin{proof}
Consider the category $\Rep(H)$. Inside this category, consider an object of the form $A=A_{((a_i,b_i))}=\bigoplus_i V^{a_i,b_i}$. By assumption, it holds that $\End_H(A)\to \End_G(A)$ is an isomorphism. Since $G$ is reductive, it holds that $\End_G(A)=\End_H(A)$ is a semisimple algebra, and thus $A$ decomposes in $\Rep(H)$ into a direct sum of objects $B$ with $\End_H(B)=K$. Moreover, for any two such objects $B$ and $B'$ it holds that $\Hom_H(B,B')$ is either zero or one dimensional, and if $B$ is a direct summand of $A$ and $B'$ a direct summand of $A'$ then $B\ot B'$ is a direct summand of $A\ot A'$ which is again a direct sum of objects of the form $V^{a,b}$. 

Consider now the full subcategory $\C\subseteq \Rep(H)$ whose objects are direct sums of direct summands of objects of the form $A_{((a_i,b_i))}$ for some $((a_i,b_i))$. The above argument shows that this is a rigid tensor subcategory of $\Rep(H)$ that contains $V$ and $V^*$. Since $V$ and $V^*$ tensor-generate $\Rep(H)$, it holds that $\C= \Rep(H)$. 
It follows that the restriction functor $\Rep(H)\to \Rep(G)$ is an isomorphism on all hom-spaces, and that it is surjective on objects, since $\Rep(G)$ is semisimple and every object there is isomorphic to a direct summand of $\bigoplus V^{a_i,b_i}$ for some $((a_i,b_i))$. The restriction functor is therefore an equivalence of categories. By Tannaka reconstruction it follows that $H=G$.
\end{proof}

Next, we claim the following:
\begin{lemma} There is a finite collection $x_i\in V^{p_i,q_i}$, $i=1,\ldots, c$ of tensors such that $G=\Stab_{\GL(V)}((x_i))$. 
\end{lemma}
\begin{proof}
Write $T$ for the set of all tensors in $V^{p,q}$ for some $p,q$ that are fixed by $G$. Write $Q$ for the set of subgroups of $\GL(V)$ that arise as stabilizers of finite subsets of $T$. 
Since $\GL(V)$ is a Noetherian topological space, the set $Q$ has a minimal element $H$. 
Since taking unions of finite sets corresponds to taking intersections in $Q$, we see that this minimal element of $Q$ is in fact unique. 
We claim that $H=G$. This follows from the fact that by minimality, $H$ is contained in the stabilizer of any finite subset of $T$. This implies that $H$ fixes all the elements in $T$. By the previous lemma, this implies that $H=G$.
\end{proof}

We fix now a tuple $(x_1,\ldots, x_c)\in U(V) = \bigoplus V^{p_i,q_i}$ such that $\Stab_{\GL(V)}((x_1,\ldots, x_c)) = G$. 
We are almost in position to construct a $\chi$ such that $\C_{\chi}=\Rep(G)$. The only problem is that it might happen that the $\GL(V)$ orbit of $(x_1,\ldots, x_c)$ inside $\bigoplus V^{p_i,q_i}$ is not closed. We solve this issue using localisation:
\begin{lemma}
Let $(x_1,\ldots, x_c)$ be a tuple as above. 
Then there are finitely many tensors $z_i\in V^{a_i,b_i}$, $i=1,\ldots, d$ such that the $\GL(V)$-orbit of $(x_1,\ldots, x_c,z_1,\ldots, z_d)$ in $U(V)\oplus \bigoplus_i V^{a_i,b_i}$ is closed, and such that $\Stab_{\GL(V)}((x_1,\ldots, x_c,z_1,\ldots, z_d))=G$. 
\end{lemma}
\begin{proof}
Recall first that if $\GL(V)$ acts on an affine space $U$, and $u\in U$ is any point, then $\dim(\Ow_u) + \dim(\Stab_{\GL(V)}(u)) = \dim(\GL(V))$, where $\Ow_u$ is the orbit of $u$. In particular, if $\Ow_1$ and $\Ow_2$ are two orbits such that $\Ow_2\subseteq \ol{\Ow_1}$, then $\dim(\Ow_2)<\dim(\Ow_1)$ and as a result $\dim(\Stab_{\GL(V)}(u_1))<\dim(\Stab_{\GL(V)}(u_2))$ where $u_1\in \Ow_1$ and $u_2\in \Ow_2$. 

For our concrete case, we can view $p=(x_1,\ldots, x_c)$ as a point in the affine space $U(V)$ upon which $\GL(V)$ acts. If the orbit $\Ow$ of this point is closed, we are done. Otherwise, consider the unique closed $\GL(V)$-orbit in the closure of the $\GL(V)$-orbit of $p$. Let $p'=(x'_1,\ldots, x'_c)$ be a point in this closed orbit. Write $L=\Stab_{\GL(V)}(p')$. Then it holds that $\dim(G)<\dim(L)$. If $\chi$ is the character of the structure defined by $(x_1,\ldots, x_c)$, then we know that there is a constructible morphism $f\in Con^{b,a}$ for some $a,b\in\N$ such that $f$ is $\chi$-negligible but $0\neq \Re^{b,a}(f)\in V^{b,a}$. 
Denote by $\Re'^{b,a}:Con^{b,a}\to V^{b,a}$ the realization map with respect to the tensors $(x'_1,\ldots, x'_c)$. The fact that $f$ is $\chi$-negligible implies that $\Re'^{b,a}(f)=0$. Since $G$ is reductive, there is an element $z\in (V^{a,b})^G$ such that $\langle z,\Re^{b,a}(f)\rangle=1$. Consider now the point $q=(x_1,\ldots, x_c,z)\in U(V)\oplus V^{a,b}$. We have a $\GL(V)$-action on this space as well, and projection gives a $\GL(V)$-equivariant map $\Phi:U'(V)\to U(V)$. If the closure of the orbit of $q$ contains a point $q'$, then the closure of the orbit of $\Phi(q)=p$ contains the orbit of $\Phi(q')$. However, the orbit of $(x'_1,\ldots, x'_c)$ is not in the image of $\Phi$. Indeed, if it was the image of $(x'_1,\ldots, x'_c,z')$ then we would have had $\langle z',\Re'^{b,a}(f)\rangle = \langle z,\Re^{b,a}(f)\rangle=1$ by continuity. But this contradicts the fact that $\Re'^{b,a}(f)=0$. 

We claim that if $q'\in U'(V)$ is in the closure of the orbit of $(x_1,\ldots, x_c,z)$ then $\dim(\Stab_{\GL(V)}(q'))<\dim(L)$. This follows from the fact that $\dim(\Stab_{\GL(V)}(q')\leq \dim(\Stab_{\GL(V)}(\Phi(q')))$. The orbit of $\Phi(q')$ is in the closure of the orbit of $p$, and the closure of the orbit of $\Phi(q')$ contains the orbit of $p'$, where we use the fact that the closure of any orbit contains a unique closed orbit. This implies that $\dim(\Stab_{\GL(V)}(q'))<\dim(L)$. 

We thus see that by passing from $(U(V),p)$ to $(U'(V),q)$ we strictly decreased the maximal dimension of the stabilizer of a point in the closure of the orbit of $p$, respectively of $q$. It also holds that $\Stab_{\GL(V)}(p)=\Stab_{\GL(V)}(q)=G$. By repeating this process finitely many times, and adding finitely many tensors, we will reach a space $U''(V)$ and a point $p''=(x_1,\ldots, x_c,z_1,\ldots, z_d)\in U''(V)$ such that $\Stab_{\GL(V)}(p'')=G$ and such that the orbit of $p''$ is closed. This will give us the desired algebraic structure and character of invariants. 
\end{proof}

\section{The good characters form a $K$-algebra}\label{sec:goodalgebra}
The ring $\winf$ is a polynomial algebra on the set $P$ of closed connected diagrams. 
This means that we have a one-to-one correspondence between characters $\winf\to K$ and functions from $P$ to $K$, given by restriction.
The set of functions $K^P = \{f:P\to K\}$ carries an additional structure of a $K$-algebra. To state this precisely, for every closed connected diagram $Di\in P$, every $\chi_1,\chi_2\in K^P$, and every $t\in K$, we have $$(\chi_1+\chi_2)(Di) = \chi_1(Di) + \chi_2(Di), (\chi_1\cdot \chi_2)(Di) = \chi_1(Di) \cdot \chi_2(Di), $$ $$(t\chi)(Di) = t(\chi(Di)).$$ 
In Section 6 of \cite{meirUIR} we introduced two coproducts on $\winf$, $\Delta$ and $\Delot$. The elements of $P$ are primitive with respect to $\Delta$ and are group-like with respect to $\Delot$. 
A direct calculation shows that as characters of $\winf$ we have $\chi_1\cdot\chi_2 = (\chi_1\ot \chi_2)\Delot$ and $\chi_1+\chi_2 = (\chi_1\ot \chi_2)\Delta$. 
We claim the following:
\begin{lemma}
The set of good characters in $K^P$ is closed under addition and multiplication.
\end{lemma}
\begin{proof}
Assume that $\chi_1$ and $\chi_2$ are good characters. Then the categories $\C_{\chi_1}$ and $\C_{\chi_2}$ are semisimple $K$-good categories. 
We can consider the Deligne product $\C_{\chi_1}\boxtimes \C_{\chi_2}$, which is again a semisimple $K$-good category. 
This category contains the algebraic structures $W_1:= W\boxtimes\one$ and $W_2:= \one\boxtimes W$. We will write $(W_1,(x_i))$ and $(W_2,(y_i))$ to indicate the specific structure tensors of these structures. 
It is easy to see that the character of invariants of $W_i$ is $\chi_i$ for $i=1,2$. Since the two structures live in the same category we can consider the structures $(W_1\oplus W_2,(x_i\oplus y_i))$ and $(W_1\ot W_2,(x_i\ot y_i))$. 
Following Section 6 of \cite{meirUIR} we see that the characters of invariants of these structures are $\chi_1+\chi_2$ and $\chi_1\cdot\chi_2$ respectively. But since these are the character of invariants of structures in a $K$-good category we get that $\chi_1+\chi_2$ and $\chi_1\cdot \chi_2$ are good characters, as required.
\end{proof}
To prove that the good characters form a $K$-algebra we just need to show that the set of good characters is closed under the action of $t\in K$. To do so, we will use the category $\Rep(S_t)$, which will be described in detail in Section \ref{sec:examples}. The proof of the next proposition finishes the proof of Theorem \ref{thm:main3}
\begin{proposition}
Assume that $\chi:\winf\to K$ is a good character. Then $t\chi$ is a good character as well.
\end{proposition}
\begin{proof}
Since we already know that the set of good characters is closed under multiplication, it will be enough to show that the character $c_t$, given by $c_t(Di) = t$ for every closed connected diagram $Di$, is a good character. For this we will use the category $\Rep(S_t)$ described in Section \ref{sec:examples}.
The category $\Rep(S_t)$ can be described as $\C_{\chi}$, where $\chi$ is the character of invariants of a certain separable commutative Frobenius algerba. Such an algebra is given by the following list of tensors: $m$ of degree $(1,2)$, $\Delta$ of degree $(2,1)$, $\epsilon$ of degree $(0,1)$, and $u$ of degree $(1,0)$. We will see in Section \ref{sec:examples} that $\chi(Di)=t$ for every closed diagram formed by the boxes $m,\Delta,\epsilon,u,\Id_W$. 

We are interested here in structures of type $((p_i,q_i))$. We will define a structure of this type in $\C_{\chi}$. Let $W$ be the tautological structure in $\C_{\chi}$. We define $y_i:W^{\ot q_i}\to W^{\ot p_i}$ to be $(\Delta\ot \Id_W^{\ot p-2})\cdots \Delta m(m\ot \Id_W)\cdots (m\ot \Id_W^{\ot q-2})$ in case $p\geq 2$ and $q\geq 2$. If $p=1$ we just remove the $\Delta$-part from the above expression, and if $p=0$ we replace it by $\epsilon$. If $q=1$ we remove the $m$-part from the above expression, and if $q=0$ we replace it by $u$. 
We get in this way an algebraic structure of type $((p_i,q_i))$ in $\C_{\chi}$. Write $\psi$ for its character of invariants. We will show that $\psi=c_t$, thus proving the claim. 
If $Di$ is any closed connected diagram for structures of type $((p_i,q_i))$, we can create a closed diagram $Di'$ for structures of type $((1,2),(2,1),(0,1),(1,0))$ by replacing any box labeled by $x_i$ by the composition $y_i$ described above. It is easy to see that if $Di$ is connected, then $Di'$ is connected as well, since the morphisms are formed by connected diagrams (even if not closed ones). 
We then get that $\psi(Di) = \chi(Di')=t$ because $Di'$ is a closed connected diagram. But this means that $\psi=c_t$, so we are done.
\end{proof}

\section{Interpolations}\label{sec:interpolations}
Our goal in this section is to show that under some mild conditions a family of characters, and their symmetric monoidal categories, can be interpolated. 
We begin with the following definition:
\begin{definition}
A one-parameter family of characters $(\chi_t)_{t\in K}$ is a collection of characters $\chi_t:\winf\to K$, such that for every closed diagram $Di$ the element $\chi_t(Di)\in K$ is a polynomial in $t$. 
A family $(\chi_t)_{t\in K}$ is called additive if $\forall t_1,t_2\in K\, \chi_{t_1}+ \chi_{t_2} = \chi_{t_1+t_2}$, and it is called multiplicative if $\forall t_1,t_2\in K\, \chi_{t_1}\chi_{t_2} = \chi_{t_1t_2}$.
\end{definition}
\begin{remark}
The condition that $\chi_t(Di)$ is a polynomial in $t$ for every closed diagram is equivalent to the condition that $\chi_t(Di)$ is a polynomial in $t$ for every closed connected diagram, because every diagram can be written as a product of closed connected diagrams.
\end{remark}
\begin{definition}
We say that a property holds for almost all $t\in K$ if it does not hold only for finitely many values of $t$.
\end{definition}
\begin{lemma}
Let $(\chi_t)_{t\in K}$ be a one-paramter family of characters, and let $c\in Con^{p,q}$ be a constructible morphism. Assume that there is an infinite subset $\{s_1,s_2,\ldots\}\subseteq K$ such that $c\in \rad(\pair^{p,q}_{\chi_{s_i}})$. Then $c\in \rad(\pair^{p,q}_{\chi_t})$ for every $t\in K$. 
\end{lemma}
\begin{proof}
The assertion $c\in \rad(\pair^{p,q}_{\chi_t})$ is equivalent to $$\forall d\in Con^{q,p}\, \chi_t(\pair^{p,q}(c\ot d))=0.$$
But for every $d\in Con^{q,p}$, $\chi_t(\pair^{p,q}(c\ot d))$ is a polynomial in $t$. By assumption, this polynomial has infinitely many zeros, and is therefore zero. 
\end{proof}

We claim the following:
\begin{lemma}\label{lem:boundingdim}
Let $(\chi_t)_{t\in K}$ be a one-paramter family of characters. Assume that there is an infinite subset $\{s_1,s_2,\ldots\}\subseteq K$ such that the following condition holds: for every $(p,q)\in\N^2$ there is a number $n(p,q)$ such that for every $i$ the codimension of $\rad(\pair^{p,q}_{\chi_{s_i}})$ in $Con^{p,q}$ is $\leq n(p,q)$. Then the codimension of $\rad(\pair^{p,q}_{\chi_t})$ in $Con^{p,q}$ is $\leq n(p,q)$ for all $t$.
If moreover there are elements $c_1,\ldots, c_{n(p,q)}\in Con^{p,q}$ such that 
$c_1+\rad(\pair^{p,q}_{\chi_t}),\ldots,c_{n(p,q)}+\rad(\pair^{p,q}_{\chi_t})$ form a basis of $Con^{p,q}/\rad(\pair^{p,q}_{\chi_t})$ for some value of $t$ then they form a basis of $Con^{p,q}/\rad(\pair^{p,q}_{\chi_t})$ for almost all $t$. 
\end{lemma}
\begin{proof}
Let $(p,q)\in \N^2$. Let $\{w^{p,q}_1,w^{p,q}_2,\ldots\}$ be a basis for $Con^{p,q}$. Similarly, we have a basis $\{w_1^{q,p},w_2^{q,p},\ldots\}$ for $Con^{q,p}$. Then the pairing $\pair^{p,q}:Con^{p,q}\ot Con^{q,p}\to \winf$ is determined by its action on tensor products of basis elements. In particular, for every natural number $N$ and every $t\in K$ it holds that the codimension of $\rad(\pair^{p,q}_{\chi_t})$ in $Con^{p,q}$ is $\leq N$ if and only if for every $i_1,i_2,\ldots i_N,i_{N+1}\in \N$ the determinant of the matrix $$(\chi_t(\pair^{p,q}(w^{p,q}_{i_j}\ot w^{q,p}_{i_k})))_{j,k}$$ vanishes. Since $\chi_t$ is a one-parameter family, this determinant is a polynomial $p(t)$ in $t$. If $N=n(p,q)$ then we know that $p(s_i)=0$ for every $i=1,2,\ldots$. But a polynomial in one variable with infinitely many zeros is zero, so we deduce that the codimension of $\rad(\pair^{p,q}_{\chi_t})$ is $\leq n(p,q)$ for every $t\in K$. 

For the second assertion, assume that $c_1,\ldots, c_{n(p,q)}$ are elements of $Con^{p,q}$. Then they form a basis modulo $\rad(\pair^{p,q}_{\chi_t})$ if and only if the following condition is satisfied: there are elements $d_1,\ldots, d_{n(p,q)}$ in the basis of $Con^{q,p}$ such that the determinant of the matrix $(\chi_t(\pair^{p,q}(c_i\ot d_j)))$ is non-zero. This is a Zariski-open condition, since this determinant is also a polynomial in $t$. Thus, if there is a single value of $t$ for which there is a set of basis elements $d_1,\ldots, d_{n(p,q)}$ such that this determinant is non-zero, then it will be non-zero for almost all $t$. 
\end{proof}

\begin{lemma}\label{lem:ssnilpotent}
Let $(\chi_t)$ be a one-parameter family of characters, and let $A$ be an object of $\C_{univ}$. Assume that for almost all $t\in K$ the endomorphism algebra $R_t:=\End_{\wt{\C}_{\chi_t}}(F_{\chi_t}(A))$ is of dimension $N$ and that for some value of $t$ it has dimension $N$ and is semisimple. Then $R_t$ semisimple for almost all $t$, and for almost all $t$ it holds that if $r\in R_t$ is nilpotent then $\Tr(r)=0$. 
\end{lemma}
\begin{proof}
Let $s\in K$ be an element for which $R_s$ has dimension $N$ and is semisimple. Then $R_s$ has a basis given by $F_{\chi_s}(c_1),\ldots, F_{\chi_s}(c_N)$, for some morphisms $c_i$ in $\C_{univ}$. It follows from the previous lemma that $F_{\chi_t}(c_1),\ldots, F_{\chi_t}(c_N)$ form a basis for $R_t$ for almost all $t$. Let $\Tr_{reg}:R_t\to K$ be the trace of the regular representation of $R_t$. Then $R_t$ is semisimple if and only if the determinant of the matrix $$\Tr_{reg}(F_{\chi_t}(c_i)\circ F_{\chi_t}(c_j))$$ is non-zero (see Subsection \ref{subsec:findimalg}) Again, this is a Zariski-open condition. Since we know that this determinant does not vanish for $t=s$ it doesn't vanish for almost all $t$, and $R_t$ is thus semisimple for almost all $t$. 
Now, if $R_t$ is semisimple, then the fact that $\Tr_{reg}$ is non-degenerate implies that there is a unique $r_t\in R_t$ such that $\Tr(a) = \Tr_{reg}(ar_t)$ for every $a\in R_t$. This is because non-degeneracy of $\Tr_{reg}$ implies that every linear functional $R_t\to K$ is of the form $\Tr_{reg}(-x)$ for some $x\in R_t$.

We claim that $r_t$ is central. Indeed, it holds that $\Tr(ab) = \Tr(ba)$ for every $a,b\in R_t$. This implies that $\Tr_{reg}(abr_t) = \Tr_{reg}(bar_t)$. We use the linearity and cyclicity of $\Tr_{reg}$ to deduce that $\Tr_{reg}(a(br_t-r_tb))=0$ for every $a,b\in R_t$. Because $\Tr_{reg}$ is non-degenerate this implies that $r_tb-br_t=0$ for every $b\in R_t$, which means that $r_t$ is central in $R_t$. Now, if $a\in R_t$ is nilpotent then $\Tr(a) = \Tr_{reg}(ar_t)$. Since $r_t$ is central, $ar_t$ is nilpotent as well, and therefore $\Tr(a)=\Tr_{reg}(ar_t)=0$ as required.
\end{proof}

\begin{proposition}\label{prop:familymain}
Let $(\chi_t)$ be a one-parameter family of characters.
Assume that there is a countable subset $\{s_1,s_2,\ldots\}$ of values of $t$ such that $\chi_{s_i}$ is a good character for every $i$. Assume moreover that for every $(p,q)\in \N^2$ there is a number $n(p,q)$ such that 
$$\dim(Con^{p,q}/\rad(\pair^{p,q}_{\chi_{s_i}}))\leq n(p,q).$$
Then $\chi_t$ is not a good character for at most countably many values of $t$. If in addition $(\chi_t)$ is additive then $\chi_t$ is a good character for all $t$, and if it is multiplicative then $\chi_t$ is a good character for all $t\neq 0$. 
\end{proposition}
\begin{proof}
We use here the fact that a character is good if and only if it is afforded by a structure in a $K$-good category.
By Lemma \ref{lem:boundingdim} we see that the first condition of Definition \ref{def:good} holds for $\chi_t$ for every $t\in K$. 

For the second condition, we proceed as follows: 
for every object $A\in \C_{univ}$ we know that $R_t:= \End_{\wt{\C}_{\chi_t}}(F_{\chi_t}(A))$ is finite dimensional and of dimension $\leq M$ for some $M$ (this follows from the condition on the radicals and the fact that every object of $\C_{univ}$ has the form $\bigoplus_i W^{a_i,b_i}$. 
If $N = \sup_t\{\dim R_t\}$ then $\dim(R_t)=N$ holds for almost all $t$. This is because the condition $\dim(R_t)\geq N$ is a Zariski-open condition. In particular, for some $s_i$ it will hold that $\dim(R_{s_i})=N$. Since $\chi_{s_i}$ is a good character, it follows that $R_{s_i}$ is semisimple, by the proof of Proposition \ref{prop:2halfmainthm1}.
The condition of Lemma \ref{lem:ssnilpotent} are thus fulfilled, and we see that for almost all $t$ it holds that if $T:A\to A$ in $\C_{univ}$ satisfies that $T^n$ is $\chi_t$-negligible, then $F_{\chi_t}(T)\in R_t$ is nilpotent, and therefore $\Tr(F_{\chi_t}(T)) = \chi_t(\Tr(T))=0$. 
We thus see that for every object $A\in \C_{univ}$ there are at most finitely many values of $t$ for which the second condition of Definition \ref{def:good} is not satisfied. 
The category $\C_{univ}$ has countably many objects $A_1,A_2,\ldots$. Since the countable union of finite sets is countable, we get the first part of the proposition. 

For the second part, assume that $K$ is uncountable. This is not really a restriction, since we can always embed $K$ in an uncountable field $L$ and deduce back to $K$. Write $C = \{t| \chi_t \text{ is a good character}\}$. Then $D:= K\backslash C$ is at most countable. Assume that $(\chi_t)$ is additive and that $D\neq \varnothing$. Let $d\in D$. Consider the set $E = \{d-t| t\in C\}$. Then $E$ has a countable complement. It follows that $E\cap C\neq \varnothing$. So there is $t\in C$ such that $d-t$ is also in $C$. But since $C$ is closed under addition it follows that $d= d-t+t\in C$, a contradiction. The proof for the case where the family $(\chi_t)$ is multiplicative is similar, where we use multiplication instead of addition and $K\backslash \{0\}$ instead of $K$.
\end{proof}
\begin{definition} We call $\{s_1,s_2,\ldots\}$ a special collection for $(\chi_t)$ if it satisfies the condition of the proposition. 
\end{definition}
\section{Examples}\label{sec:examples}
\subsection{The empty structure, and the categories $\Rep(\GL_t(K))$}
Consider first the empty structure, where $r=0$ and there are no structure tensors. 
In this case the algebra of invariants $\winf$ is just $K[D]$, where $D$ is the dimension invariant. Define a one-parameter family of characters by the formula $\chi_t(D) = t$ for every $t\in K$. Notice that this is an additive family. 
We have the special collection $C=\{0,1,2,3,4,\ldots\}$. In this particular case it holds that $Con^{p,q}=0$ if $p\neq q$, while $Con^{p,p}$ has already finite rank $p!$ over $\winf=K[D]$, and it has a basis given by $L^{(p)}_{\sigma}$ for $\sigma\in S_p$. As a result, the condition for finite dimensionality of the hom-spaces in $\wt{\C}_{\chi_t}$ holds for all $t$, and $\chi_t$ is a good character for every $t$ by Proposition \ref{prop:familymain}. For $t=n$, a non-negative integer, we get the category $\Rep(\GL_n(K))$. We thus write $\C_{\chi_t} = \Rep(\GL_t(K))$ for every $t\in K$. The algebra of good characters in this case is simply $K$. This example was given by Deligne in \cite{Deligne3}.

\subsection{A single endomorphism}\label{subsec:example-idempotent}
Consider now the algebraic structure consisting of a vector space $W$ and a single endomorphism $T:W\to W$. Alternatively, we can think of such an algebraic structure as a $K[t]$-module, where $t$ acts by $T$. 
In this case the ring $\winf$ is $K[D,\Tr(T),\Tr(T^2),\ldots]$. Let $\chi:\winf\to K$ be a good character. 
By Corollary \ref{cor:endshape} we know that the rational function $\sum_i \chi(\Tr(T^i))X^i$ has the form $\sum_{j=1}^n \frac{t_j}{1-\la_jX}$ where $t_j\in K$ and $\{\la_j\}$ is the spectrum of $T$. 
The algebra $\End_{\C_{\chi}}(W)$ is generated by the endomorphism $T$. Since $\C_{\chi}$ is semisimple, this algebra is semisimple, and $W$ splits as $W=\bigoplus_{j=1}^n W_j$, where $W_j=\Ker(T-\la_j)$. It then follows easily that $t_j = \dim W_j$. Moreover, we get an equivalence of categories $$F:\C_{\chi}\equiv \Rep(\GL_{t_1}(K))\boxtimes\cdots\boxtimes\Rep(\GL_{t_n}(K)),$$
where $F(W_i) = \one\boxtimes\cdots\boxtimes W\boxtimes\one\boxtimes\cdots\boxtimes \one\in \Rep(\GL_{t_i}(K))$. The morphism $F(T)$ is then given by $F(T)|_{F(W_i)} = \la_i\Id_{F(W_i)}$. We thus see that any good rational function can be received here. For $\la\in K$ write $U_{\la}:\winf\to K$ for the character $\Tr(T^i)\mapsto \la^i$. A direct calculation shows that $U_{\la}U_{\mu} = U_{\la\mu}$. It follows that the algebra of good characters here is the monoid algebra of $(K,\times)$. 

We next use this type of algebraic structure to show that the category $\C_{\chi}$ does not necessarily satisfy a universal property, as was stated in Remark \ref{rem:reftoexample}. 
Take $A=K^2\in\Vec_K$, and take $T:A\to A$ to be the linear transformation represented by the matrix $\begin{pmatrix} 0 & 1 \\ 0 & 0 \end{pmatrix}$.
Then $\dim(A)=2$ and $\Tr(T^n)=0$ for all $n\geq 0$. Let $\chi$ be the character of invariants of $(A,T)$. Then $T\in \C_{univ}$ is a $\chi$-negligible morphism, and as a result $T=0$ in $\C_{\chi}$. This shows that there is no functor $F:\C_{\chi}\to \Vec_K$ such that $F(W,T)= (A,T)$, because $F$ is $K$-linear and therefore cannot send the zero vector to a non-zero vector.
 
\subsection{Non-degenerate symmetric pairings and the categories $\Rep(\text{O}_t(K))$}
For the next example, we would like to consider vector spaces $W$ with a non-degenerate symmetric pairing $c:W\ot W\to K$. The non-degeneracy is equivalent to the existence of $d\in W\ot W$ such that $ev(c\ot d)=\Id_W$, or in graphical terms: 
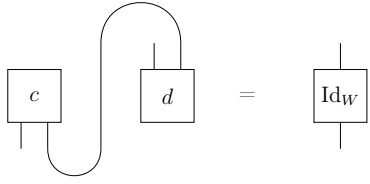
\begin{center}\scalebox{0.7}{
\begin{tikzpicture}
	\begin{pgfonlayer}{nodelayer}
		\node [style=2function] (0) at (-2.25, 2.75) {$c$};
		\node [style=2function] (1) at (0.25, 2.75) {$d$};
		\node [style=none] (2) at (-2, 2.25) {};
		\node [style=none] (3) at (-2.5, 2.25) {};
		\node [style=none] (4) at (-2, 1.75) {};
		\node [style=none] (5) at (-2.5, 1.75) {};
		\node [style=none] (6) at (0.5, 3.25) {};
		\node [style=none] (7) at (0, 3.25) {};
		\node [style=none] (8) at (0, 3.75) {};
		\node [style=none] (9) at (0.5, 3.75) {};
		\node [style=none] (10) at (-1, 3.75) {};
		\node [style=none] (11) at (-1, 1.75) {};
		\node [style=none] (12) at (1.75, 2.75) {=};
		\node [style=1function] (13) at (3.5, 2.75) {$\Id_W$};
		\node [style=none] (14) at (3.5, 3.25) {};
		\node [style=none] (15) at (3.5, 3.75) {};
		\node [style=none] (16) at (3.5, 2.25) {};
		\node [style=none] (17) at (3.5, 1.75) {};
	\end{pgfonlayer}
	\begin{pgfonlayer}{edgelayer}
		\draw (3.center) to (5.center);
		\draw (2.center) to (4.center);
		\draw (11.center) to (10.center);
		\draw (8.center) to (7.center);
		\draw (9.center) to (6.center);
		\draw [bend left=270, looseness=1.75] (9.center) to (10.center);
		\draw [bend left=90, looseness=1.75] (11.center) to (4.center);
		\draw (15.center) to (14.center);
		\draw (16.center) to (17.center);
	\end{pgfonlayer}
\end{tikzpicture}}
\end{center}
The fact that $c$ is symmetric implies that $d$ is symmetric as well, considered as a map $W^*\ot W^*\to K$. This can also be proved directly using diagrams. 
We thus consider the theory $\T$ which contains two axioms: the axioms $ev(c\ot d)-\Id_W$ and the axiom $cL^{(2)}_{(12)} - c$.
In this case, the resulting hom-spaces in $\C_{univ}^{\T}$ already have finite rank over $\winf/I_{\T}$. Indeed, whenever a diagram contains a connection between an input string of a $c$-box and an output string of a $d$-box we can reduce this to the identity morphism on $W$ (we use here the fact that both $c$ and $d$ are symmetric). Thus, every diagram is equivalent to a diagram in which no $c$- and $d$-boxes are connected. But there are only finitely many such diagrams with a given number of input and output strings. 

For the scalar invariants, notice that the only closed connected diagram in which no input strings of $c$ are connected to output strings of of $d$ is the dimension invariant. Thus $\winf/I_{\T}\cong K[D]$ in this case, and characters for models of $\T$ are determined by their value on $D$. Write $\chi_t:\winf\to \winf/I_{\T}\cong K[D]\to K$ for the unique character that satisfies $\chi(D)=t$. Then $\chi_n$ is a good character for every $n\in \N$ . Indeed, it is the character of invariants of the structure $(K^n,(c_n,d_n))$ where $c_n = \sum_{i=1}^n e^i\ot e^i$ and $d_n = \sum_{i=1}^n e_i\ot e_i$, where $e_i$ is the standard basis for $K^n$. Up to isomorphism, this is the only $n$-dimensional vector space with a non-degenerate symmetric pairing, and it thus have a closed orbit. It is easy to see that $(\chi_t)$ is an additive family. Since $\chi_n$ is good for every $n\in\N$, and the dimensions of the hom-spaces is uniformly bounded, Proposition \ref{prop:familymain} implies that $\chi_t$ is good for every $t$. In case $t=n$ is an integer, we get by Proposition \ref{prop:mainvecstr} that $\C_{\chi_n}\cong \Rep(\Aut(K^n,(c_n,d_n))) = \Rep(O_t(K))$. We thus denote $\C_{\chi_t}$ by $\Rep(O_t(K))$. This example, as well as the next two examples, were given by Deligne in \cite{Deligne3}. 

\subsection{Non-degenerate skew-symmetric pairings and the categories $\Rep(\text{Sp}_t(K))$}
Consider now skew-symmetric pairings. In a very similar way to the last example, such structures are given by two structure tensors: $c$ of degree $(0,2)$ and $d$ of degree $(2,0)$. The theory here is made of the two axioms 
$c L^{(2)}_{(12)} + c$ and $ev(c\ot d) - \Id_W$. As before, one can show that $d$ is also skew-symmetric, and that $\winf/I_{\T}\cong K[D]$. We define $\chi_t:\winf\to \winf/I_{\T}\cong K[D]\to K$ to be the character which sends $D$ to $t\in K$. Just as in the last example we can show that the hom-spaces in $\C_{univ}^{\T}$ have finite rank over $\winf/I_{\T}$. For $t=2n$ where $n\in \N$, the character $\chi_{2n}$ is a good character, since it is the character of the structure $(K^{2n},c_n,d_n)$ where $c_n = \sum_{i=1}^n e^i\ot e^{i+n} - e^{i+n}\ot e^i$ and $d_n =  \sum_{i=1}^n e_i\ot e_{i+n} - e_{i+n}\ot e_i$. Since $(\chi_t)$ is an additive family, Proposition \ref{prop:familymain} implies that $\chi_t$ is a good character for all $t$. Since $\C_{\chi_{2n}} \cong \Rep(\Aut(K^{2n},(c_n,d_n)))\cong \Rep(\text{Sp}_{2n}(K))$ we write $\C_{\chi_t} = \Rep(\text{Sp}_t(K))$. As pointed out in Section 9.5 of \cite{Deligne3} for the case that $t$ is an integer, there is a connection between the categories $\Rep(\text{O}_t)$ and $\Rep{Sp}_{-t}$. 
We have an equivalence of symmetric monoidal categories $F:\Rep(\text{O}_t)\boxtimes \text{sVec}_K\cong \Rep(\text{Sp}_{-t})\boxtimes \text{sVec}_K$. This equivalence is given by acting as the identity on $\text{sVec}_K$, and by sending the tautological structure $W$ of $\Rep(\text{O}_t)$ to $W\boxtimes k_{-}$, where $k_{-}$ is the odd vector space of dimension 1, and $W\in \Rep(\text{Sp}_{-t})$ is the tautological object. One can easily check that $W\in \Rep(\text{O}_t)$ and $W\boxtimes k_{-}\in \Rep(\text{Sp}_{-t})$ have the same invariants, and therefore the functor $F$ is well defined. It is also easy to write its inverse, by a similar formula. 

\subsection{Separable commutative algebras and the categories $\Rep(S_t)$}
In this example we consider the structure of separable commutative algebras. Such an algebra is an associative commutative unital algebra $W$ such that the multiplication $m:W\ot W\to W$ splits as a $W-W$-bimodule morphism. The splitting $s$ is determined by $c:=s(1)\in W\ot W$. 

To phrase everything in the language of algebraic structures, a separable commutative algebra contains a multiplication $m:W\ot W\to W$, a unit $u:\one\to W$, and separability idempotent $c:\one\to W\ot W$. Another formulation is given by replacing $c$ with $\Delta:= ev(m\ot c):W\to W\ot W$. The separability idempotent $c$ can be recovered from $\Delta$ as $\Delta\circ u:\one\to W\to W\ot W$. The comultiplication $\Delta$ was used in Section \ref{sec:goodalgebra}. 

The theory $\T_{sep}$ of separable commutative algebras contains the axioms saying that $m$ is associative and commutative, that $u$ is a unit for $m$, and that $c$ defines a splitting of $m$ as a $W-W$-bimodule morphism. 
This boils down to the following equality of morphisms: 
\begin{center} 
\begin{equation}\label{fig:sepalg} \scalebox{0.7}{
\begin{tikzpicture}
	\begin{pgfonlayer}{nodelayer}
		\node [style=2function] (0) at (-4, 3.5) {$c$};
		\node [style=2function] (1) at (-4, 5) {$m$};
		\node [style=none] (2) at (-4.25, 3.75) {};
		\node [style=none] (3) at (-3.75, 3.75) {};
		\node [style=none] (4) at (-3.75, 4.75) {};
		\node [style=none] (5) at (-4.25, 4.75) {};
		\node [style=none] (6) at (-4, 5.25) {};
		\node [style=none] (7) at (-4, 6) {};
		\node [style=none] (8) at (-2.5, 4.5) {=};
		\node [style=1function] (9) at (-0.75, 4.5) {$u$};
		\node [style=none] (10) at (-0.75, 4.75) {};
		\node [style=none] (11) at (-0.75, 6) {};
		\node [style=2function] (12) at (4.5, 3.25) {$c$};
		\node [style=2function] (13) at (2.75, 5) {$m$};
		\node [style=none] (14) at (4.25, 3.5) {};
		\node [style=none] (15) at (2.5, 3) {};
		\node [style=none] (16) at (2.5, 4.75) {};
		\node [style=none] (17) at (3, 4.75) {};
		\node [style=none] (18) at (2.75, 5.25) {};
		\node [style=none] (19) at (2.75, 6.25) {};
		\node [style=none] (20) at (4.75, 3.5) {};
		\node [style=none] (21) at (4.75, 6.25) {};
		\node [style=2function] (22) at (6.75, 3.25) {$c$};
		\node [style=2function] (23) at (8.25, 5) {$m$};
		\node [style=none] (24) at (7, 3.5) {};
		\node [style=none] (25) at (8.5, 3) {};
		\node [style=none] (26) at (8.5, 4.75) {};
		\node [style=none] (27) at (8, 4.75) {};
		\node [style=none] (28) at (8.25, 5.25) {};
		\node [style=none] (29) at (8.25, 6.25) {};
		\node [style=none] (30) at (6.5, 3.5) {};
		\node [style=none] (31) at (6.5, 6.25) {};
		\node [style=none] (32) at (5.5, 4.5) {=};
		\node [style=none] (33) at (0.25, 3.25) {,};
	\end{pgfonlayer}
	\begin{pgfonlayer}{edgelayer}
		\draw (5.center) to (2.center);
		\draw (4.center) to (3.center);
		\draw (7.center) to (6.center);
		\draw (11.center) to (10.center);
		\draw [in=90, out=-90, looseness=1.25] (17.center) to (14.center);
		\draw (16.center) to (15.center);
		\draw (19.center) to (18.center);
		\draw (21.center) to (20.center);
		\draw [in=90, out=-90, looseness=1.25] (27.center) to (24.center);
		\draw (26.center) to (25.center);
		\draw (29.center) to (28.center);
		\draw (31.center) to (30.center);
	\end{pgfonlayer}
\end{tikzpicture}}
\end{equation}
\end{center}
The axioms in $\T_{sep}$ also imply the following equality of constructible morphisms:
\begin{center}\begin{equation}\label{fig:sepalg2}\scalebox{0.7}{
\begin{tikzpicture}
	\begin{pgfonlayer}{nodelayer}
		\node [style=2function] (0) at (-6.5, 2.25) {$c$};
		\node [style=2function] (1) at (-8.25, 4) {$m$};
		\node [style=none] (2) at (-6.75, 2.5) {};
		\node [style=none] (3) at (-8.5, 2.25) {};
		\node [style=none] (4) at (-8.5, 3.75) {};
		\node [style=none] (5) at (-8, 3.75) {};
		\node [style=none] (6) at (-8.25, 4.25) {};
		\node [style=none] (7) at (-8.25, 5) {};
		\node [style=none] (8) at (-6.25, 2.5) {};
		\node [style=none] (9) at (-6.25, 5) {};
		\node [style=2function] (10) at (-4, 2.25) {$c$};
		\node [style=2function] (11) at (-2.5, 4) {$m$};
		\node [style=none] (12) at (-3.75, 2.5) {};
		\node [style=none] (13) at (-2.25, 2.75) {};
		\node [style=none] (14) at (-2.25, 3.75) {};
		\node [style=none] (15) at (-2.75, 3.75) {};
		\node [style=none] (16) at (-2.5, 4.25) {};
		\node [style=none] (17) at (-2.5, 5) {};
		\node [style=none] (18) at (-4.25, 2.5) {};
		\node [style=none] (19) at (-4.25, 5) {};
		\node [style=none] (20) at (-5, 3.75) {=};
		\node [style=none] (21) at (-9.75, 5) {};
		\node [style=none] (22) at (-9.75, 2.25) {};
		\node [style=none] (23) at (-0.75, 5) {};
		\node [style=none] (24) at (-0.75, 2.75) {};
		\node [style=none] (25) at (0, 3.75) {=};
		\node [style=2function] (26) at (1, 2.25) {$c$};
		\node [style=none] (27) at (1.25, 2.5) {};
		\node [style=none] (28) at (0.75, 2.5) {};
		\node [style=2function] (29) at (1, 4) {$m$};
		\node [style=none] (30) at (1.25, 3.75) {};
		\node [style=none] (31) at (0.75, 3.75) {};
		\node [style=none] (32) at (1, 4.25) {};
		\node [style=none] (33) at (1, 5) {};
		\node [style=none] (34) at (2, 3.75) {=};
		\node [style=2function] (35) at (3.25, 2.25) {$c$};
		\node [style=none] (36) at (3.5, 2.5) {};
		\node [style=none] (37) at (3, 2.5) {};
		\node [style=2function] (38) at (3.25, 4) {$m$};
		\node [style=none] (39) at (3, 3.75) {};
		\node [style=none] (40) at (3.5, 3.75) {};
		\node [style=none] (41) at (3.25, 4.25) {};
		\node [style=none] (42) at (3.25, 5) {};
		\node [style=none] (43) at (4.25, 4) {=};
		\node [style=1function] (44) at (5.25, 4) {$u$};
		\node [style=none] (45) at (5.25, 4.25) {};
		\node [style=none] (46) at (5.25, 5) {};
	\end{pgfonlayer}
	\begin{pgfonlayer}{edgelayer}
		\draw [in=90, out=-90, looseness=1.25] (5.center) to (2.center);
		\draw (4.center) to (3.center);
		\draw (7.center) to (6.center);
		\draw (9.center) to (8.center);
		\draw [in=90, out=-90, looseness=1.25] (15.center) to (12.center);
		\draw (14.center) to (13.center);
		\draw (17.center) to (16.center);
		\draw (19.center) to (18.center);
		\draw (21.center) to (22.center);
		\draw [bend right=90, looseness=1.75] (22.center) to (3.center);
		\draw [bend left=90, looseness=1.50] (21.center) to (7.center);
		\draw [bend left=90] (19.center) to (23.center);
		\draw (24.center) to (23.center);
		\draw [bend left=90, looseness=1.25] (24.center) to (13.center);
		\draw [in=90, out=-90, looseness=1.25] (30.center) to (28.center);
		\draw [in=90, out=-90, looseness=1.50] (31.center) to (27.center);
		\draw (32.center) to (33.center);
		\draw [in=90, out=-90, looseness=1.25] (39.center) to (37.center);
		\draw [in=90, out=-90, looseness=1.50] (40.center) to (36.center);
		\draw (41.center) to (42.center);
		\draw (46.center) to (45.center);
	\end{pgfonlayer}
\end{tikzpicture}}
\end{equation}
\end{center}
We then get the following equality of diagrams:
\begin{center}
\begin{equation}\label{fig:sepalg3}\scalebox{0.7}{
\begin{tikzpicture}
	\begin{pgfonlayer}{nodelayer}
		\node [style=2function] (0) at (3.25, 2.25) {$c$};
		\node [style=2function] (1) at (2, 4) {$m$};
		\node [style=none] (2) at (3, 2.5) {};
		\node [style=none] (3) at (1.75, 2.25) {};
		\node [style=none] (4) at (1.75, 3.75) {};
		\node [style=none] (5) at (2.25, 3.75) {};
		\node [style=none] (6) at (2, 4.25) {};
		\node [style=none] (7) at (2, 5) {};
		\node [style=none] (8) at (3.5, 2.5) {};
		\node [style=none] (9) at (3.5, 5) {};
		\node [style=none] (10) at (1, 5) {};
		\node [style=none] (11) at (1, 3.25) {};
		\node [style=2function] (12) at (-0.25, 4.5) {$m$};
		\node [style=none] (13) at (-0.5, 4.25) {};
		\node [style=none] (14) at (0, 4.25) {};
		\node [style=none] (15) at (-0.25, 4.75) {};
		\node [style=none] (16) at (-0.25, 5.5) {};
		\node [style=none] (17) at (-1.5, 5.5) {};
		\node [style=none] (18) at (-1.5, 3.25) {};
		\node [style=none] (19) at (-0.5, 3.25) {};
		\node [style=none] (20) at (0, 3.25) {};
		\node [style=none] (21) at (4.25, 3.5) {=};
		\node [style=2function] (22) at (7.25, 2.25) {$c$};
		\node [style=2function] (23) at (8.5, 4.25) {$m$};
		\node [style=none] (24) at (7.5, 2.5) {};
		\node [style=none] (25) at (8.75, 2.5) {};
		\node [style=none] (26) at (8.75, 4) {};
		\node [style=none] (27) at (8.25, 4) {};
		\node [style=none] (28) at (8.5, 4.5) {};
		\node [style=none] (29) at (8.5, 5.25) {};
		\node [style=none] (30) at (7.25, 2.5) {};
		\node [style=none] (31) at (7.25, 3.5) {};
		\node [style=2function] (34) at (6.25, 5.25) {$m$};
		\node [style=none] (35) at (6, 5) {};
		\node [style=none] (36) at (6.5, 5) {};
		\node [style=none] (37) at (6.25, 5.5) {};
		\node [style=none] (38) at (6.25, 6.25) {};
		\node [style=none] (39) at (5, 6.25) {};
		\node [style=none] (40) at (5, 4) {};
		\node [style=none] (41) at (6, 4) {};
		\node [style=none] (42) at (6.5, 4) {};
		\node [style=none] (43) at (10.25, 3.5) {=};
		\node [style=1function] (44) at (11, 3.5) {$u$};
		\node [style=2function] (45) at (13, 3.5) {$m$};
		\node [style=none] (46) at (11, 4.5) {};
		\node [style=none] (47) at (12.75, 3.25) {};
		\node [style=none] (48) at (13.25, 3.25) {};
		\node [style=none] (49) at (13, 3.75) {};
		\node [style=none] (50) at (12, 4.5) {};
		\node [style=none] (51) at (12, 2.5) {};
		\node [style=none] (52) at (12.75, 2.5) {};
		\node [style=none] (53) at (13.25, 2.5) {};
		\node [style=none] (54) at (13, 4.5) {};
		\node [style=none] (55) at (14.25, 3.5) {=};
		\node [style=1function] (56) at (15.75, 3.5) {$\Id_W$};
		\node [style=none] (57) at (15.75, 3.75) {};
		\node [style=none] (58) at (15.75, 4.5) {};
		\node [style=none] (59) at (15.75, 3.25) {};
		\node [style=none] (60) at (15.75, 2.75) {};
		\node [style=none] (61) at (11, 3.75) {};
	\end{pgfonlayer}
	\begin{pgfonlayer}{edgelayer}
		\draw [in=90, out=-90, looseness=1.25] (5.center) to (2.center);
		\draw (4.center) to (3.center);
		\draw (7.center) to (6.center);
		\draw (9.center) to (8.center);
		\draw (10.center) to (11.center);
		\draw [bend left=90, looseness=1.50] (10.center) to (7.center);
		\draw (16.center) to (15.center);
		\draw [bend left=270, looseness=1.50] (16.center) to (17.center);
		\draw (17.center) to (18.center);
		\draw [bend right=90, looseness=2.00] (18.center) to (19.center);
		\draw (19.center) to (13.center);
		\draw [bend left=90, looseness=2.00] (11.center) to (20.center);
		\draw (20.center) to (14.center);
		\draw [in=90, out=-90, looseness=1.25] (27.center) to (24.center);
		\draw (26.center) to (25.center);
		\draw (29.center) to (28.center);
		\draw (31.center) to (30.center);
		\draw (38.center) to (37.center);
		\draw [bend left=270, looseness=1.50] (38.center) to (39.center);
		\draw (39.center) to (40.center);
		\draw [bend right=90, looseness=2.00] (40.center) to (41.center);
		\draw (41.center) to (35.center);
		\draw (42.center) to (36.center);
		\draw [in=90, out=-90, looseness=1.50] (42.center) to (31.center);
		\draw [bend left=90, looseness=1.75] (46.center) to (50.center);
		\draw (50.center) to (51.center);
		\draw [bend right=90, looseness=2.25] (51.center) to (52.center);
		\draw (52.center) to (47.center);
		\draw (48.center) to (53.center);
		\draw (54.center) to (49.center);
		\draw (59.center) to (60.center);
		\draw (58.center) to (57.center);
		\draw (46.center) to (61.center);
	\end{pgfonlayer}
\end{tikzpicture}}
\end{equation}

\end{center}
We claim the following: 
\begin{lemma}
The algebra $\winf^{\T_{sep}}$ is isomorphic to $K[D]$. The natural projection $\winf\to\winf^{\T_{sep}}$ sends every closed connected diagram to $D=\dim(W)$. 
\end{lemma}
\begin{proof}
Let $Di$ be a closed connected diagram. We will show that we can reduce $Di$ to $D$ using the axioms in $\T_{sep}$. We shall do so by induction on the number $n$ of appearances of the tensor $c$ in $Di$. If $n=0$ then $Di$ contains only the tensors $u$ and $m$. Since $Di$ is a closed diagram, the number of $u$-boxes and the number of $m$-boxes must be equal. Since only the tensor $m$ contains input strings, the output of every $u$-box will be connected to one of the input strings of $m$. But since $u$ is a unit for $m$, we can replace this by $\Id_W$. We can get rid of all appearances of $u$ in this way. Since a closed diagram cannot have only $m$-boxes, because the number of input and output strings will not be balanced, we are left with a closed connected diagram with no $u$-,$m$-, or $c$-boxes. Such a diagram can only be the trace of $\Id_W$, or $D=\dim(W)$.

Consider now the case where $n>0$. If $Di$ contains a diagram of the form $mc$ then we can use  
the axioms in Figure \ref{fig:sepalg} to reduce these two boxes to the box $u$. This also reduces the number of appearances of $c$ by one. 
In the general case, since $m$ is the only type of box with input strings, if we have any $c$-box then its output strings will be connected to an input string of some $m$-box. We follow the output string of this $m$-box, which must be connected as well to the input string of another $m$-box, or to the other input string of itself. Since the number of $m$-boxes in the diagram is finite, a circle will be closed eventually. Since $m$ is associative and commutative, we can find an equivalent diagram to $Di$ which contains a diagram which appear in Figure \ref{fig:sepalg2} or \ref{fig:sepalg3}. In any case, we can remove a sub-diagram which contains the $c$-box and replace it with either $\Id_w$ or $u$. Again, this will reduce the number of $c$-boxes by one, and the induction is completed. 
\end{proof}
Thus, for every $t\in K$ we have a character $\chi_t:\winf\to K$ which sends all closed connected diagrams to $t$. It is easy to see that $(\chi_t)$ is an additive one-parameter family. 
We claim the following:
\begin{proposition}
All the characters $\chi_t$ are good characters.
\end{proposition}
\begin{proof}
For $t=n$, a natural number, $\chi_n$ is the character of invariants associated to the structure $(K^n,m_n,u_n,c_n)$ where $m_n = \sum_{i=1}^n e_i\ot e^i\ot e^i$, $u_n= \sum_{i=1}^n e_i$, and $c_n = \sum_{i=1}^n e_i\ot e_i$. This is just the algebra $K^n$ with pointwise addition and multiplication. 
Since the dimension of this algebra is $n$ we see that $\chi_n$ is a good character. The automorphism group of the algebra $K^n$ is $S_n$, acting by permuting the basis. For every natural $a,b$ it holds that the dimension of $((K^n)^{a,b})^{S_n}$ is bounded by the number of equivalence relations on $\{1,2,\ldots, a+b\}$, which does not depend on $n$. The conditions of Proposition \ref{prop:familymain} are satisfied, so we see that indeed $\chi_t$ is a good character for every $t\in K$.
\end{proof}
Since $K^n$ is the only $n$-dimensional commutative separable algebra, it has a closed orbit, and $\C_{\chi_n}\cong \Rep(S_n)$. We thus get an interpolation of the categories $\Rep(S_n)$. The resulting category $\C_{\chi_t}$ is then the Deligne's category $\Rep(S_t)$. The family of good characters that correspond to separable commutative algebras then just gives the algebra $K$.  

\subsection{Commutative Frobenius algebras and the categories $\ul{\dcob}_{\alpha}$}\label{subsec:dcob}
We consider now the structure of commutative Frobenius algebras. Recall that such a structure $W$ has the following structure tensors: multiplication $m:W\ot W\to W$, unit $u:\one\to W$, and counit $\epsilon:W\to\one$. The axioms for a commutative Frobenius algebra are the following: $m$ is an associative and commutative multiplication, $u$ is a unit for $m$, and the pairing $\epsilon m:W\ot W\to \one$ is non-degenerate. As in the case of $\text{O}_t$, we will record the non-degeneracy by adding $c:\one\to W\ot W$ such that $ev(c\ot \epsilon m) = \Id_W$. 
We write $\Delta:W\to W\ot W$ for the dual of $m$ with respect to the pairing $\epsilon m$. One can prove that $\Delta$ is also given by $ev(c\ot m)$. We denote the theory of commutative Frobenius algebras by $\T_{ComFr}$. 
It is known (see 
Section 2.3. in \cite{KKO}) that commutative Frobenius algebras in a given category $\C$ correspond to 2-dimensional oriented topological quantum field theories (TQFTs) in $\C$. The correspondence between the 2-dimensional cobordisms and our diagrams is given by the following dictionary:
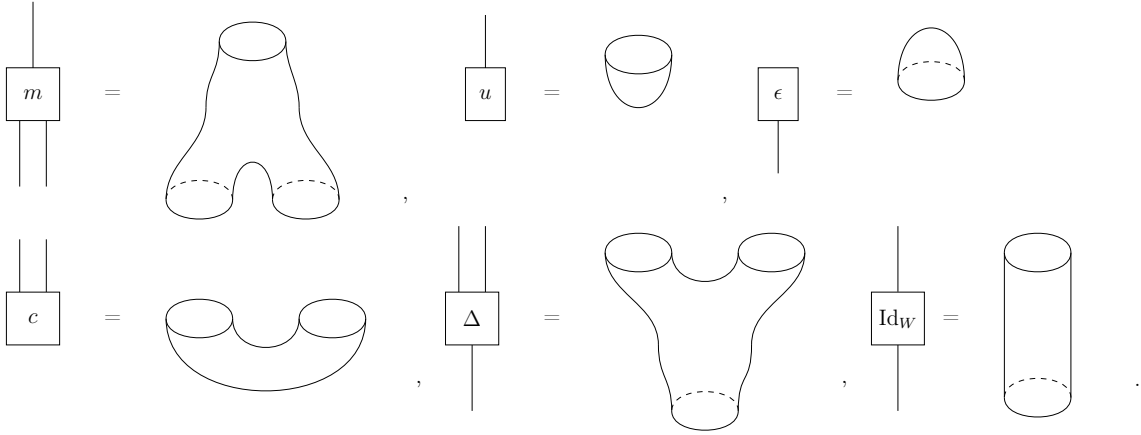
\begin{center}\scalebox{0.7}{
\begin{tikzpicture}
	\begin{pgfonlayer}{nodelayer}
		\node [style=none] (0) at (-5.25, -1.25) {};
		\node [style=none] (1) at (-4, -1.25) {};
		\node [style=none] (2) at (-3.25, -1.25) {};
		\node [style=none] (3) at (-2, -1.25) {};
		\node [style=none] (4) at (-4.25, 1.75) {};
		\node [style=none] (5) at (-3, 1.75) {};
		\node [style=none] (6) at (-4.5, 0.5) {};
		\node [style=none] (7) at (-2.75, 0.5) {};
		\node [style=2function] (8) at (-7.75, 0.75) {$m$};
		\node [style=none] (9) at (-7.5, 0.5) {};
		\node [style=none] (10) at (-8, 0.5) {};
		\node [style=none] (11) at (-7.75, 1) {};
		\node [style=none] (12) at (-7.5, -1) {};
		\node [style=none] (13) at (-8, -1) {};
		\node [style=none] (14) at (-7.75, 2.5) {};
		\node [style=none] (15) at (-6.25, 0.75) {=};
		\node [style=none] (16) at (-0.75, -1.25) {,};
		\node [style=1function] (17) at (0.75, 0.75) {$u$};
		\node [style=none] (18) at (0.75, 1) {};
		\node [style=none] (19) at (0.75, 2.25) {};
		\node [style=none] (20) at (2, 0.75) {=};
		\node [style=none] (21) at (3, 1.5) {};
		\node [style=none] (22) at (4.25, 1.5) {};
		\node [style=none] (23) at (5.25, -1.25) {,};
		\node [style=1function] (24) at (6.25, 0.75) {$\epsilon$};
		\node [style=none] (25) at (6.25, -0.75) {};
		\node [style=none] (26) at (6.25, 0.5) {};
		\node [style=none] (27) at (7.5, 0.75) {=};
		\node [style=none] (28) at (8.5, 1) {};
		\node [style=none] (29) at (9.75, 1) {};
		\node [style=none] (30) at (-5.25, -3.5) {};
		\node [style=none] (31) at (-4, -3.5) {};
		\node [style=none] (32) at (-2.75, -3.5) {};
		\node [style=none] (33) at (-1.5, -3.5) {};
		\node [style=2function] (34) at (-7.75, -3.5) {$c$};
		\node [style=none] (35) at (-7.5, -3.25) {};
		\node [style=none] (36) at (-8, -3.25) {};
		\node [style=none] (38) at (-6.25, -3.5) {=};
		\node [style=none] (39) at (-8, -2) {};
		\node [style=none] (40) at (-7.5, -2) {};
		\node [style=none] (41) at (-0.5, -4.75) {,};
		\node [style=2function] (50) at (0.5, -3.5) {$\Delta$};
		\node [style=none] (51) at (0.75, -1.75) {};
		\node [style=none] (52) at (0.25, -1.75) {};
		\node [style=none] (53) at (0.5, -5.25) {};
		\node [style=none] (54) at (0.75, -3.25) {};
		\node [style=none] (55) at (0.25, -3.25) {};
		\node [style=none] (56) at (0.5, -3.75) {};
		\node [style=none] (57) at (2, -3.5) {=};
		\node [style=none] (58) at (3, -2.25) {};
		\node [style=none] (59) at (4.25, -2.25) {};
		\node [style=none] (60) at (5.5, -2.25) {};
		\node [style=none] (61) at (6.75, -2.25) {};
		\node [style=none] (62) at (4.25, -5.25) {};
		\node [style=none] (63) at (5.5, -5.25) {};
		\node [style=none] (64) at (4, -4) {};
		\node [style=none] (65) at (5.75, -4) {};
		\node [style=none] (66) at (7.5, -4.75) {,};
		\node [style=1function] (67) at (8.5, -3.5) {$\Id_W$};
		\node [style=none] (68) at (8.5, -3.25) {};
		\node [style=none] (69) at (8.5, -1.75) {};
		\node [style=none] (70) at (8.5, -3.75) {};
		\node [style=none] (71) at (8.5, -5.25) {};
		\node [style=none] (72) at (9.5, -3.5) {=};
		\node [style=none] (73) at (10.5, -2.25) {};
		\node [style=none] (74) at (11.75, -2.25) {};
		\node [style=none] (75) at (10.5, -5) {};
		\node [style=none] (76) at (11.75, -5) {};
		\node [style=none] (77) at (13, -4.75) {.};
	\end{pgfonlayer}
	\begin{pgfonlayer}{edgelayer}
		\draw [bend left=90] (1.center) to (0.center);
		\draw [style=black dashed, bend left=270] (1.center) to (0.center);
		\draw [bend left=90] (3.center) to (2.center);
		\draw [style=black dashed, bend left=270] (3.center) to (2.center);
		\draw [bend left=90] (5.center) to (4.center);
		\draw [bend left=270] (5.center) to (4.center);
		\draw [in=90, out=-90, looseness=1.25] (4.center) to (6.center);
		\draw [in=90, out=-90] (6.center) to (0.center);
		\draw [in=90, out=-90] (5.center) to (7.center);
		\draw [in=90, out=-90] (7.center) to (3.center);
		\draw [bend left=90, looseness=3.25] (1.center) to (2.center);
		\draw (14.center) to (11.center);
		\draw (10.center) to (13.center);
		\draw (9.center) to (12.center);
		\draw (19.center) to (18.center);
		\draw [bend left=90] (22.center) to (21.center);
		\draw [bend left=270] (22.center) to (21.center);
		\draw [bend right=90, looseness=2.75] (21.center) to (22.center);
		\draw (26.center) to (25.center);
		\draw [bend left=90] (29.center) to (28.center);
		\draw [bend left=90, looseness=2.75] (28.center) to (29.center);
		\draw [style=black dashed, bend left=270] (29.center) to (28.center);
		\draw [bend left=90] (31.center) to (30.center);
		\draw [bend left=270] (31.center) to (30.center);
		\draw [bend left=90] (33.center) to (32.center);
		\draw [bend left=270] (33.center) to (32.center);
		\draw [bend left=90, looseness=1.50] (32.center) to (31.center);
		\draw [bend right=90, looseness=1.25] (30.center) to (33.center);
		\draw (39.center) to (36.center);
		\draw (40.center) to (35.center);
		\draw (56.center) to (53.center);
		\draw (52.center) to (55.center);
		\draw (51.center) to (54.center);
		\draw [bend left=90] (59.center) to (58.center);
		\draw [bend left=270] (59.center) to (58.center);
		\draw [bend left=90] (61.center) to (60.center);
		\draw [bend left=270] (61.center) to (60.center);
		\draw [bend left=90, looseness=1.50] (60.center) to (59.center);
		\draw [bend left=90] (63.center) to (62.center);
		\draw [style=black dashed, bend left=270] (63.center) to (62.center);
		\draw [in=-90, out=90, looseness=1.25] (62.center) to (64.center);
		\draw [in=-90, out=90] (64.center) to (58.center);
		\draw [in=-90, out=90, looseness=1.50] (63.center) to (65.center);
		\draw [in=-90, out=90] (65.center) to (61.center);
		\draw (69.center) to (67);
		\draw (70.center) to (71.center);
		\draw [bend left=90] (74.center) to (73.center);
		\draw [bend left=270] (74.center) to (73.center);
		\draw [bend left=90] (76.center) to (75.center);
		\draw [style=black dashed, bend left=270] (76.center) to (75.center);
		\draw (73.center) to (75.center);
		\draw (74.center) to (76.center);
	\end{pgfonlayer}
\end{tikzpicture}}
\end{center}
Composition of morphisms is given by gluing of cobordisms. The closed connected diagrams then correspond to closed connected two-dimensional orientable manifolds. Such manifolds are classified by their genus.
We write $M_g$ for the connected oriented two-dimensional manifold of genus $g$. If $A$ is a commutative Frobenius algebra and $\chi_A$ is its character of invariants, we can consider $\chi_A(M_g)$, where we understand that we evaluate $\chi$ on the closed connected diagram that corresponds to a manifold of genus g. We write $\alpha_g:=\chi_A(M_g)\in K$. The sequence $(\alpha_g)$ also appears in \cite{KKO}. 

We thus see that characters $\winf^{\T_{ComFr}}\to K$ are in one-to-one correspondence with sequences $(\alpha_g)$ of elements in $K$. 
The category $\C_{\chi}$ is the category $\ul{\dcob}_{\alpha}$ from \cite{KKO}, where $(\alpha_g)$ is the sequence that corresponds to the character $\chi$. 

We next determine the sequences which correspond to good characters. For this, define the \emph{handle} endomorphism $x=m\Delta:A\to A$. This is given by the cobordism 
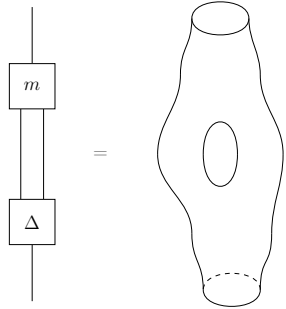
\begin{center}\scalebox{0.6}{
\begin{tikzpicture}
	\begin{pgfonlayer}{nodelayer}
		\node [style=none] (0) at (0.75, -1.25) {};
		\node [style=none] (1) at (1.75, -1.25) {};
		\node [style=none] (2) at (2.5, -1.25) {};
		\node [style=none] (3) at (3.5, -1.25) {};
		\node [style=none] (4) at (1.5, 1.75) {};
		\node [style=none] (5) at (2.75, 1.75) {};
		\node [style=none] (6) at (1.25, 0.5) {};
		\node [style=none] (7) at (3, 0.5) {};
		\node [style=2function] (24) at (-2, -2.75) {$\Delta$};
		\node [style=none] (25) at (-1.75, 0) {};
		\node [style=none] (26) at (-2.25, 0) {};
		\node [style=none] (27) at (-2, -4.5) {};
		\node [style=none] (28) at (-1.75, -2.5) {};
		\node [style=none] (29) at (-2.25, -2.5) {};
		\node [style=none] (30) at (-2, -3) {};
		\node [style=none] (31) at (-0.5, -1.25) {=};
		\node [style=none] (32) at (0.75, -1.25) {};
		\node [style=none] (33) at (1.75, -1.25) {};
		\node [style=none] (34) at (2.5, -1.25) {};
		\node [style=none] (35) at (3.5, -1.25) {};
		\node [style=none] (36) at (1.75, -4.25) {};
		\node [style=none] (37) at (3, -4.25) {};
		\node [style=none] (38) at (1.5, -3) {};
		\node [style=none] (39) at (3.25, -3) {};
		\node [style=2function] (40) at (-2, 0.25) {$m$};
		\node [style=none] (41) at (-2, 0.5) {};
		\node [style=none] (42) at (-2, 2) {};
	\end{pgfonlayer}
	\begin{pgfonlayer}{edgelayer}
		\draw [bend left=90] (5.center) to (4.center);
		\draw [bend left=270] (5.center) to (4.center);
		\draw [in=90, out=-90, looseness=1.25] (4.center) to (6.center);
		\draw [in=90, out=-90] (6.center) to (0.center);
		\draw [in=90, out=-90] (5.center) to (7.center);
		\draw [in=90, out=-90] (7.center) to (3.center);
		\draw [bend left=90, looseness=3.25] (1.center) to (2.center);
		\draw (30.center) to (27.center);
		\draw (26.center) to (29.center);
		\draw (25.center) to (28.center);
		\draw [bend left=90, looseness=3.25] (34.center) to (33.center);
		\draw [bend left=90] (37.center) to (36.center);
		\draw [style=black dashed, bend left=270] (37.center) to (36.center);
		\draw [in=-90, out=90, looseness=1.25] (36.center) to (38.center);
		\draw [in=-90, out=90] (38.center) to (32.center);
		\draw [in=-90, out=90, looseness=1.50] (37.center) to (39.center);
		\draw [in=-90, out=90] (39.center) to (35.center);
		\draw (42.center) to (41.center);
	\end{pgfonlayer}
\end{tikzpicture}}
\end{center}
It holds that $\Tr(x^n) = \alpha_{n+1}$ for every $n\in\N$. We consider now the function $Z(X) = \sum_i\alpha_iX^i$. By Corollary \ref{cor:endshape} we see that if $(\alpha_g)$ corresponds to a good character then this function has the form $\frac{P(X)}{Q(X)}$, where $Q(X)$ has no multiple roots, $\deg(P)+1\leq \deg(Q)$, and $Q(0)\neq 0$. We call rational functions of this specific form \emph{loyal}.
We claim the following:
\begin{proposition}[See Theorem 3.4. in \cite{KKO}]
The sequence $(\alpha_g)$ corresponds to a good character if and only if $Z(X)$ is a loyal rational function.
\end{proposition}
\begin{proof}
We have already seen that if $(\alpha_g)$ comes from a good character then $Z(X)$ is loyal. In the other direction, the set of loyal functions is a linear subspace of $K[[X]]$ spanned by the functions $\{1,X\}\cup\{\frac{1}{1-\la X}\}_{\la\in K^{\times}}$. Since we know that the set of good characters is a linear subspaces, it will be enough to show that all functions in this basis correspond to good characters.

For $\frac{1}{1-\la X}$ we have the following algebra: take $A=K$ with basis element $e$ and dual basis $f$. The structure tensors are given by $m=e\ot f\ot f$, $u=e$, $\epsilon= \frac{1}{\la}f$, $c = \lambda e\ot e$, and $\Delta = \lambda e\ot e\ot f$. We get that $x=\la \Id_K$, so $\Tr(x^n) = \la^{n}$, and $\alpha_0=\epsilon(u) = \frac{1}{\la}$. We get the rational function $\frac{1}{\la}\sum_i (\la x)^i = \frac{1}{\la}\frac{1}{1-\la X}$, which is the basis element $\frac{1}{1-\la X}$ rescaled by $\frac{1}{\la}$. Since the set of good characters is closed under multiplication by scalars this is good enough.

For the rational functions $1$ and $X$ we consider the following algebras: Let $A = k[y]/(y^2)$. 
We define $\epsilon_1(1) = 0$, $\epsilon_1(y) = 1$ and $\epsilon_2(1)= \epsilon_2(y)=1$. 
For $i=1,2$ let $A_i$ be the Frobenius algebra with $\epsilon=\epsilon_i$. We will calculate the resulting rational functions. 
A direct calculation shows that in both cases the handle endomorphism $x=m\Delta$ is given by $1\mapsto 2y$ and $y\mapsto 0$. This is a nilpotent endomorphism, and therefore all the scalars $\alpha_i$ for $i\geq 2$ vanish. We have that $\alpha_1=\dim A = 2$ in both cases. We have $\alpha_0 = \epsilon(u)$, so it is $0$ for $A_1$ and $1$ for $A_2$. We thus get the rational functions $2X$ and $1+2X$. Since we get a spanning set for the loyal rational functions we are done.
\end{proof}
The collection of good characters arising from commutative Frobenius algebras is closed under sums, products, and multiplication by a scalar in $K$, it thus form a subalgebra of $K^P$. By considering the proposition above we see that this algebra is isomorphic to $KM\oplus Ke$ where $M$ is the monoid $(K,\times)$, we write $U_{\la}$ for a basis of $KM$, $e$ is an idempotent, and $U_{\la}e = \la e$. Here $e$ corresponds to the characters with $\alpha$-sequence $(0,1,0,0,\ldots)$. 
\begin{remark}
The paper \cite{KKO} also studies the case where the ground field is of characteristic $p$. 
\end{remark}

\subsection{Wreath products with $S_t$}
If $G$ is any group one can construct the wreath product $S_n \wr G := S_n\ltimes G^n$, where $S_n$ acts by permuting the entries in $G^n$. 
We will show here how we can interpolate this construction when $G$ is any reductive group. We will use Theorem \ref{thm:mainallG}, that says that $G$ is the automorphism group of some algebraic structure with a closed orbit. Our construction generalizes the construction of Knop from \cite{Knop}, where he constructed the categories $\Rep(S_t\wr G)$ for $G$ a finite group. The construction we present here work for general good characters as well. 

Let $(A,(y_i))$ be an algebraic structure in some good category $\D$. Write $\chi:\winf\to K$ for the character of invariants of $(A,(y_i))$. 
We will define now a new algebraic structure, and show that its direct sums can be interpolated. 
We define $\ol{A} = \one\oplus A$, and we consider the new algebraic structure $(\ol{A},(y_1,\ldots, y_r,P,u,\epsilon,m,c))$ where $P:\ol{A}\to\ol{A}$ is the projection with kernel $A$ and image $\one$, $u:\one\to \ol{A}$ is the natural inclusion, $\epsilon:\ol{A}\to \one$ is the natural projection, $m:\ol{A}\ot\ol{A}\to \ol{A}$ is given by $\ol{A}\ot\ol{A}\stackrel{P\ot 1}{\to} \one\ot \ol{A}\cong \ol{A}$, and $c:\one\to \ol{A}\ot\ol{A}$ is given by $\one\cong\one\ot\one\stackrel{u\ot u}{\to} \ol{A}\ot\ol{A}$.
We will write the new structure as $(\ol{A},(z_i))$. 
We claim the following:
\begin{lemma}
Assume that $(A,(y_i))$ is an algebraic structure in $Vec_K$ with automorphism group $G$. Then the automorphism group of $(\ol{A},(z_i))^{\oplus n}$ is $S_n\wr G$
\end{lemma}
\begin{proof}
We write $\ol{A}^{\oplus n} = A^{\oplus n}\oplus \one^{\oplus n}$. Write $B:=\one^{\oplus n} = Ke_1\oplus\cdots\oplus Ke_n$. There is an action of $S_n\wr G = S_n\ltimes G^n$ on $(\ol{A},(z_i))^{\oplus n}$, where $G^n$ acts diagonally on $A^{\oplus n}$ and trivially on $B$, and $S_n$ permutes the direct summands. We will show that these are all the possible automorphisms of the structure $(\ol{A},(z_i))^{\oplus n}$. 

The constructible map $P$ is just the projection onto $B$ with kernel $A^{\oplus n}$. Thus, every automorphism of $(\ol{A},(z_i))^{\oplus n}$ must preserve the direct sum decomposition $A^{\oplus n}\oplus B$. The map $m$ restricted to $B$ gives an algebra structure $m_B:B\ot B\to B$ on $B$. This algebra is just isomorphic to $\one^n$, and so its multiplication is given by the rule $e_i\cdot e_j = \delta_{i,j}e_i$. We thus see that restriction gives a group homomorphism $\phi$ from the automorphism group of $(\ol{A},(z_i))^{\oplus n}$ to the automorphism group of the algebra $B$, which is just $S_n$. 

We consider now an automorphism $g$ in $\Ker(\phi)$, and show that it must be in $G^n$. Indeed, the morphism $m$ induces a $B$-module structure on $A^{\oplus n}$. Since $g$ acts trivially on $B$, it sends $e_i\cdot W$ to itself. Thus, for every $i$ we get an induced automorphism $g_i$ of $(A,(y_i))$. The assignment $g\mapsto (g_1,\ldots, g_n)$ thus gives us the desired isomorphism between $Ker(\phi)$ and $G^n$ and we are done.
\end{proof}

We next claim the following:
\begin{lemma}
Assume that $(A,(y_i))$ is a $\Vec_K$-structure with a closed orbit and automorphism group $G$. Then $(\ol{A},(z_i))^{\oplus n}$ has a closed orbit and therefore $\C_{n\cdot\ol{\chi}}\cong \Rep(S_n\wr G)$. 
\end{lemma}
\begin{proof} 
By Corollary \ref{cor:closedorbit} we know that having a closed orbit is equivalent to the non-degeneracy of the pairings $X^{(p,q)}\ot X^{(q,p)}\to K$, where $X^{(p,q)}\subseteq A^{p,q}$ is the subspace of constructible elements. Since we will deal with a few different structures here we will write $X^{(p,q)}(A,(y_i))$ to make it clear what structure we are considering. We first show that $(\ol{A},(z_i))$ has a closed orbit. Using the projection $P:\ol{A}\to\ol{A}$ it follows easily that $X^{(p,q)}(\ol{A},(z_i))$ splits as $$X^{(p,q)}(\ol{A},(z_i)) = \bigoplus_{i_1,\ldots, i_{p+q}=1}^2X^{(p,q)}(\ol{A},(z_i))\cap (A_{i_1}\ot A_{i_2}\ot\cdots\ot A_{i_p}\ot (A_{i_{p+1}})^*\ot\cdots\ot (A_{i_{p+q}})^*),$$ where $A_1 = \one$ and $A_2 = A$. It then follows that we need to prove that the restriction of the pairing is non-degenerate on every one of these direct summands, using the fact that the space $X^{(q,p)}(\ol{A},(z_i))$ has a similar decomposition. But this follows easily from the fact that the pairing $X^{(p',q')}(A,(y_i))\ot X^{(q',p')}(A,(y_i))\to K$ is non-degenerate for $p'\leq p$ and $q'\leq q$, and from the fact that $u$ and $\epsilon$ provide a dual basis of $\one\subseteq \ol{A}$ and $\one^*\cong \one\subseteq \ol{A}^*$. 

Consider now the structure $(\ol{A},(z_i))^{\oplus n}$. Write $\ol{A}_i$ for the $i$-th copy of $\ol{A}$ in $\ol{A}^{\oplus n}$. We will begin by considering the structure $(\ol{A}^{\oplus n},z_{11},\ldots, z_{1n},\ldots,z_{l1},\ldots,z_{ln})$, where we write $z_{ij}$ for the tensor $z_i$ applied to $\ol{A}_j$. For convenience, we assume that $\Id_{\ol{A}}$ is one of the tensors in $\{z_j\}$, so that we get all the projections $\ol{A}^{\oplus n}\to \ol{A}_j\to \ol{A}^{\oplus n}$. 
The fact that the pairing on $X^{(p,q)}(\ol{A},(z_i))\ot X^{(q,p)}(\ol{A},(z_i))\to K$ is non-degenerate easily implies that the pairings for $(\ol{A}^{\oplus n},(z_{ij}))$ are non-degenerate. 
The group $S_n$ acts on $\ol{A}^{\oplus n}$ by permuting the direct summands $\ol{A}_i$. We will show that $X^{(p,q)}((\ol{A},(z_i))^{\oplus n}) = X^{(p,q)}(\ol{A}^{\oplus n},(z_{ij}))^{S_n}$. This will prove that the pairing is non-degenerate, since the field $K$ has characteristic zero, and $S_n$ is therefore a reductive group.

For this, notice first that all the structure tensors of $(\ol{A},(z_i))^{\oplus n}$ are $f_i:=z_{i1} + z_{i2} + \cdots + z_{in}$, which are invariant under the action of $S_n$. It follows easily that all the constructible tensors for $(\ol{A},(z_i))^{\oplus n}$ are $S_n$-invariant as well, and we thus have an inclusion $X^{(p,q)}((\ol{A},(z_i))^{\oplus n})\subseteq X^{(p,q)}(\ol{A}^{\oplus n},(z_{ij}))^{S_n}$. 

To prove the inclusion in the other direction, we notice that by reordering the tensors, $X^{(p,q)}(\ol{A}^{\oplus n},(z_{ij}))$ is spanned by elements of the form $c(\sigma,\tau,v_1,\ldots,v_n):=L^{(p)}_{\sigma}v_1\ot\cdots\ot v_nL^{(q)}_{\tau}$, where $v_i\in X^{(a_i,b_i)}(\ol{A}_i,z_{1i},\ldots, z_{li})$, $\sum_i a_i = p$, $\sum_i b_i = q$, $\sigma\in S_p$ and $\tau\in S_q$. By the reductivity of $S_n$ it will be enough to show that $\sum_{\mu\in S_n}\mu(c(\sigma,\tau,v_1,\ldots, v_n))$ is in $X^{(p,q)}((\ol{A},(z_i))^{\oplus n})$. Since the action of $S_n$ commutes with the action of $S_p\times S_q$, we can assume without loss of generality that $\sigma = \Id$ and $\tau = \Id$. We next use the fact that $c = \sum_i e_i\ot e_i$ and $(u\ot u)-c = \sum_{i\neq j} e_i\ot e_j$, using the notation for the basis of $B$ from the previous lemma. The constructible morphism $m$ enables us to consider $\ol{A}^{\oplus n}$ as a $B$-module. Write now $v_i = v_i(z_{1i},\ldots, z_{li})$. Using the elements $c$ and $u\ot u$ we can construct the element $u_a:=\sum_{\mu\in S_n} e_{\mu(1)}^{\ot a_1}\ot\cdots\ot e_{\mu(n)}^{\ot a_n}$ and  
$u_b:=\sum_{\mu\in S_n} e_{\mu(1)}^{\ot b_1}\ot\cdots\ot e_{\mu(n)}^{\ot b_n}$. It then holds that 
$\sum_{\mu\in S_n}\mu(c(\Id,\Id,v_1,\ldots, v_n)) = u_a(v'_1\ot\cdots\ot v'_n)u_b$, where we identify $u_a$ and $u_b$ with their action on $(\ol{A}^{\oplus n})^{p,q}$ via $m$, and where $v'_i = v_i(f_1,\ldots,f_l)$. 
But the last expression belongs to $X^{(p,q)}((\ol{A},(z_i))^{\oplus n})$, so we are done. 
\end{proof}

Following the above lemmas, we give the following definition:
\begin{definition} 
If $\chi:\winf\to K$ is the character of invariants of $(A,(y_i))$ we write $\ol{\chi}:K[\ol{X}]_{aug}\to K$ for the character of invariants of $(\ol{A},(z_i))$, where $K[\ol{X}]_{aug}$ is the universal ring of invariants for structures of type $((p_1,q_1),\ldots, (p_r,q_r),(1,1),(1,0),(0,1),(1,2),(2,0))$. 
If $\C_{\chi} \cong \Rep(G)$ we write $\Rep(S_t\wr G):=\C_{t\cdot\ol{\chi}}$ as by the above lemma it interpolates the categories $\Rep(S_n\wr G)$. 
\end{definition}
All the characters $t\cdot\ol{\chi}$ are good characters. This follows from the fact that $\ol{\chi}$ is a good character, and the set of good characters is closed under multiplication by a scalar. 
It was necessary to form the auxiliary structure $\ol{A}$ because it is possible that the automorphism group of $A^{\oplus n}$ will be strictly bigger than $S_n\wr G$. This happens for example for the families $\GL_t,\text{O}_t$, and $\text{Sp}_t$. The construction we have works as well also in case $\C_{\chi}$ does not admit a fiber functor to $\Vec_K$. 
 
\subsection{Finite modules over a discrete valuation ring and the categories $\Rep(\Aut(M_{a_1,\ldots, a_r}))$}
Let $\Ow$ be a discrete valuation ring with a uniformizer $\pi$ and a finite residue field of cardinality $q$.
We will consider here finitely generated modules over $\Ow_r:=\Ow/(\pi^r)$, where $r>0$ is some integer.
Such a module $M$ gives rise to the group algebra $KM=\text{span}_K\{U_m\}_{m\in M}$.
This group algebra is a Hopf algebra. The multiplication is given by $U_{m_1}U_{m_2} = U_{m_1+m_2}$, the comultiplication by $\Delta(U_m) = U_m\ot U_m$, the unit is $U_0$, the counit is $\epsilon(U_m)=1$ for every $m\in M$, and the antipode is given by $S(U_m) = U_{-m}$. In addition, for every $x\in \Ow_r$ we have a Hopf albgera homomorphism $T_x:KM\to KM$ given by $T_x(U_m) = U_{mx}$. These homomorphisms satisfy in addition the conditions $T_xT_y = T_{xy}$, $m(T_x\ot T_y)\Delta = T_{x+y}$, and $T_1=\Id_{KM}$. 
We can thus consider $(KM,m,\Delta,u,\epsilon,S,(T_x)_{x\in \Ow_r})$ as an algebraic structure. Let $\T$ be the theory containing the axioms for a commutative and cocommutative Hopf algebra, the axioms saying that $T_x$ is a homomorphism of Hopf algebras for every $x\in \Ow_r$, the axiom $T_xT_y = T_{xy}$ for all $x,y\in \Ow_r$, the axiom $m(T_x\ot T_y)\Delta = T_{x+y}$ for all $x,y\in \Ow_r$, and the axiom $T_1=\Id_{KM}$. 
\begin{claim} Isomorphism classes of models for $\T$ inside $\Vec_K$ are in one-to-one correspondence with isomorphism types of finite modules over $\Ow_r$. 
\end{claim}
\begin{proof} We have seen that if $M$ is a finite module over $\Ow_r$ then $KM$ is a model for the above theory $\T$. On the other hand, if $W$ is a model for $\T$, then by Cartier-Milnor-Moore-Kostant Theorem we know that a commutative cocomuutative Hopf algebra is necessarily the group algebra $KM$ of some finite abelian group $M$. The fact that $T_x:KM\to KM$ is a Hopf algebra homomorphism implies that $T_x$ arises from a group homomorphism $M\to M$. The other axioms ensure us that we get indeed a structure of an $\Ow_r$-module on $M$. 
\end{proof}

Since $\Ow$ is a principal ideal domain, the structure theorem for finitely generated modules over a PID applies here. Adapted to $\Ow/(\pi^r)$, we see that every module over $\Ow_r$ has the form 
$$M = (\Ow/(\pi))^{a_1}\oplus(\Ow/(\pi^2))^{a_2}\oplus\cdots\oplus (\Ow/(\pi^r))^{a_r},$$ for some $a_1,a_2,\ldots, a_r\in \N$.  Moreover, the tuple $(a_1,\ldots, a_r)\in \N^r$ is a complete set of invariants for $M$. We will write $M=M_{a_1,\ldots, a_r}$.

The integers $a_1,\ldots, a_r$ can be recovered from the associated character of invariants of $KM_{a_1,\ldots,a_r}$, which we denote by  $\chi_{a_1,\ldots, a_r}:\winf\to K$. Indeed, for every $i=1,2,\ldots, r$ it holds that \begin{equation}\label{eq:ci} c_i:=\Tr(T_{1+\pi^i}) = q^{a_1+2a_2+\cdots + ia_i + ia_{i+1}+\cdots + ia_r},\end{equation} where $q=|\Ow/(\pi)|$. This is because $\Tr(T_{1+\pi^i}) = |\{m\in M| (1+\pi^i)m=m\}| = 
\{m\in M|\pi^im=0\}| = |\Hom_{\Ow_r}(\Ow_i,M)|$. 

We can write $q^{a_i}$ in the form $\prod_j c_j^{x_{ij}}$ for some $x_{ij}\in \Z$. 
We will need the following lemma:
\begin{lemma}\label{lem:extensionhom}
Let $L\subseteq N$ be an inclusion of finitely generated $\Ow_r$-modules. Let $M$ be another $\Ow_r$-modules. Then there are elements $\alpha_i\in L$ and integers $n_i\in \{0,1,\ldots, r-1\}$ such that a homomorphism $\psi:L\to M$ can be extended to a homomorphism $N\to M$ if and only if the set of equations $\pi^{n_i}\xi_i = \psi(\alpha_i)$ has a solution $(\xi_i)_i\in M^r$. If $\psi$ is extendable to $N$, then there are $|\Hom_{\Ow_r}(N/L,M)|$ possible extensions. 
\end{lemma}
\begin{proof}
Write $N/L=Q = \bigoplus_i \langle \ol{q_i}\rangle$, where $\langle \ol{q_i}\rangle\cong \Ow_{n_i}$. Write $q_i$ for a preimage of $\ol{q_i}$ in $N$. Then $\pi^{n_i}q_i\in L$. We write $\pi^{n_i}q_i = \alpha_i\in L$. It then holds that $\psi$ is extendable to $N$ exactly when we can choose $\psi(q_i)$ such that $\pi^{n_i}\psi(q_i) = \psi(\alpha_i)$, as required. The second result follows from the fact that any two extensions of $\psi$ to $N$ differ by a homomorphism inflated from a homomorphism $Q\to M$. 
\end{proof}

\begin{proposition}\label{prop:boundmodules}
For every $a_1,\ldots, a_r,a,b\in \N$ write $n(a_1,\ldots, a_r,a,b)$ for the dimension of the space of constructible elements in $KM_{a_1,\ldots, a_r}^{a,b}$. Then there is a number $n(a,b)$ such that $n(a_1,\ldots, a_r,a,b)\leq n(a,b)$ for every $a_1,\ldots, a_r$. 
\end{proposition}
\begin{proof}
By the space of constructible elements we mean the image of $Con^{a,b}$ under $F_{KM}:\Hom_{\C_{univ}}(W^{\ot b}, W^{\ot a})\to \Hom_{\Vec_K}(KM^{\ot b},KM^{\ot a})$. 
Write $M=M_{a_1,\ldots, a_r}$, and write $\{e_m\}$ for the dual basis of $\{U_m\}$. For every $\Ow_r$-module $N$ and every two tuples of elements of $N$, $(s_1,\ldots,s_a)$ and $(l_1,\ldots,l_b)$, we define 
$$R_{(N,(s_i),(l_j))} = \sum_{\psi} U_{\psi(s_1)}\ot\cdots\ot U_{\psi(s_a)}\ot e_{\psi(l_1)}\ot\cdots\ot e_{\psi(l_b)}\in KM^{a,b},$$
where the sum is taken over all $\Ow_r$-module homomorphisms $\psi:N\to M$. 
We claim that all the constructible elements are of the form $R_{(N,(s_i),(t_j))}$ for some $(N,(s_i),(t_j))$. 
For the structure tensors this follows directly. For example: for the multiplication, $\sum_{m_1,m_2} U_{m_1+m_2}\ot e_{m_1}\ot e_{m_2}$, take $N=\Ow_r\mu_1\oplus \Ow_r\mu_2$, $s_1= \mu_1+\mu_2$ and $(l_1,l_2)=(\mu_1,\mu_2)$. Similar constructions hold for $\Delta,u,\epsilon,S$, and $T_x$. 

The set of tensors $R_{(N,(s_i),(l_j))}$ is closed under taking tensor products and under applying tensor permutations. If we can show that it is also closed under applying $ev$ then it will follow that it must contain all the constructible elements. For this, recall that $ev(U_m\ot e_{m'}) = \delta_{m,m'}$. Thus, after applying $ev$ to $R_{(N,(s_i),(t_j))}$ only the homomorphisms $\psi:N\to M$ for which $\psi(s_a)= \psi(t_b)$ will survive. These are in one-to-one correspondence with homomorphisms $N_1:=N/(s_a-t_b)\to M$. Write $\ol{n}$ for the image of $n\in N$ in $N_1$. We thus see that 
$$ev(R_{(N,(s_i),(l_j))}) = R_{(N_1,(\ol{s_1},\ldots, \ol{s_{a-1}}),(\ol{t_1},\ldots,\ol{t_{b-1}}))},$$
and all constructible elements are of the form $R_{(N,(s_i),(l_j))}$.

We need to show that for a given $(a,b)\in\N^2$ the elements $R_{(N,(s_i),(l_j))}$ span a vector space of a limited dimension in $KM^{a,b}$. For this, write $L = \langle s_1,\ldots, s_a,t_1,\ldots, s_b\rangle\subseteq N$ and write $N/L=Q$. The $\psi$-summand in the tensor $R_{(N,(s_i),(t_j))}$ is determined by the restriction of $\psi$ to $L$. 
By Lemma \ref{lem:extensionhom}, we see that there are elements $\alpha_1,\ldots, \alpha_w$ in $L$ and numbers $n_1,\ldots,n_w\in \{0,1,\ldots r-1\}$ such that a homomorphism $\psi:L\to M$ is extendable to $N$ if and only if the equation $\pi^{n_i}\xi = \psi(\alpha_i)$ has a solution in $M$ for every $i=1,\ldots w$. If $\psi$ is extendable to $N$, then it has exactly $\Hom_{\Ow_r}(Q,M)$ extensions. We can thus write $R_{(N,(s_i),(t_j)}$ as $|\Hom_{\Ow_r}(Q,M)|\sum_{\psi}U_{\psi(s_1)}\ot\cdots\ot U_{\psi(s_a)}\ot e_{\psi(l_1)}\ot\cdots\ot e_{\psi(l_b)}$, where the sum is taken over all homomorphisms $L\to M$ which satisfy the conditions given by the tuples $(\alpha_i)$ and $(n_i)$. 

Since there are only finitely many modules $L$ of rank $\leq a+b$, And since the possible tuples $(s_i),(t_j),(\alpha_i)$ and $(n_i)$ are all taken from finite sets, we get the desired result.
We also see that the scalar invariants that we get are all of the form $|\Hom_{\Ow_r}(Q,M)|$ for some $\Ow_r$-module $Q$. By writing $Q$ as the direct sum of cyclic modules, we see that every such invariant is a product of the elements $c_i$ from Equation \ref{eq:ci}.
\end{proof}

We can now prove that there exists an interpolation of the categories $\Rep(\Aut(M_{a_1,\ldots,a_r}))$. It holds that $$\chi_{a_1,\ldots a_r}\cong \chi_{a_1,0,\ldots, 0}\cdots \chi_{0,0,\ldots a_r}.$$ It will thus be enough to show that for every $i$ there is an interpolation of the family $\chi_{0,\ldots,0, a_i,0,\ldots, 0}$. For this, we use the fact that all the character values of $\chi_{0,\ldots,0,a_i,\ldots,0}$ are integer powers of $q^{a_i}$. We define a one-parameter family of character $\psi^{(i)}_t$ by the following formula: if $\chi_{0,\ldots,0,a_i,\ldots,0}(Di) = q^{na_i}$ then $\psi^{(i)}_t(Di) = t^n$. This gives us a multiplicative family, and all the conditions of Proposition \ref{prop:familymain} hold. Indeed, all the hom-spaces are of bounded finite dimension by Proposition \ref{prop:boundmodules}, and the special collection is given by the elements $\{q^i\}_{i\in \N}$. 
The categories $\Rep(\GL_t(\Ow_1)$ were also constructed by Deligne (unpublished). In \cite{Knop} Knop constructed the categories $\Rep(\GL_t(\Ow_r))$. We get here a bigger family of tensor categories, as for general values of $(t_1,\ldots, t_r)$ it holds that 
$$\C_{\psi^{(1)}_{t_1}\cdots \psi^{(r)}_{t_r}}\ncong \boxtimes_{i=1}^r\C_{\psi^{(i)}_{t_i}}$$ (in fact, there is no equivalence of categories even when all the parameters $t_i$ are positive integer powers of $q$.
\section*{Acknowledgments}
I would like to thank Pierre Deligne for discussion and comments on an earlier version of the paper. 
I would also like to thank L{\'o}r{\'a}nt Szegedy for his help with the tikzit package.

\end{document}